\documentclass[10pt,reqno]{amsart}
\usepackage{amsfonts,amssymb, latexsym}
\usepackage{amsmath,amscd}
\usepackage{cases}

\usepackage{amssymb}
\usepackage{mathrsfs}
\usepackage{graphicx}
\usepackage{color}
\usepackage[all]{xy}
\usepackage{bigstrut}
\usepackage{mathtools}

\begin{document}
\newtheorem{theorem}{Theorem}[section]
\newtheorem{proposition}[theorem]{Proposition}
\newtheorem{lemma}[theorem]{Lemma}
\newtheorem{claim}[theorem]{Claim}
\newtheorem{corollary}[theorem]{Corollary}
\newtheorem{defn}[theorem]{Definition}
\newtheorem{conjecture}[theorem]{Conjecture}
\newtheorem{remark}[theorem]{Remark}
\newtheorem{exe}{Exercise}

\theoremstyle{definition}%
\newtheorem{definition}[theorem]{Definition}

\def\pasdegrille{\let\grille = \pasgrille}
\def\ecriture#1#2{\setbox1=\hbox{#1}
\dimen1= \wd1 \dimen2=\ht1 \dimen3=\dp1 \grille #2 \box1 }
\def\aat#1#2#3{
\divide \dimen1 by 48 \dimen3=\dimen1 \multiply \dimen1 by #1
\advance \dimen1 by -\dimen3 \divide \dimen1 by 101 \multiply
\dimen1 by 100 \divide \dimen2 by \count11 \multiply \dimen2 by #2
\setbox0=\hbox{#3}\ht0=0pt\dp0=0pt
  \rlap{\kern\dimen1 \vbox to0pt{\kern-\dimen2\box0\vss}}\dimen1= \wd1
\dimen2=\ht1}
\def\pasgrille{
\count12= \dimen1 \divide \count12 by 50 \divide \dimen2 by \count12
\count11 =\dimen2 \ \divide \dimen1 by 48
\setlength{\unitlength}{\dimen1} \smash{\rlap{\ }} \dimen1= \wd1
\dimen2=\ht1 }
\def\grille{
\count12= \dimen1 \divide \count12 by 50 \divide \dimen2 by \count12
\count11 =\dimen2 \ \divide \dimen1 by 48
\setlength{\unitlength}{\dimen1}
\smash{\rlap{\graphpaper[1](0,0)(50, \count11)}} \dimen1= \wd1
\dimen2=\ht1 }

\pasdegrille

\numberwithin{equation}{section}

\newcommand{\R}{\mathbb{R}}
\newcommand{\C}{\mathbb{C}}
\newcommand{\TT}{\mathbb T}
\newcommand{\Z}{\mathbb Z}
\newcommand{\N}{\mathbb N}
\newcommand{\Q}{\mathbb Q}
\newcommand{\Sol}{\operatorname{Sol}}
\newcommand{\ND}{\operatorname{ND}}
\newcommand{\ord}{\operatorname{ord}}
\newcommand{\leg}[2]{\left( \frac{#1}{#2} \right)}  
\newcommand{\Sym}{\operatorname{Sym}}
\newcommand{\vE}{\mathcal E} 
\newcommand{\ave}[1]{\left\langle#1\right\rangle} 
\newcommand{\Var}{\operatorname{Var}}
 \newcommand{\tr}{\operatorname{tr}}
\newcommand{\supp}{\operatorname{Supp}}
\newcommand{\intinf}{\int_{-\infty}^\infty}
\newcommand{\beq}{\begin{equation}}
\newcommand{\eeq}{\end{equation}}
\newcommand{\ben}{\begin{eqnarray}}
\newcommand{\een}{\end{eqnarray}}
\newcommand{\beno}{\begin{eqnarray*}}
\newcommand{\eeno}{\end{eqnarray*}}
\newcommand{\dist}{\operatorname{dist}}
\newcommand{\area}{\operatorname{area}}
\newcommand{\vol}{\operatorname{vol}}
\newcommand{\diam}{\operatorname{diam}}
\newcommand{\Sp}{\mathbb S}
\newcommand{\indic}{1\!\!1}
\newcommand{\Bin}{\operatorname{Binom}}
\newcommand{\pois}{\operatorname{Pois}}
\newcommand{\myK}{V}
\newcommand{ \myvar}{\operatorname{Var}}
\newcommand{\ripleyK}{\^K}
\newcommand{\E}{\mathbb{E}}
\newcommand{\T}{\mathbb{T}}
\newcommand{\Prob}{\operatorname{Prob}}
\newcommand{\rta}{\longrightarrow}
\newcommand{\mi}{\mathrm{i}}
\newcommand{\q}{\quad}
\def\d{\delta}
\def\a{\alpha}
\def\eps{\varepsilon}
\def\ld{\lambda}
\def\p{\partial}
\newcommand{\bx}{\mathbf x}
\newcommand{\by}{\mathbf y}
\newcommand{\bz}{\mathbf z}
\newcommand{\bl}{\mathbf \lambda}
\newcommand{\SF}{\mathcal H} 

\newcommand{\la}{\langle}
\newcommand{\ra}{\rangle}

\newcommand{\tu}{\tilde{u}}

\newcommand{\Nsf}{\mathsf{N}}
\newcommand{\Ssf}{\mathsf{S}}

\newcommand{\AAA}{\mathcal{A}}
\newcommand{\BBB}{\mathcal{B}}
\newcommand{\EEE}{\mathcal{E}}
\newcommand{\FFF}{\mathcal{F}}
\newcommand{\HHH}{\mathcal{H}}
\newcommand{\JJJ}{\mathcal{J}}
\newcommand{\LLL}{\mathcal{L}}
\newcommand{\MMM}{\mathcal{M}}
\newcommand{\NNN}{\mathcal{N}}
\newcommand{\OOO}{\mathcal{O}}
\newcommand{\PPP}{\mathcal{P}}
\newcommand{\RRR}{\mathcal{R}}
\newcommand{\SSS}{\mathcal{S}}
\newcommand{\TTT}{\mathcal{T}}
\newcommand{\UUU}{\mathcal{U}}
\newcommand{\VVV}{\mathcal{V}}
\newcommand{\WWW}{\mathcal{W}}
\newcommand{\YYY}{\mathcal{Y}}
\newcommand{\ZZZ}{\mathcal{Z}}

\newcommand{\ubf}{\mathbf{u}}
\newcommand{\vbf}{\mathbf{v}}
\newcommand{\tubf}{\tilde{\mathbf{u}}}
\newcommand{\tTheta}{\widetilde{\Theta}}
\newcommand{\tY}{\widetilde{Y}}
\newcommand{\tU}{\widetilde{U}}
\newcommand{\ty}{\tilde{y}}

\newcommand{\loc}{\mathrm{loc}}

\newcommand{\ltri}{\arrowvert\arrowvert\arrowvert}

\newcommand{\uv}{\underline{v}}

\newcommand{\Id}{\mathrm{Id}\,}

\newcommand{\vvec}{\overrightarrow}

\newcommand{\vertiii}[1]{{\left\vert\kern-0.25ex\left\vert\kern-0.25ex\left\vert #1 
    \right\vert\kern-0.25ex\right\vert\kern-0.25ex\right\vert}}

\newcommand{\rouge}[1]{{\color{red} #1}}

\title[Center stable manifolds for wave outside a ball]{Center stable manifolds for the radial semi-linear wave equation outside a ball}

\author{Thomas Duyckaerts${}^1$}
\address{Thomas Duyckaerts, LAGA, Institut Galil\'ee, Universit\'e Paris 13
99, avenue Jean-Baptiste Cl\'ement,
93430 - Villetaneuse, France and
DMA, \'Ecole Normale Sup\'erieure, Universit\'e PSL, CNRS, 75005 Paris, France.
}
\email{duyckaer@math.univ-paris13.fr}

\author{Jianwei Yang${}^2$}
\address{Jianwei Yang, Department of Mathematics, Beijing Institute of Technology, Beijing 100081, P. R. China}
\email{jw-urbain.yang@bit.edu.cn}

\thanks{$^1$LAGA (UMR 7539), Universit\'e Sorbonne Paris Nord and DMA, \'Ecole Normale Sup\'erieure, Universit\'e PSL. Partially supported by the Labex MME-DII and the \emph{Institut Universitaire de France}}
\thanks{$^2$Department of Mathematics, Beijing Institute of Technology, Beijing 100081, P. R. China. Supported by NSFC grant No. 12371239 and Research fund program for young scholars of Beijing Institute of Technology.}

\begin{abstract}
We consider the nonlinear wave equation, with a large exponent, power-like non-linearity,
outside a ball of the Euclidean 3-dimensional space. In a previous article, we have proved that any global solution converges, up to a radiation term, to a stationary solution of the equation. In this work, we construct the center-stable manifold associated to each of the stationary solution, giving a complete description of the dynamics of global solutions. We also study the behaviour of solutions close to each of the center-stable manifold.
\end{abstract}

\maketitle

\section{Introduction}
\subsection{Background and main results}

Let $\mathbb{B}\subset \mathbb R^{3}$ be the unit ball centered at the origin and  set
$\Omega=\mathbb R^{3}\setminus \mathbb{B}$.
Consider the following Cauchy-Dirichlet problem on $\Omega$
\begin{equation}
\label{eq:NLW}
\begin{cases}
(\partial_{t}^{2}-\Delta)u(t,x)=u^{2m+1},\,(t,x)\in\mathbb R\times \Omega\\
(u,\partial_{t}u)|_{t=0}=(u_{0}, u_{1})\in \HHH,\;u_{\restriction\partial \Omega}=0
\end{cases}
\end{equation}
where $\HHH$ is the space of real-valued radial functions in $\dot{H}^1_0(\Omega) \times L^2(\Omega)$ and $m>2$ is an integer. The equation \eqref{eq:NLW} has a conserved energy
\begin{equation}
 \label{defE}
 E(u,\partial_tu)=\frac{1}{2}\int_{\Omega} |\nabla u|^2+\frac{1}{2}\int_{\Omega}(\partial_tu)^2-\frac{1}{2m+2}\int_{\Omega} |u|^{2m+2}.
\end{equation} 
The space $\HHH$ is precisely the space of finite energy states. In our previous work \cite{DuyckaertsYang21}, we have proved that the set of (radial) finite energy stationary solutions of \eqref{eq:NLW} is of the form 
$$\mathcal{Q}=\{0\}\cup \{Q_k\}_{k\geq 0}\cup \{-Q_k\}_{k\geq 0},$$
where $Q_k$ has exactly $k$ zeros on $(1,\infty)$, the sequence $\left\{E(Q_k,0)\right\}_{k\geq 0}$ is increasing and 
$$ \forall k\geq 0,\quad \lim_{r\to\infty} rQ_k(r)=\ell_k>0,$$
where the sequence $(\ell_k)_{k}$ is also increasing (see Subsection \ref{sub:stationary} below for the details). The equation \eqref{eq:NLW} is well-posed in $\HHH$: for any initial data $(u_0,u_1)$, there exists a solution of \eqref{eq:NLW} defined on a maximal time of existence $(T_-(u),T_+(u))$. Furthermore, if $T_{+}(u)$ is finite, then
\begin{equation}
 \label{BupCriterion}
\lim_{t\to T_+(u)} \|u(t)\|_{L^{\infty}}=+\infty.
\end{equation}
In particular, the energy norm must go to infinity. 
An analogous statement holds for negative time.

The equation \eqref{eq:NLW} was introduced to the first author by Piotr Bizo{\'{n}} in 2015 (see also the article \cite{BizonMaliborski20} of Bizo\'n and Maliborski for numerical and theoretical investigations on \eqref{eq:NLW}). We see it as a toy model for nonlinear focusing dispersive equations, where a complete picture of the global dynamics seems achievable, whereas it is out of reach, with current methods, for classical dispersive equations on the whole Euclidean space. Due to the restriction to radial solutions and the presence of the obstacle, the problem is essentially energy-subcritical and the energy is a relevant conservation law (in contrast with the general, nonradial case where the energy lives at lower regularity than the critical Sobolev space, and is thus ineffective). Moreover, the tools used in \cite{DuKeMe13} to prove soliton resolution for the radial energy-critical wave equation on $\R^3$ are still available in the context of equation \eqref{eq:NLW}. In \cite{DuyckaertsYang21}, the authors, using these tools, have described the dynamics of global solutions of \eqref{eq:NLW}:
\begin{theorem}
\label{T:DY19}
Let $(u_0,u_1)\in \HHH$ and $u$ be the corresponding solution of \eqref{eq:NLW}. Assume that $T_+(u)=+\infty$. Then there exists $Q\in \mathcal{Q}$ and a solution $u_L$ of the linear wave equation on $\Omega$:
\begin{equation}
\label{eq:LW}
(\partial_{t}^{2}-\Delta)u_L(t,x)=0,\,(t,x)\in\mathbb R\times \Omega,\quad u_{\restriction\partial \Omega}=0
\end{equation}
with initial data in $\HHH$, such that 
\begin{equation}
\label{asymptotic_u}
\lim_{t\to\infty} \Big\|\bigl(u(t)-u_L(t)-Q,\partial_tu(t)-\partial_t u_L(t)\bigr)\Big\|_{\HHH}=0.
\end{equation}
Furthermore, 
\begin{equation}
\label{eq:energy}
E(u_0,u_1)=E(Q,0)+\frac{1}{2}\|(u_L(0),\partial_tu_L(0))\|^2_{\HHH}. 
\end{equation} 
 \end{theorem}
 Note that if $u$ is a solution of equation \eqref{eq:NLW}, then $t\mapsto u(-t)$ is also a solution, so that the analog of Theorem \ref{T:DY19} for negative times is also valid.
 
 If \eqref{asymptotic_u} holds, we say that $u$ scatters to $Q$ (forward in time). We denote by $\mathcal{S}^+$ the open set of initial data in $\HHH$ such that the corresponding solution scatters to $0$, $\mathcal{M}_k^+$ the set of initial data such that it scatters to $Q_k$, and by $\mathcal{F}^+$ the set of initial data such that it blows up in finite positive time. We will use the same notations, with $-$ superscript, for the set of solutions having the same behaviour for negative times. By Theorem \ref{T:DY19} and the time reversibility of the equation,
 $$ \HHH=\FFF^+\cup \SSS^+ \cup \bigcup_{k\geq 0} \pm \MMM_k^{+},\quad \HHH=\FFF^-\cup \SSS^- \cup \bigcup_{k\geq 0} \pm \MMM_k^{-}$$
 where $-\MMM_k^{\pm}=\{-(u_0,u_1),\; (u_0,u_1)\in \MMM_k^{\pm}\}$.
%
 The purpose of this work is to go further than Theorem \ref{T:DY19} and to describe in more details the global dynamics of \eqref{eq:NLW}, especially in the neighborhood of the sets $\mathcal{M}_k^{\pm}$. We will prove in particular that the sets $\mathcal{M}_k^{\pm}$ are submanifolds of codimension $k+1$ of $\HHH$ that are unstable by blow-up for $k\geq 1$, and by blow-up and scattering to $0$ for $k=0$.
 
 The dynamics close to the stationary waves $Q_k$ and $-Q_k$ are related to the linearized equation
 \begin{equation}
  \label{eq:linearized_intro}
  (\partial_t^2+L_k)h=0,
 \end{equation} 
 and thus to the spectral properties of the linearized operator $L_k=-\Delta -(2m+1)Q^{2m}_k $. Our first result is the description of the spectrum of $L_k$, considered as an operator acting on radial functions with Dirichlet boundary condition:
\begin{theorem}
	\label{T:essspec}
	For each $k\in \mathbb{N}$, the essential spectrum of $L_k$
	is $[0,+\infty)$, $L_k$ has no zero energy state,
	and exactly $k+1$  eigenvalues 
    that are negative and simple. 
\end{theorem} 
 In the statement of the theorem the fact that $L_k$ has no \emph{zero energy state} means that zero is not an eigenvalue or a resonance for $L_k$, see Definition \ref{dfn:resonance}. The statement about the essential spectrum of $L_k$ is of course completely standard, and the proof of the absence of zero energy state relies on quite elementary estimates on solutions of an ordinary differential equation. The main novelty of Theorem \ref{T:essspec} concerns the counting of the eigenvalues of $L_k$, which is much more difficult, and ultimately relies on a nontrivial uniqueness result for second order elliptic differential equations on annulus, proved in \cite{NiNussbaum85}. Analogous results on $\R^3$ are not known, except for ground state solutions, see e.g. \cite{ChGuNaTs07} for references and a theoretical and numerical study of the spectrum of linearized operators at excited states for Schr\"odinger equations. 
 
 Note that for elliptic problems on $\R^3$, even the related question of uniqueness of excited state is very far to be understood: see the recent preprint \cite{CohenLiSchlag21P} for discussions on the subject and a computer-assisted proof of the uniqueness of the first radial excited states of $-\Delta Q+Q=Q^3$ on $\R^3$.  The uniqueness of radial excited states for the equation $-\Delta Q=Q^{2m+1}$ outside the ball was proved in \cite{DuyckaertsYang21}, however this case is much easier since all solutions are obtained by rescaling a unique object, see Subsection \ref{sub:stationary} below for more details.

 Using Theorem \ref{T:essspec}, we will prove that $\MMM_k^+$ is a codimension $k+1$ manifold of the energy space. We use the standard notation $\oplus$ for direct sums: if $V_1$, \ldots, $V_k$ are subspaces of a vector space $V$ we denote $V=V_1\oplus \ldots \oplus V_k$ when every element of $V$ can be written in a unique way as $\sum_{j=1}^k v_j$, $v_j\in V_j$. 
 \begin{theorem}[Global center stable manifold]
\label{T:cs}
Let $k\geq 0$, and $\pm$ be one of the signs $+$ or $-$. Then $\MMM_k^{\pm}$  is a smooth manifold of codimension $k+1$.
 \end{theorem}
 The definition of a smooth manifold is recalled in Appendix \ref{A:submanifold}. 
 By Theorem \ref{T:essspec}, we have the following spectral decomposition of $\HHH$ related to the linearized equation \eqref{eq:linearized_intro}:
 $$ \HHH= \HHH^s \oplus \HHH^u \oplus \HHH^c,$$
 where $\HHH^s$ is the stable space, $\HHH^u$ is the unstable space, $\dim \HHH^s=\dim \HHH^u=k+1$, and $\HHH^c$ is the projection of $\HHH$ on the continuous spectrum of $L_k$ (see \eqref{defHcus}, \eqref{defHcs} for the details)\footnote{Of course, these spaces depend on $k$, which we do not indicate to lighten notations}. From the proof of Theorem \ref{T:cs}, we obtain:
 \begin{proposition}[Tangent spaces]
 \label{P:intersection}
 The tangent space of $\MMM_k^+$ at $(Q_k,0)$ is $\HHH^{s}\oplus \HHH^c$. The tangent space of $\MMM_k^-$ at $(Q_k,0)$ is $\HHH^{u}\oplus \HHH^c$. In particular, the manifolds $\MMM_k^+$ and $\MMM_k^-$ intersect transversally in a neighborhood of $(Q_k,0)$, and $\MMM_k^+\cap \MMM_k^-$ is a submanifold of codimension $2k+2$ in a neighborhood of $(Q_k,0)$.
 \end{proposition}

By Theorem \ref{T:cs}, the scattering to a nonzero stationary solution is an unstable behaviour. We next give more information on this instability. From the construction of $\MMM_k^+$ (see e.g. Proposition \ref{P:P40}), we obtain:
\begin{proposition}
 \label{P:unstable}
 Let $k \geq 0$, $(u_0,u_1)\in \MMM_k^+$. Then there exists an open neighborhood $\mathcal{U}$ of $(u_0,u_1)$ in $\HHH$ such that 
 $$\mathcal{U}\cap -\MMM_k^+=\emptyset\quad\text{and}\quad\forall j\geq k+1,\;  \mathcal{U} \cap \MMM_j^+=\emptyset\text{ and }\mathcal{U} \cap -\MMM_j^+=\emptyset.$$
 \end{proposition}
We next prove that the behaviour of the solutions with initial data close to $\MMM_k^+$ is different in the ground-state case $k=0$ and the excited case $k\geq 1$. The behaviour close to the ground state manifold $\MMM_0^{+}$ can be described quite precisely. In this case, the solutions are unstable by blow-up and scattering:
\begin{theorem}
\label{T:M0}
 There exists open sets $\UUU^+$, $\UUU_F^+$, $\UUU_S^+$ of $\HHH$ such that
 $$ \UUU^+=\UUU_F^+\cup \UUU_S^+\cup \MMM_0^+,$$
 and any solution with initial data in $\UUU_F^+$ blows up in finite positive time, whereas any solution with initial data in $\UUU_S^+$ scatters to zero forward in time. Furthermore $\MMM_0^+$ is included in the boundary of $\UUU_F^+$ and in the boundary of $\UUU_S^+$.
\end{theorem}
Theorem \ref{T:M0} is coherent with the conjecture of Bizo\'n and Maliborski that the ground state sits at the threshold for blow-up (see \cite{BizonMaliborski20}). 
It is in the same spirit as known results on the whole Euclidean space: see the book \cite{NaSch11Bo}, and \cite{NaSch11} for Klein-Gordon, \cite{NaSch12} for the Nonlinear Schr\"odinger equation (in both cases, slightly above the ground state energy), and \cite{KrNaSc15} for critical waves. Let us mention however that our proof uses different arguments than in these works, based on the positivity of the linear flow. These arguments also give indication on the signs of the blow-up:
\begin{proposition}
\label{P:positive_blowup}
The blow-up in the set $\UUU_F^+$ of Theorem \ref{T:M0} is \emph{positive} in the following sense:
 for $(u_0,u_1)\in \UUU_F^+$, 
 $$\lim_{t\to T_+(u)} \sup_{x \in \Omega}u(t,x)=+\infty$$
 and $u(t,x)$ is uniformly bounded from below as $t\to T_+(u)$.
\end{proposition}
Using similar arguments as in the proof of Theorem \ref{T:M0}, we can obtain informations on the dynamics close to the manifold $\MMM_k^{+}$. Indeed we will show, when $k\geq 1$, that the corresponding solutions are unstable by positive blow-up (as the solutions with data in $\MMM_0^+$), but also by negative blow-up\footnote{We say that the blow-up of $u$ is negative when the blow-up of $-u$ is positive}, as illustrated by the two next results.
\begin{theorem}[Blow-up manifold]
\label{T:blow-up}
 Let $k\geq 1$ and $(v_0,v_1)\in \MMM_k^+$. Then there exists a neighborhood $\UUU$ of $(v_0,v_1)$ in $\HHH$ and a submanifold $\FFF_k^+$ of $\UUU$ of codimension $k$ such that $\MMM_k^+\cap \UUU$ is a codimension $1$ submanifold of $\FFF_k^+$, and for all $(u_0,u_1)\in \FFF_k^+$, if $(u_0,u_1)\notin \MMM_k^+$, then the corresponding solution blows up in finite positive time. More precisely, we can divide $\FFF_k^+$ as the union of $3$ connected sets:
 $$ \FFF_k^+= (\UUU \cap \MMM_k^+) \cup  \FFF_{k,p}^+ \cup \FFF_{k,n}^+,$$
 where for $(u_0,u_1)\in \FFF_{k,p}^+$, the blow-up of the corresponding solution $u$ is positive (in the sense of Proposition \ref{P:positive_blowup}), and if $(u_0,u_1)\in \FFF_{k,n}^+$, the blow-up of $u$ is negative.
\end{theorem}
With the same method, we can also construct an open set of initial data close to $\MMM_k^+$ leading to blow-up in finite time.
\begin{theorem}[Stable blow-up]
\label{T:blow-up'}
Let $k\geq 1$. There exists two disjoint connected open subsets $\UUU_{k,p}^+$ and $\UUU_{k,n}^+$ of $\HHH$ such that $\MMM_k^+$ is in the closure of $\UUU_{k,p}^+$ (respectively $\UUU_{k,n}^+$) in $\HHH$ and every solution with initial data in $\UUU_{k,p}^+$ (respectively $\UUU_{k,n}^+$) blows up in finite positive times with a positive blow-up (respectively a negative blow-up).
\end{theorem}
As an illustration of the difference of dynamics close to the ground state and the excited states, one can consider the solutions $u^{\alpha}$ with initial data 
$$ (Q_k,0)+\alpha (Y_0,e_0Y_0),$$
where $Y_0$ and $-e_0^2$ (depending on $k$) are the first eigenfunction and eigenvalue of the linearized equation around $Q_k$, i.e.
$$ \Delta Y_0+(2m+1)Q_k^{2m}Y_0=e_0^2Y_0$$
and $Y_0$ is positive, $C^{\infty}$ and exponentially decaying. The proofs of Theorems \ref{T:M0} and \ref{T:blow-up} show that if $k\geq 1$, $u^{\alpha}$ blows up in finite positive time for all $\alpha\neq 0$ (the sign of the blow-up is the sign of $\alpha$), whereas if $k=0$, $u^{\alpha}$ blows up in finite positive time for $\alpha>0$, but scatters to $0$ for positive time for $\alpha<0$, $|\alpha|$ small (see Lemma \ref{L:Y0}).

The proof of Theorem \ref{T:blow-up'} suggests that the instability by blow-up is, in some sense, generic, since the coefficients of the unstable direction $(Y_0,e_0Y_0)$ is the one with the fastest exponential growth for solutions escaping from $\MMM_k^+$, and thus should generically dominate all the other modes after some time, for solutions escaping from $\MMM_k^+$, leading to blow-up. The existence or non-existence of data in $\MMM_j^+$, $j<k$ or in $\SSS^+$ arbitrarily close in the neighborhood of $\MMM_k^+$ is however still open for $k\geq 1$. 

A related question is the behaviour of the solutions with initial data in $\MMM_k^+$ or close to $\MMM_k^+$ in the past. By Proposition \ref{P:intersection}, for any $k\geq 0$, $\MMM_k^+\cap \MMM_k^-$ contains a manifold of codimension $2k+2$.
Thus there exist solutions scattering to $Q_k$ in both time directions. Notice also that a ``nine sets theorem'', extending Theorem \ref{T:M0} and giving a complete description of the dynamics in both time directions of solutions with initial data close to $(Q_0,0)$ could be proved following the methods of \cite{NaSch11Bo}. We will not pursue this issue here. 
Let us mention however that with the techniques that are similar to the one employed in the proofs of Theorems \ref{T:blow-up} and \ref{T:blow-up'}, we can prove:
\begin{theorem}
 \label{T:negative_time}
 For any $k\geq 0$, there exists $(u_0,u_1)\in \MMM_k^+$, arbitrarily close to $(Q_k,0)$ in $\HHH$ such that the corresponding solution $u$ blows up in finite negative times with positive blow-up. If $k\geq 1$, the same conclusion holds with negative blow-up. 
 Moreover, for any $k\geq 0$, there exists $(u_0,u_1)\in \HHH$ arbitrarily close to $(Q_k,0)$ in $\HHH$ such that the corresponding solution $u$ blows up in both time directions. If $k\geq 1$, then for each time direction, one can choose the sign of the blow-up.
\end{theorem}

\subsection{Related works}

For the first articles on construction of center-stable manifold for nonlinear partial differential equations, and infinite dimensional systems, see e.g. \cite{Ball73}, \cite{ChowLu88}, \cite{BatesJones89}, \cite{IoosVDB92},\cite{Gallay93}.

The construction of center stable manifold for nonlinear dispersive equations, was initiated by Nakanishi and Schlag  (see \cite{NaSch11Bo}, \cite{NaSch11}, \cite{NaSch12}). As mentioned above, these constructions are restricted to the ground state soliton, and local: they do not show that the set of all solutions scattering to a stationary wave or a soliton is  a manifold. 

The case of nonlinear critical wave equation was considered by Krieger, Nakanishi and Schlag  \cite{KrNaSc15}. The manifold constructed there is global, in the sense that it contains all the global solutions that are asymptotically the sum of a modulated ground state and a radiation term.

In \cite{JiLiScXu17} (radial case), \cite{JiLiScXu20} (nonradial case) Jia, Liu, Schlag and Xu have constructed \emph{global} center stable manifolds associated to a non-degenerate excited states for the defocusing critical wave equation with a linear potential, in space dimension $3$. These results are in the spirit of Theorem \ref{T:cs}. Of course, in their general context, the nondegeneracy of the excited state has to be assumed, and the number of negative eigenvalues of the linearized operator is not known.

Very few is known on the dynamics of dispersive equations close to excited states, outside the associated center stable manifold. See for example \cite{TsaiYau02a}, \cite{TsaiYau02b} and \cite{SofferWeinstein04} for nonlinear Schr\"odinger equations with potential.

The work \cite{KrNaSc15} implies the existence of solutions scattering to the ground state in the future, and blowing-up in finite time in the past, similar to the solutions of Theorem \ref{T:negative_time} with $k=0$. A particular example of such a solution (for the same energy critical wave equation) had been constructed before in \cite{DuMe08}. To our knowledge, the solutions given by Theorem \ref{T:negative_time} with $k\geq 1$ are the first examples of solutions scattering to an excited state in the future, and blowing-up in finite time in the past.

The equation \eqref{eq:NLW} without symmetry assumption on the initial data is an example of focusing energy-supercritical equation, for which a general description of the dynamics is completely out of reach by current methods. One can see the study of the radial solutions of \eqref{eq:NLW} as a first step toward the understanding of this dynamics. We refer to \cite{DAncona19P}, where this point of view is adopted for the \emph{defocusing} supercritical wave equation.

The precise description of solutions blowing up in finite time close to the blow-up time is not treated here. Away from the obstacle, the equation is very close to the one dimensional semilinear wave equation on $\R$, and the theory developed by Merle and Zaag should apply (see \cite{MeZa07}, \cite{MerleZaag11} and references therein). The description of the blow-up close to $r=1$ is an open question. An interesting related question is the stability of the blow-up: we conjecture that the set $\FFF^+$ is open, however the method developed here are not sufficient to prove this property.

\subsection{Outline of the article}
In Section \ref{sec:linearizedoptr} we study the linearized operator $-\Delta-(2m+1)Q_j^{2m}$ at a stationary solution $Q_j$, proving Theorem \ref{T:essspec}. 

In Section \ref{sec:center-manifold}, we construct the global center stable manifold. Our proof is close to the one in \cite{JiLiScXu17}, \cite{JiLiScXu20}. In particular, to deduce the existence of a global manifold from the construction of the local manifold, we use a no-return property proved by a channel of energy argument inspired by \cite{DuKeMe12c}, as in \cite{JiLiScXu20}. The proof requires Strichartz and local energy decay estimates for the wave equation with a potential in the exterior domain $\Omega$, that we prove in a separate article \cite{DuyckaertsYang23P}.

In Section \ref{sec:instability}, we prove the instability/blow-up results Theorem \ref{T:M0}, Proposition \ref{P:positive_blowup}, and Theorems \ref{T:blow-up}, \ref{T:blow-up'} and \ref{T:negative_time}. The proofs use comparison arguments relying on the positivity of the linear wave equation on $\Omega$ with the Dirichlet boundary condition. This type of argument to prove blow-up for nonlinear wave equations is classical and goes back at least to \cite{John79}. However, to our knowledge, it is the first time that they are used in combination with Soliton resolution type results such as Theorem \ref{T:DY19}, and to prove also scattering, as in Theorem \ref{T:M0}.

\subsection*{Acknowledgment}
This work was initiated during Jianwei Yang's visit in the LAGA (Universit\'e Sorbonne Paris Nord) in 2018-2019, as an international chair of the Labex MME-DII. 

\section{Linearized operators around stationary solutions}\label{sec:linearizedoptr}
This section concerns the linearized operator 
\begin{equation}
\label{eq:Lj}
L_{j}=-\Delta-(2m+1)Q_{j}^{2m}\;,
\end{equation}
around a nonzero static solution $Q_j$. In Subsection \ref{sub:stationary} we recall from our previous work \cite{DuyckaertsYang21} some facts about stationary solutions of \eqref{eq:NLW}. In Subsection \ref{sub:nullspace} we start the proof of Theorem \ref{T:essspec}, by studying the essential spectrum of $L_j$, and by proving that zero is not an eigenvalue or a resonance for $L_j$. In Subsection \ref{sub:atleast}, we prove that $L_j$ has at least $j+1$ negative eigenvalues. The core of this section is Subsection \ref{sub:atmost}, where we prove that $L_j$ has at most $j+1$ negative eigenvalues. We deduce this from the fact that the function $\Lambda Q_j=x\cdot\nabla Q_j+\frac{1}{m}Q_j$ has exactly $j+1$ zeros, a fact that we prove using a uniqueness result for semilinear elliptic equations on annulus due to Ni and Nussbaum \cite{NiNussbaum85}. Finally, in Subsection \ref{sub:Strichartz}, we will recall dispersive estimates for $L_j$ that we proved in \cite{DuyckaertsYang23P}. 

Let us mention that some of the results of this section are contained in \cite{BizonMaliborski20} with proofs combining numerical and theoretical arguments.

\subsection{Existence of stationary solutions}
\label{sub:stationary}
We state several properties on a class of singular stationary solutions  involved in 
\cite{ DuKeMe12c,DuyckaertsRoy17,DuyckaertsYang18,DuyckaertsYang21}.
\begin{proposition}
	\label{P:DKM-stationary}
	Let $m>2$, $m\in\mathbb{N}$. 
	Then there exists a radial, $C^{2}$ solution $Z(x)=Z(|x|)$ of 
	\begin{equation}
	\label{eq:stationary}
	\Delta Z+Z^{2m+1}=0\quad\text{on}\quad
	\mathbb R^{3}\setminus\{0\},
	\end{equation}
	such that 
	\begin{equation}
	\label{eq:decay-infty}
	\forall\;r\geq 1,\quad
	\bigl | r \, Z(r)-1\bigr |\leq \frac{C}{r^{2}}
	\end{equation}
	\begin{equation}
	\label{eq:derivative-infty}
	\lim_{r\rightarrow\infty}
	r^{2}\frac{d Z}{dr}=-1\,.
	\end{equation}
	Moreover, the zeros of $Z(r)$ are given by a sequence $\{r_{j}\}_{j=0}^{\infty}$ such that 
	\[
	r_{0}>r_{1}>\cdots>r_{j}>\cdots\longrightarrow 0,
	\quad j\rightarrow\infty.
	\]
\end{proposition}

Let $Q_{j}(r)=r_{j}^{1/m}Z(r_{j}r)$.
Then $Q_{j}(|x|)$ is the radial solution of the following elliptic equation outside the unit ball 
$\Omega=\mathbb R^{3}\setminus \mathbb B$
with the Dirichlet boundary condition
\begin{equation}
\label{eq:Q}
-\Delta Q=Q^{2m+1},\quad Q_{\restriction{\partial\Omega}}=0,\quad
x\in\Omega,
\end{equation}
where $\Delta=\Delta_D$ is the Dirichlet-Laplacian, and $Q$ belongs to $ \dot H^{1}_{0}(\Omega)$. 
By Lemma 2.19 in \cite{DuyckaertsYang21}, $\{\pm Q_j\}\cup\{0\}$ are the only radial stationary solutions of \eqref{eq:Q} in the energy space.\\

\subsection{Essential spectrum and the null space of $L_j$}\label{hhhh}
\label{sub:nullspace}
We consider the operator $L_j$, with domain
$$\text{Dom}(L_j)=(H^2\cap H^1_0)_{\rm rad}(\Omega)=\{f\in H^1_0(\Omega): f\text{ radial and }-\Delta f \in L^2(\Omega)\}.$$
Since $Q_j^{2m}\in L^2(\Omega)\cap L^\infty(\Omega)$, we have from Theorem 8.2.2 in \cite{Davies95BO} that $L_j$ is a self-adjoint operator. 

We denote by
$$\mathfrak{H}_{j}(h)=\int_{\{|x|>1\}}|\nabla h|^2dx-(2m+1)\int_{\{|x|>1\}}Q_j^{2m}(x)h^2(x)\,dx,$$
for $h\in \dot H^{1}_{0}(\Omega)$, radial,
the quadratic form associated to $L_{j}$.
\begin{lemma}
The essential spectrum of $L_j$ is $[0,\infty)$. There is no embedded eigenvalue on $(0,\infty)$.  
\end{lemma}
\begin{proof}[Sketch of proof]
Indeed it is well-known that the spectrum of $-\Delta$ with domain $(H^2\cap H^1_0)_{rad}(\Omega)$ is $[0,+\infty)$. To prove that the essential spectrum of $L_j$ is $[0,\infty)$, it suffices to show that $L_j$ and
$-\Delta$ have the same essential spectrum. Let $V=(2m+1)Q_j^{2m}\indic_{\{|x|\geqslant 1\}}(x)$.
By Weyl's theorem (see Theorem 14.6 in \cite{HislopSigal96BO}), it suffices to show that $V $ is relatively $-\Delta$-compact. This follows easily from the radial Sobolev inequality
\begin{equation}
\label{rad_Sob}
\forall R>1,\quad
|f(R)|\lesssim \frac{1}{\sqrt{R}}\|f\|_{\dot{H}^1_{\rm rad}}.
\end{equation}
The fact that there is no embedded eigenvalue in the essential spectrum is standard (see e.g.\cite{IoJe03}).
\end{proof}

\begin{definition}
	\label{dfn:resonance}
	A nonzero function $h\in H^2_{\mathrm{loc}}(\overline{\Omega})$, such that $\langle x\rangle^{-1}h(x)\in L^2(\Omega)$ and
	\begin{equation}
	\label{eq:resonance}
	L_jh=0,\quad h_{\restriction \partial \Omega}=0,
	\end{equation} 
	is called a \emph{zero energy state} of $L_j$. 
\end{definition}
We next prove:
\begin{proposition}
	\label{P:ker}
	The operator $L_j$ has no radial zero energy state.
\end{proposition}
\begin{proof}
 Let $h\in H^1_{\loc}(\Omega)$ be a radial zero energy state, i.e. $h\neq 0$, 
 $$ L_jh=0,\; h(1)=0,\quad \int_1^{\infty} |h(r)|^2dr<\infty.$$
 We will show that $h(r)$ has a finite nonzero  limit as $r\to\infty$, which is an obvious contradiction.
 
 Note that the assumptions imply that $h$ is a $C^1$ function of $r$ on $[1,\infty)$. For $\lambda>0$,  $Q_{j}^\lambda(r)=\lambda^{\frac{1}{m}}Q_j(\lambda r)$
	satisfies $-Q''-\frac{2}{r}Q'=Q^{2m+1}$. Differentiating, we obtain that
	\begin{equation}
	\label{ZZZZ}
	\Lambda Q_j(r):=\frac{d}{d\lambda}\Bigg|_{\lambda=1}Q^\lambda_j(r)=\frac{1}{m}Q_j(r)+rQ_j'(r)
	\end{equation}
	satisfies the equation $L_j(\Lambda Q_j)=0$. As a consequence,
	$$\frac{d}{dr}\left( r^2h'\Lambda Q_j-r^2 (\Lambda Q_j)'h \right)=0.$$
	Thus $r^2h'\Lambda Q_j-r^2 (\Lambda Q_j)'h$ is equal to a constant $c$, and, since $h(1)=0$, $c=\Lambda Q_j(1)h'(1)=Q_j'(1)h'(1)\neq 0$.
Since $\displaystyle\lim_{r\to\infty}r\Lambda Q_j(r)=\left( \frac{1}{m}-1 \right)r_j^{\frac 1m-1}$, we see that $\Lambda Q_j(r)<0$ for large $r$. Letting $g= h/(\Lambda Q_j)$, we obtain that the equation
$$r^2h'\Lambda Q_j-r^2 (\Lambda Q_j)'h=c$$ 
implies $r^2(\Lambda Q_j)^2g'=c$. As a consequence,
$$ \lim_{r\to \infty} g'(r)=\left(\frac{m}{m-1}\right)^2 \frac{c}{r_j^{\frac{2}{m}-2}}:=c'\neq 0.$$
By l'H\^opital's rule, we obtain
$$\lim_{r\to\infty} \frac{g(r)}{r}=c'.$$
We have proved as announced that $h(r)=g(r)\Lambda Q_j(r)$ has a finite, nonzero limit as $r\to\infty$. 
\end{proof}

\subsection{Existence of at least $j+1$ negative eigenvalues for $L_j$}\label{sub:atleast}
From this subsection, we start to study the solvability of 
\begin{equation}
\label{hgutw}
L_j \phi = \mu \phi
\end{equation} 
with $\mu\in \R$ and $\phi\in \text{Dom}(L_j)$.
Taking account of the radial setting, it is convenient to 
introduce the following notation. Let $\phi\in \text{Dom}(L_j)$ and set $\psi(r)=r\phi(r)$.
Then \eqref{hgutw} becomes
\begin{equation}
\label{vmgd}
-\psi''-(2m+1)Q_j^{2m}\psi(r)=\mu\psi(r),
\end{equation}
with 
\begin{equation}
\label{integrability}
\int_1^{+\infty}|\psi(r)|^2dr<\infty,\quad\psi(1)=0. 
\end{equation} 
By the elliptic regularity theory, if $\psi$ is the solution
of \eqref{vmgd} that satisfies \eqref{integrability}, then $\frac{1}{r}\psi(r)\in \text{Dom}(L_j)$.
Thus, \eqref{hgutw} with the condition $\phi \in \text{Dom}(L_j)$ is equivalent to \eqref{vmgd}, \eqref{integrability}.\\

In view of Subsection \ref{hhhh},
if \eqref{vmgd} has a nontrivial solution, 
then $\mu<0$ and the collection of such $\mu$'s
consists of the discrete negative eigenvalues of the operator $L_j$. Denote by $\mathcal{N}(L_j)$ the number of negative eigenvalues of 
$L_j$. Since we restrict to radial functions, these eigenvalues are all simple by standard uniqueness theory for ordinary differential equations. We show in this subsection that for every $j\in \mathbb{N}$
\begin{equation}
\label{hbaa}
\mathcal{N}(L_j)\geqslant j+1.
\end{equation}

To see that \eqref{hbaa} is true, we employ the Rayleigh-Ritz variational formulae for eigenvalues.
Let $C_{j}=\sup\{(2m+1)\,Q_j^{2m}(x)\,:\,x\in \Omega\} $ and 
$$\mathcal{L}_j:=L_j+C_{j}.$$
Then $\mathcal{L}_j$ is a non-negative self-adjoint operator on $L^2_{\mathrm {rad}}(\Omega)$  with domain ${\rm Dom}(L_j)$.
Let $S$ be any finite-dimensional subspace of ${\rm Dom}(L_j)$, and define 
$$\lambda_j(S)=\sup\{\left\langle \mathcal{L}_j f, f\right\rangle: f\in S,\, \|f\|_{L^2}=1\},$$
where $\langle\cdot,\cdot\rangle$ is the real inner product in $L^2(\Omega)$.
The eigenvalues $\mu_{j,k}$ of $\mathcal{L}_j$ are given by 
\begin{equation}
\label{eq:evformula}
\mu_{j,k}=\inf\{\lambda_j(S):S\subset {\rm Dom}(L_j),\,{\rm dim}\, S=k\}.
\end{equation}
Since $\mathcal{L}_j$ has essential spectrum $[C_j,\infty)$ as a consequence of Theorem \ref{T:essspec}, we have
according to Theorem 4.5.2 in \cite{Davies95BO} that one of the two possibilities occurs:
\begin{enumerate}
	\item[(i)]$\mu_{j,k}<C_j$ for all 
	$k$ and $\mu_{j,k}\rightarrow C_j$ as $k\rightarrow+\infty$. Moreover, the intersection of the spectrum of $\mathcal{L}_j$ with $[0,C_j) $ consists of the eigenvalues
	$\mu_{j,k}$.
	\item[(ii)]
	There exists $K_j<\infty$ such that $\mu_{j,K_j}<C_j$ but $\mu_{j,k}=C_j$ for all
	$k>K_j$. Moreover, the intersection of the spectrum of $\mathcal{L}_j$ with $[0,C_j)$ consists of the eigenvalues
	$\mu_{j,1},\ldots, \mu_{j,K _j}$.
\end{enumerate}
If we are in the first case, then $\mathcal{N}(L_j)=+\infty$ and we are done. If we are in the second case, it suffices to construct a subspace
$S_j\subset {\rm Dom}(L_j)$ with ${\rm dim}\, S_j=j+1$ such that
$\lambda_j(S_j)<C_{j}$.
\medskip

Denote by $\gamma_{j,i}=r_{j-i}/r_{j}$ for $i=0,1,\ldots j$. Then $1=\gamma_{j,0}<\gamma_{j,1}<\ldots<\gamma_{j,j}$ and $Q_{j}(\gamma_{j,i})=0$ for $i= 0,\ldots, j$.
We let $I_{j,i}=[\gamma_{j,i},\gamma_{j,i+1}]$
for $i=0,1,\ldots,j-1$, $I_{j,j}=[\gamma_{j,j},\infty)$ and
\begin{equation}
\label{mnnb}
R_{j,i}(r)=c_{j,i}Q_{j}(r)\indic_{I_{j,i}}(r),\quad c_{j,i}=\frac{1}{\|\indic_{I_{j,i}}Q_{j}\|_{L^2}},\; i=0,1,\ldots, j.
\end{equation}
Then, by a straightforward integration by parts,
\[
0>\mathfrak{H}_{j}(R_{j,i})=-2m\int_{\gamma_{i}}^{\gamma_{i+1}}
|R_{j,i}(r)|^{2m+2}r^{2}dr,\quad
(R_{j,i},R_{j,i'})_{L^2}=\delta^{i}_{i'}\,,
\]
where $\delta^{i}_{i'}$ is the Kronecker symbol and $i,i'\in\{0,1,\ldots, j\}$. 
Thus 
$$ \left\langle \mathcal{L}_j R_{j,i}, R_{j,i}\right\rangle = \mathfrak{H}_{j}(R_{j,i}) +C_j<C_j.$$ 
It is clear that one may take $S_j={\rm span}\{R_{j,0},R_{j,1},\ldots, R_{j,j}\}$ to complete the proof.

\subsection{Existence of at most $j+1$ negative eigenvalues for $L_j$}
\label{sub:atmost}
Given $j\in\mathbb{N}$, we have obtained a sequence of 
negative eigenvalues of $L_j$
\[
\mu_{j,1}<\mu_{j,2} <\ldots <\mu_{j,\ell}< \cdots<0
\]
which could be infinitely many. 
We suppress the subscript $j$ for 
brevity in this section. 

To show that $L_j$ has at most $j+1$ negative eigenvalues is more intricate
and we shall divide the proof
of this subsection into three parts
\subsubsection{Reduction to counting zeros of $\Lambda Q_j$}
Let $Y_\ell$ be an eigenfunction associated to 
$\mu_\ell$ as in \eqref{hgutw} with 
$\ell\in\{1,2,\ldots\}$. Let $e_{\ell}=\sqrt{-\mu_{\ell}}$, so that 
\begin{equation}
	\label{EV}
	Y_{\ell}''+\frac{2}{r}Y_{\ell}'+(2m+1)Q_j^{2m}Y_{\ell}=e_\ell^2Y_{\ell},\quad  r>1
	\end{equation}

Then $\psi_\ell(r)=rY_\ell(r)$ satisfies 
\begin{equation}
\label{vmgd'}
\psi_{\ell}''-e_{\ell}^2\psi_{\ell}=-(2m+1)Q_j^{2m}\psi_{\ell}.
\end{equation}
The asymptotic behaviour of $Y_{\ell}$ is given by Agmon's estimates:
\begin{lemma}
	\label{kdlre}
	The eigenfunction $Y_{\ell}$ is exponentially decreasing. More precisely, replacing $Y_{\ell}$ by $-Y_{\ell}$ if necessary, and writing $\mu_{\ell}=-e_\ell^2$ with $e_\ell>0$, there exists $c_\ell>0$ such that
	$$
	\psi_{\ell}(r)=rY_{\ell}(r)=\left(c_{\ell}+\OOO(r^{1-2m})\right)e^{-e_{\ell}r},\quad r\to\infty$$
	$$rY_{\ell}'(r)=\left(-c_{\ell}e_{\ell}+\OOO(r^{-1})\right)e^{-e_{\ell}r},\quad r\to\infty.$$
\end{lemma}
\begin{proof}
This is well-known (see \cite{Agmon85}, \cite{Meshkov89} for the nonradial case). We sketch a proof, in the spirit of  the proof of Proposition 3.9 in \cite{DuKeMe16a}. 

	Define
	\[G(R)=\int_{R}^{+\infty}
	\left(|Y_{\ell}'(r)|^2+e_\ell^2|Y_{\ell}(r)|^2\right)r^2dr, \;R\gg 1.\]
By direct computations and integrations by parts:
\begin{multline}
\label{eqG}
	2e_\ell G(R)+G'(R)=-R^2\left(Y_{\ell}'(R)+e_\ell Y_{\ell}(R)\right)^2
	+2e_\ell(2m+1)\int_{R}^{+\infty}
	Q_j^{2m}(r)Y_{\ell}^2(r)r^2dr\\
	\leq  \frac{C}{R^{2m}}\,G(R).
	\end{multline}
	As a consequence, we have
	\[
	\frac{d}{d R}\left(\log \left( e^{2e_\ell R}G(R)\right)\right)
	\leq C R^{-2m},
	\]
	which yields $G(R)\leq C e^{-2e_\ell R}$, and also, by \eqref{eqG}, $G'(R)\leq Ce^{-2e_\ell R}$. Thus $|Y_{\ell}|+|Y'_{\ell}|\lesssim \frac{e^{-e_{\ell}r}}{r}$. Using this decay and solving the equation \eqref{vmgd'} satisfied by $\psi_{\ell}$, we obtain that there exists a constant $c\in \R$ such that
	\begin{multline}
	\label{formula}
	\psi_{\ell}(r)=ce^{-e_{\ell}r} +e^{e_{\ell} r}\frac{2m+1}{2e_{\ell}}\int_{r}^{\infty} Q^{2m}(s)\psi_{\ell}(s)e^{-e_{\ell}s}ds\\
	-e^{-e_{\ell} r}\frac{2m+1}{2e_{\ell}}\int_{r}^{\infty} Q^{2m}(s)\psi_{\ell}(s)e^{e_{\ell}s}ds.
	\end{multline}
	It is easy to see that if $c=0$, then $\psi_{\ell}(r)=0$ for large $r$, and thus by uniqueness for all $r$. Thus $c\neq 0$ and one deduces from \eqref{formula} the asymptotic of $Y_{\ell}$ and $Y'_{\ell}$. 
\end{proof}

\begin{lemma}
	\label{hfbchd}
	Let $N(\psi_\ell)$
	be the number of zeros of the function $\psi_\ell(r)$ on $(1,+\infty)$. Then 
	$N(\psi_\ell)= \ell-1$, for all $\ell\in\{1,2,\ldots\}$.
\end{lemma}
\begin{proof}
	We prove $N(\psi_{\ell})\geq \ell-1$ using an induction argument going back to \cite{Bargmann52} (see e.g. the proof of \cite[Theorem XIII.8]{ReSi.T4}). It is obvious that 
	$N(\psi_1)\ge0$.  Assume $\ell\ge2$ and 
	$N(\psi_{\ell-1})\ge \ell-2$. By Sturm comparison theorem and $\mu_{\ell-1}<\mu_\ell$, $\psi_\ell(r)$ has a zero between every two consecutive zeros of $\psi_{\ell-1}(r)$.
	Let $0<\rho<+\infty$ be the largest zero
	of $\psi_{\ell-1}(r)$. By Lemma \ref{kdlre}, $\psi_{\ell-1}(r), \psi_\ell(r)$ tend to zero exponentially fast as $r\rightarrow+\infty$. We have by Sturm-Liouville's argument that $\psi_\ell(r)$ has  a zero on $(\rho,+\infty)$.
	Indeed, we may otherwise assume 
	$\psi_{\ell-1}(r)\psi_\ell(r)>0$ on $(\rho,+\infty)$ and consider 
	\begin{align}
	\int_{\rho}^{+\infty}
	\left( \psi_{\ell-1}'(r)\psi_\ell(r)-\psi_{\ell-1}(r)\psi_\ell'(r)\right)'dr
	=-\psi_{\ell-1}'(\rho)\psi_\ell(\rho)\le 0.
	\end{align}
	On the other hand,
	\begin{align*}
	&\int_{\rho}^{+\infty}
	\left( \psi_{\ell-1}'(r)\psi_\ell(r)-\psi_{\ell-1}(r)\psi_\ell'(r)\right)'dr\\
	=&\int_{\rho}^{+\infty}\left(\psi_{\ell-1}''(r)\psi_\ell(r)-\psi_{\ell-1}(r)\psi_\ell''(r)\right)dr\\
	=&(\mu_\ell-\mu_{\ell-1})\int_{\rho}^{+\infty}\psi_{\ell-1}(r)\psi_\ell(r)dr>0.
	\end{align*}
	The case $\psi_{\ell-1}(r)\psi_\ell(r)<0$ on $(\rho,+\infty)$ can be handled similarly.
	Therefore, we conclude $N(\psi_\ell)\ge N(\psi_{\ell-1})+1\ge \ell-1$.
	\medskip
	
	It remains to show $N(\psi_{\ell})\le \ell-1$. If not, then $N(\psi_{\ell})\ge \ell$. Thus $\psi_{\ell}$ changes sign $\ell$ times on $(1,\infty)$. Using the same argument as in Subsection \ref{sub:atleast} we obtain that $L_j$ has at least $\ell$ eigenvalues of $(-\infty,\mu_{\ell})$, a contradiction with the fact that $\mu_{\ell}$ is the $\ell$-th eigenvalue.
\end{proof}
Let $N(\Lambda Q_j)$ be the number
of zeros in $(1,+\infty)$ of the function $\Lambda Q_j$
given by \eqref{ZZZZ}.
Recalling that $L_j(\Lambda Q_j)=0$, by Sturm-Liouville comparison theorem,
we have 
\begin{equation}
 \label{ineg_zeros}
N(\psi_{\ell})\le N(\Lambda Q_j)-1.
 \end{equation} 
Indeed, there exists at least one zero of $\Lambda Q_j$ between every two consecutive zeros of $\psi_{\ell}$ on $[1,+\infty)$. 
Moreover, since $\psi_{\ell}(r)$ tends to zero exponentially at infinity, we have again by Sturm-Liouville argument that
$\Lambda Q_j$ has at least one zero
between the largest zero of $\psi_{\ell}$ and the infinity.

In the next two subsections we will prove:
\begin{equation}
\label{ncbs} N(\Lambda Q_j)=j+1,\,\quad \forall \;j\in\mathbb{N}.
\end{equation}
We claim that \eqref{ncbs} implies that $L_j$ has exactly $j+1$ eigenvalues, concluding the proof of Theorem \ref{T:essspec}.
Indeed, assuming \eqref{ncbs} for the moment, we have by \eqref{ineg_zeros} that $\ell-1=N(\psi_{\ell})\le j$, for every eigenfunction $\psi_{\ell}$, which shows that there are at most $j+1$ negative eigenvalues of $L_j$ and hence completes the proof.
\subsubsection{An auxiliary result via variational argument}
We fix again $j\in \mathbb{N}$.
\begin{proposition}
	\label{P:A} 
	Let, for $i\in\{0,1,\ldots,j\}$, $\Omega_{j,i}=\{x\in\R^3:|x|\in I_{j,i}\}$
	where $I_{j,i}=[\gamma_{j,i},\gamma_{j,i+1}]$ if $0\leqslant i\leqslant j-1$ and $I_{j,j}=[\gamma_{j,j},\infty)$ with $\gamma_{j,i}=r_{j-i}/r_{j}$ for all $i\in\{0,1,\ldots,j\}$. Then for every $i\in\{0,1,\ldots,j\}$, $L_j$,
	regarded as an operator acting on $L^2_{\rm rad}(\Omega_{j,i})$ subject to the Dirichlet boundary conditions, has exactly one negative eigenvalue.
\end{proposition}
\begin{proof} We divide the proof in three steps
	\subsubsection*{Step 1. A variational inequality}
	In this step, we show that for all $i\in\{0,1,\ldots,j\}$,  and all radial function $f\in\dot{H}^1_0(\Omega_{j,i})$, 
	\begin{equation}
	\label{eq:minimizer}
	\|f\|_{L^{2(m+1)}(\Omega_{j,i})}\|\nabla Q_{j}\|_{L^{2}(\Omega_{j,i})}
	\leq \|Q_{j}\|_{L^{2(m+1)}(\Omega_{j,i})}\|\nabla f\|_{L^{2}(\Omega_{j,i})}.
	\end{equation}
	It suffices to show that if we set
	\[
	J_{j,i}(f)=\|f\|_{L^{2(m+1)}(\Omega_{j,i})}^{2(m+1)}\,/\,\|\nabla f\|_{L^{2}(\Omega_{j,i})}^{2(m+1)}, 
	\]
	and
	\[
	a_{j,i}=\sup\{ J_{j,i}(f):f\in \dot H^{1}_{0}(\Omega_{j,i})\setminus \{0\}, f \;\text{radial}\},
	\]
	then $a_{j,i}=J_{j,i}(Q_{j})$. By the radial Sobolev inequality, the above quantities are well-defined and $0< a_{j,i}<\infty$.
	\medskip

	The argument is reminiscent of \cite{Weinstein82}.
	Take a maximizing sequence $f_{\nu}\in \dot H^{1}_{0}(\Omega_{j,i})$, radial, such that $J_{j,i}(f_{\nu})\rightarrow a_{j,i}$ as $\nu\rightarrow+\infty$. We may assume (replacing $f_{\nu}$ by $|f_{\nu}|$ if necessary), that $f_{\nu}$ is nonnegative. 
	Setting $\varphi_{\nu}=f_{\nu}/\|f_{\nu}\|_{\dot H^{1}_{0}(\Omega_{j,i})}$, we have $J_{j,i}(\varphi_{\nu})=J_{j,i}(f_{\nu})$ and
	$\|\nabla \varphi_{\nu}\|_{L^{2}(\Omega_{j,i})}=1$. Hence there
	exists a subsequence $\varphi_{\nu_{k}}$ converges 
	weakly in $\dot H^{1}$ to $\varphi_{*}$ as $k\rightarrow +\infty$
	with $\|\varphi_{*}\|_{\dot H^{1}_{0}(\Omega_{j,i})}\leq 1$.
	By using the radial Sobolev inequality for the case $i=j$ and the Rellich-Kondrachov theorem, 
	one can show $\varphi_{\nu_{k}}$
	converges to $\varphi_{*}$ strongly in $L^{2(m+1)}(\Omega_{j,i})$.
	As a consequence, we have $\|\varphi_{*}\|_{\dot H^{1}_{0}(\Omega)}>0$, since otherwise we would have $J_{j,i}(\varphi_{\nu_{k}})\underset{k\to\infty}{\longrightarrow} 0$ by the strong convergence. It follows from the above discussion that 
	\[
	a_{j,i}\geqslant J_{j,i}(\varphi_{*})\geqslant \|\varphi_{*}\|_{L^{2(m+1)}(\Omega_{j,i})}^{2(m+1)}=\lim_{k\rightarrow+\infty}\|\varphi_{\nu_{k}}\|_{L^{2(m+1)}(\Omega_{j,i})}^{2(m+1)}=a_{j,i}\,.
	\]
	Thus $J_{j,i}(\varphi_{*})=\|\varphi_{*}\|_{L^{2(m+1)}(\Omega_{j,i})}^{2(m+1)}=a_{j,i}$ and $\|\nabla \varphi_{*}\|_{L^{2}(\Omega_{j,i})}=1$, which along with the weak convergence implies
	$\varphi_{\nu_{k}}\underset{k\to\infty}{\longrightarrow} \varphi_{*}$ in $\dot H^{1}_{0}(\Omega_{j,i})$ strong topology.
	\medskip
	
	From the above facts, it follows that $\varphi_{*}$ is the maximizer of the function $J_{j,i}$ and has to satisfy the 
	Euler-Lagrange equation:
	\[
	\frac{d}{d\eps}\Bigg|_{\eps=0}J_{j,i}(\varphi_{*}+\eps \eta)=0\,,\quad 
	\forall\;\eta\in C_{0}^{\infty}(\Omega_{j,i}).
	\]
	Taking $\|\nabla \varphi_{*}\|_{L^{2}(\Omega_{j,i})}=1$  into account, we have
	\[
	-\Delta  \varphi_{*}=\frac{1}{\|\varphi_{*}\|_{2(m+1)}^{2(m+1)}}\varphi_{*}^{2m+1}.
	\]
	Let $Q(x)=\varphi_{*}(x)\|\varphi_{*}\|_{L^{2(m+1)}}^{-(m+1)/m}$.
	Then we have $-\Delta Q=Q^{2m+1}$ on $\Omega_{j,i}$ and $Q|_{\partial \Omega_{j,i}}=0$, $Q(x)\geqslant 0$ for $x\in \Omega_{j,i}$.
	If $i<j$, it is a classical result of Ni and Nussbaum that there is only one nonzero solution for this problem, see \cite{NiNussbaum85}. Uniqueness is still valid if $i=j$ as a consequence of \cite[Lemma 2.19]{DuyckaertsYang21}. Thus $Q=\pm Q_{j}\indic_{\Omega_{j,i}}$ (where the sign $\pm$ is the sign of $Q_j$ on $\Omega_{j,i}$).

	\subsubsection*{Step 2. Positivity}
	In this step, we prove that for all $i\in\{0,1,\ldots,j\}$, if we define
	\[
	\widetilde{\mathfrak{H}}_{j,i}(f)=\int_{\Omega_{j,i}}
	\Bigl(|\nabla f|^2-(2m+1)Q_j^{2m}|f|^2\Bigr)dx,
	\]
	then we have
	\begin{equation}
	\label{eq:coercivity}
	\widetilde{\mathfrak{H}}_{j,i}(f)\geq 0,
	\end{equation}
	for all radial $f\in \dot H^{1}_{0}(\Omega_{j,i})$ satisfying $(f,Q_{j})_{\dot H^{1}(\Omega_{j,i})}:=\int_{\Omega_{j,i}}\nabla f\cdot \nabla Q_j dx=0$.\\
	
	First, we show that the orthogonality condition $(f,Q_{j})_{\dot H^{1}(\Omega_{j,i})}=0$ implies
	$\widetilde{\mathfrak{H}}_{j,i}(f)\geqslant 0$. 
	
	The argument is standard (see \cite{Weinstein85}). We sketch it for the sake of completeness.
	Let $f_{\eps}(x)=\eps f(x)$ with $\eps>0$ small. 
	From Step 1, 
	we have
	\begin{equation}
	\label{eq:sharp}
	\|Q_{j}\|_{L^{2(m+1)}(\Omega_{j,i})}^{2}
	\|\nabla Q_{j}+\nabla f_{\eps}\|_{L^{2}(\Omega_{j,i})}^{2}
	\geq \|\nabla Q_{j}\|_{L^{2}(\Omega_{j,i})}^{2}
	\|Q_{j}+f_{\eps}\|_{L^{2(m+1)}(\Omega_{j,i})}^{2}.
	\end{equation}
	From $(f_{\eps}, Q_{j})_{\dot H^{1}(\Omega_{j,i})}=0$ and \eqref{eq:Q}, 
	we have $\left(Q_{j}^{2m+1},f_{\eps}\right)_{L^{2}(\Omega_{j,i})}=0$.
	By taking $\eps$ sufficiently small and using Taylor's formula, we obtain
	\[
	\|Q_{j}+f_{\eps}\|_{L^{2(m+1)}(\Omega_{j,i})}^{2(m+1)}
	=\|Q_{j}\|_{L^{2(m+1)}(\Omega_{j,i})}^{2(m+1)}
	+(2m+1)(m+1)\left(Q_{j}^{2m},\,f_{\eps}^{2}\right)_{L^{2}(\Omega_{j,i})}
	+\mathcal O(\varepsilon^3).
	\]
	Plugging this to \eqref{eq:sharp}, we have by Taylor's formula again
	\[
	\frac{\|\nabla Q_{j}\|^{2}_{2}+{\eps}^{2}\|\nabla f\|_{2}^{2}}{\|\nabla Q_{j}\|_{2}^{2}}
	\geq 1+{\eps}^{2}(2m+1)\frac{\left(Q_{j}^{2m},f^{2}\right)_{L^{2}}}{\|Q_{j}\|_{2(m+1)}^{2(m+1)}}+\mathcal O(\eps ^{3}\| f\|^{3}_{\dot H^{1}_{0}(\Omega_{j,i})}),
	\]
	where the norms in the above inequality are all defined by integration in $\Omega_{j,i}$.
	Letting $\eps\rightarrow 0+$, we conclude 
	$\widetilde{\mathfrak{H}}_{j,i}(f)\geq 0$.
	
	\subsubsection*{Step 3. End of the proof}
We next prove that $L_j$ (as an operator on $L^2(\Omega_{j,i})$), has at most one eigenvalue, arguing by contradiction. If not, there exists $\lambda_1<\lambda_2<0$, $Y_1\in \dot{H}^1_0(\Omega_{j,i})$, $Y_2\in \dot{H}^1_0(\Omega_{j,i})$, with 
$$ L_jY_1=\lambda_1Y_1,\quad L_j Y_2=\lambda_2Y_2.$$ Then, if $(\alpha_1,\alpha_2)\in \R^2\setminus \{(0,0)\}$, 
$$ \big\langle L_j(\alpha_1 Y_1+\alpha_2Y_2),  \alpha_1 Y_1+\alpha_2Y_2\big\rangle=-\alpha_1^2\lambda_1-\alpha_2^2\lambda_2<0.$$
This contradicts Step 2 since we can choose $(\alpha_1,\alpha_2)\neq (0,0)$ such that 
$$\int_{\Omega_{j,i}} \nabla(\alpha_1 Y_1+\alpha_2Y_2)\cdot \nabla Q_{j}=0.$$ 
\end{proof}
\begin{remark}
	Notice that by definition, $\Omega_{0,0}=\Omega$. Thus Theorem \ref{T:essspec} is proved in the case $j=0$.
\end{remark}

\subsubsection{$\Lambda Q_j$ has $j+1$ zeros on $(1,+\infty)$} We are ready to prove \eqref{ncbs} and complete the proof of Theorem \ref{T:essspec}.\medskip

We start with a simple topological remark
on the number of zeros of $\Lambda Q_j$. 
By \eqref{eq:decay-infty}, \eqref{eq:derivative-infty} and \eqref{ZZZZ}, we have
\begin{equation}
\label{nxbsdc}
\lim_{r\rightarrow\infty}r\Lambda Q_j(r)=-\frac{m-1}{m}r_j^{\frac{1}{m}-1}<0,
\end{equation}
and $\Lambda Q_j(1)$ is positive if $j$ is an even number including $0$ whilst $\Lambda Q_j(1)$ is negative if $j$ is odd.
Notice that $\Lambda Q_j$ can only have simple zeros since it is a nonzero solution of a second order ordinary differential equation. Thus $\Lambda Q_j$ has an odd number of zeros for even $j$ whilst $\Lambda Q_j$ has an even number of zeros for odd $j$.\\

Next, we show that $\Lambda Q_j$ has $j+1$ zeros on $(1,+\infty)$ by an inductive argument with respect to $j$.
Let us start with $j=0$ and show that $\Lambda Q_0$ has precisely one zero on $(1,+\infty)$. Suppose otherwise, by the above topological judgement, $\Lambda Q_0$ has at least three zeros, say $z_{0,1}, z_{0,2}, z_{0,3}\in (1,\infty)$ with $z_{0,1}<z_{0,2}<z_{0,3}$.
Let $\eps>0$ be a small number to be fixed later and consider a perturbated problem
\begin{equation}
\label{eq:J}
\begin{cases}
\frac{d^2}{dr^2}R_j^\eps+\frac{2}{r}\frac{d}{dr}R_j^\eps+(1-\eps)(2m+1)Q_j^{2m}R_j^\eps=0\\
R_j^\eps(1)=\Lambda Q_j(1), \quad \frac{d}{dr}R_j^\eps(1)=\Lambda Q_j'(1)
\end{cases}
\end{equation}
with $j=0$. By standard regular perturbation theory of ordinary differential equations, such 
a solution $R_0^\eps$ exists and is sufficiently close
to $\Lambda Q_0$ uniformly in any compact subset of $[1,\infty)$ by taking $\eps$ small enough when necessary.
Since the zeros $z_{0,1},z_{0,2},z_{0,3}$ of $\Lambda Q_0$ are simple,
$R_0^\eps$ has zeros, say $z_{0,1}^\eps, z_{0,2}^\eps,z_{0,3}^\eps$ in $(1,\infty)$, each being sufficiently close to $z_{0,1},z_{0,2},z_{0,3}$ respectively by taking $\eps$ smaller if necessary.
Let $\mathfrak{R}_{0,1}(r)=R_0^\eps(r)\indic_{(z_{0,1}^\eps,z_{0,2}^\eps)}(r)$
and $\mathfrak{R}_{0,2}(r)=R_0^\eps(r)\indic_{(z_{0,2}^\eps,z_{0,3}^\eps)}(r)$.
Then it is readily to verify that for $k=1,2$
\[\mathfrak{H}_{0}(\mathfrak{R}_{0,k})
=-\eps(2m+1)\int_{z_{0,k}^\eps}^{z_{0,k+1}^\eps}
Q_0^{2m}(r)\big(R_0^\eps(r)\big)^2r^2dr<0.
\]
These two inequalities along with the same argument employed  in Subsection 3.2 based on the Rayleigh-Ritz variational formulae lead to
the existence of at least two distinct negative eigenvalues for $L_0$ on $L^2_{\rm rad}(\Omega)$, which is a contradiction to Proposition \ref{P:A}. Hence $\Lambda Q_0$
has only one zero in $(1,\infty)$.\\

Let $j\geqslant 1$ and assume that $\Lambda Q_k$ has exactly $k+1$ zeros on $(1,\infty)$ for all $k\in\{0,1,\ldots,j-1\}$. Assume furthermore that if we denote the zeros of $\Lambda Q_k$ as a list $z_{k,0}<z_{k,1}<\cdots<z_{k,k}$, then $z_{k,i}\in I_{k,i}$ for all $i\in\{0,1,\ldots,k\}$, where we adopt the notation in Proposition \ref{P:A}. Notice that this hypothesis is true when $j=1$.

We are intended to show that $\Lambda Q_j$ has precisely $j+1$ zeros on $(1,\infty)$. We assume to fix ideas that $j$ is an even number. The other case can be handled by the same argument.
Noting that
\[
\Lambda Q_j(r)=\left(\frac{r_j}{r_{j-1}}\right)^{\frac{1}{m}}\Lambda Q_{j-1}\left(\frac{r_j}{r_{j-1}}r\right),
\]
we see that $\{\frac{r_{j-1}}{r_j}z_{j-1,i}\}_{i\in\{0,1,\ldots,\,j-1\}}$
are zeros of $\Lambda Q_j(r)$ and that $$\frac{r_{j-1}}{r_j}z_{j-1,i}\in I_{j,i+1},\quad \forall \;i\in\{0,1,\ldots,j-1\}.$$
The one-one correspondence $z\mapsto \frac{r_{j-1}}{r_j}z$
from $I_{j-1,i}$ to  $I_{j,i+1}$
and the inductional hypothesis on the distribution of zeros of $\Lambda Q_{j-1}$
imply that if the number of zeros of $\Lambda Q_j$, which is odd hence and at least $j+1$, is greater or equal to $j+3$, then $I_{j,0}$ would have to contain at least three zeros of $\Lambda Q_j$. Now, considering the perturbed problem \eqref{eq:J}, and taking $\eps>0$ small enough, we may argue as we did for the $j=0$ case to yield two distinct negative eigenvalues of $L_j$ as an operator on $L^2_{\rm rad}(\Omega_{j,0})\cap \dot{H}^1_0(\Omega_{j,0})$, which contradicts Proposition \ref{P:A}. The proof is complete.

\subsection{Dispersive estimates}
\label{sub:Strichartz}

We next recall from \cite{DuyckaertsYang23P} Strichartz estimates and local energy decay for the linearized wave equation. 
We fix $k$, and denote by $P_c$ the projection on the continuous spectrum of $L_k$:
\begin{equation}
\label{def_Pc}
P_c=\mathrm{Id}-\sum_{0\le j\le k}\langle\;\cdot\;,{Y}_{j}\rangle {Y}_{j}.
\end{equation} 
Since the $Y_j$ are smooth, exponentially decaying and satisfy Dirichlet boundary conditions, this is a bounded operator on $L^q(\Omega)$, $2\leq q\leq \infty$, and also on $\dot{H}^1_0(\Omega)$. Of course $P_c$ depends on $k$, but we did not indicate this fact in the notations to lighten them. We let $\PPP_c=(P_c,P_c)$, which is a bounded operator on $\HHH$. 

Let $I$ be an interval. We introduce the spaces $\Ssf(I)$ and $\Nsf(I)$ defined by the following norms:\begin{equation}
\label{defS}
\|u\|_{\mathsf{S}(I)}=\|u\|_{L^{2m+1}_t\bigl(I,L^{2(2m+1)}_x(\Omega)\bigr)} +\|u\|_{L^{4}_t\bigl(I,L^{12}_x(\Omega)\bigr)} +\left\| \langle x\rangle^{-3}u \right\|_{L^2(I,L^2(\Omega))},
\end{equation} 
and
\begin{equation}
 \label{defN}
\|f\|_{\Nsf(I)}=\inf_{f=f_1+f_2} \|f_1\|_{L^{1}(I,L^2(\Omega))}+\left\|\langle x\rangle^{2}f_2\right\|_{L^2(I,L^2(\Omega))}.
 \end{equation} 
where $u,f:I\times \Omega\to \R$. We claim:
\begin{proposition}[Mixed Strichartz/weighted energy estimates]
 \label{P:mixed}
Let $u$ be a solution to
	\begin{equation}
	\label{eq:linearized_in}
	\partial_{t}^2 u+L_Q u=P_c f,\quad (u,\partial_tu)_{\restriction t=0}=\mathcal{P}_c (u_0,u_1),
	\end{equation}
	where $(u_0,u_1)\in \HHH$ is radial, $f\in \Nsf(I)$ with $f(t)$ radial for all $t\in I$.
 Then $u\in \Ssf(I)$, $\vec{u}\in C^0(I,\HHH)$ and 
 \begin{equation}
  \label{Mixed}
  \|u\|_{\mathsf{S}(I)}+\|\vec{u}\|_{L^{\infty}(I,\HHH)}\lesssim \|(u_0,u_1)\|_{\HHH}+\|f\|_{\Nsf(I)}.
 \end{equation} 
\end{proposition}
\begin{proof}
Proposition \ref{P:mixed} follow from \cite{DuyckaertsYang23P}. Note that $\Omega$, $L_Q$ satisfy the assumptions of this paper: the non-trapping assumption is obvious for a ball, and the fact that $0$ is not an eigenfunction or a resonance (at least in a radial setting, see \cite[Remark 1.5]{DuyckaertsYang23P}) was proved in Subsection \ref{sub:nullspace} above. Since the ball is strictly convex, the global Strichartz estimates for the linear wave equation without a potential go back to \cite{SmithSogge95,SmithSogge00}.

We first prove that for any pair $(p,q)$ such that $\frac{1}{p}+\frac{3}{q}=\frac{1}{2}$, $p>2$,
\begin{multline}
 \label{dispersive1}
 \|\vec{u}\|_{L^{\infty}(I,\HHH)}+\|u\|_{L^{p}_t\bigl(I,L^{q}_x(\Omega)\bigr)} +\left\| \langle x\rangle^{-3}u \right\|_{L^2(I,L^2(\Omega))}\\
 \lesssim \|(u_0,u_1)\|_{\HHH}+\|f\|_{\Nsf(I)}.
\end{multline}
By linearity, we can consider separately the following two cases:
\begin{itemize}
 \item $f\in L^1(I,L^2(\Omega))$, then 
 \eqref{dispersive1} follows from Theorem 1.6 and Appendix C in \cite{DuyckaertsYang23P}.
 \item $\langle x\rangle^{2}f\in L^2(I,L^2(\Omega))$ and $(u_0,u_1)=(0,0)$. Then \eqref{dispersive1} follows from Theorem 1.9 and (3.25) in \cite{DuyckaertsYang23P}.
 \end{itemize}
 Note that we can take $(p,q)=(4,12)$ in \eqref{dispersive1}.
 Next, we observe that by the radial Sobolev embedding $\dot{H}^1(\Omega)\subset L^{\infty}(\Omega)\cap L^6(\Omega)$, one has that \eqref{dispersive1} holds also for $p=\infty$, $6\leq q\leq \infty$.
By H\"older's inequality, we deduce that  \eqref{dispersive1} holds for any $(p,q)$ with $\frac{1}{p}+\frac{3}{q}\leq \frac 12$, $p>2$, thus in particular for $(p,q)=(2m+1,2(2m+1))$, concluding the proof. 
 \end{proof}

\begin{remark}
 \label{R:mixed_free}
 The estimate \eqref{Mixed} also holds for solutions of the linear equation $\partial_t^2u-\Delta u=f$ with Dirichlet boundary conditions and initial data in $\HHH$. The proof is the same.
\end{remark}

\section{Construction of the center-stable manifold}
\label{sec:center-manifold}

In this section, we  prove Theorem \ref{T:cs} for a given stationary solution $\overrightarrow{Q_k}$,
by showing that $\mathcal{M}_k^+$ is a submanifold of $\HHH$ of codimension $k+1$. We use the results of Section 2, and the arguments (which we slightly revisit) of the articles \cite{JiLiScXu17},  \cite{JiLiScXu20} which concern the defocusing energy-critical wave equation with a potential.
  
The proof is divided into two parts. We first construct a local center stable manifold close to $\vec{U}(T)$, where $U$ is a solution of \eqref{eq:NLW} with initial data in $\MMM_k^+$, and $T$ is large. We then prove, using a no-return result, that this manifold structure can be extended in a neighborhood of any element of $\MMM_k^+$.

To lighten notations, we shall often drop from now on the subscript $k\geq 0$, which is fixed in all this section. We will denote $Q=Q_k$, the linearized operator as $L_Q$ and its Hamiltonian form (defined below in \eqref{defLLL}) as $\mathcal{L}_Q$. 

\medskip

To construct the local center-stable manifold,
we shall use the Lyapunov-Perron method, as initiated in \cite{NaSch11} for the Klein-Gordon equation (see also the book \cite{NaSch11Bo}).
The local center-stable manifold
is constructed by solving an integral equation in certain functional spaces. The corresponding  solution obtained in this way naturally defines a \emph{reduction map}, 
whose graph realises the desired manifold.

As mentioned in the introduction, the spectral resolution of the linearized operator $\mathcal{L}_Q$ at $\overrightarrow{Q}$   leads to a natural decomposition $\mathcal{H}=\HHH^u\oplus \HHH^c \oplus \HHH^s$ (see \eqref{defHcus}, \eqref{defHcs} below).
Let $h=u-Q$. The integral equation on $h$ to be employed is obtained by writing the equation \eqref{eq:NLW} into its Duhamel form using the linearized operator $\LLL_Q$, and impose a \emph{stability condition} on the $\HHH^u$ component of $h$, so that $u$ remains close to $\overrightarrow{Q}$ in the positive direction of time. We will call \emph{reduced integral equation} the system of equations obtained by combining the Duhamel form of equation \eqref{eq:NLW} (written in term of $\LLL_Q$) and the stability condition.

It turns out that this integral equation can be solved by the fixed point argument,  using the full range  Strichartz estimates for the linearized operator, proved in \cite{DuyckaertsYang23P} and recalled in Subsection \ref{sub:Strichartz} above.  This defines a 
nonlinear transformation from $\HHH^{cs}:=\HHH^c\oplus \HHH^s$ to $\HHH$. 

Let $\vvec{v_0}\in \HHH^{cs}$ and  $\vec{h}(t)=\vec{h}(t,x; \vvec{v_0})$ be the solution of the reduced integral equation with data $\vvec{v_0}$. The reduction map from $\HHH^{cs}$ to $\HHH^u$ can be defined by simply projecting $\vec{h}(0,x; \vvec{v_0})$ to the subspace $\HHH^u$. 
The graph of the reduction map obtained this way is preserved by the flow and hence yields the local invariant manifold.

We note that the manifold constructed that way is ``local'' in a weak sense: the elements of the manifold are close to $Q$ in the space-time norm $\|\cdot\|_{\Ssf}$ but not necessarily in the strong topology of $\HHH$. This fact is crucial in deducing from this local manifold a global one that contains all the elements of $\MMM_k^+$. 

\medskip

To deduce the fact that $\MMM_k$ is a manifold from the local center-stable manifold construction, we need a no-return lemma, stating that any solution that is close to $\MMM_k$ (in the strong topology of $\HHH$) at some time $t_0$ and further away from $\MMM_k$ at a time $t_1>t_0$ cannot get close to $\MMM_k$ again for $t>t_1$. To prove this lemma, we use the strategy of \cite{JiLiScXu20}, based on a ``channel of energy'' argument on the unstable modes, from \cite{DuKeMe16a}.
\subsection{The canonical decomposition of $\mathcal{H}$ and the projected linearized equation} 


To prove Theorem \ref{T:cs}, we decompose every solution of the nonlinear wave equation  \eqref{eq:NLW} around $(Q,0)$ as $u(t)=Q+h(t)$.
Our goal is to solve the linearized equation 
\begin{equation}
\label{eq:linearizedeq}
\partial_{t}^2h+L_Q h=R_Q(h),
\end{equation}
where 
\begin{equation*}
L_Q:=-\Delta-(2m+1)Q^{2m},\quad
R_Q(h):=(Q+h)^{2m+1}-Q^{2m+1}-(2m+1)Q^{2m}h
\end{equation*}
subject to a stability condition to be specified later. 


Note that we have the  pointwise bound:
\begin{equation}
 \label{pppp}
	|R_Q(h)-R_Q(\ell)|\le C_m\bigl(h^{2m}+\ell^{2m}+(|h|+|\ell|)\cdot Q^{2m-1}\bigr)|h-\ell|.	
 \end{equation}
We denote by $S_Q(t)$ the flow  associated to the linearized equation: $S_Q(t)\vvec{v_0}=v(t)$, and $\overrightarrow{S_Q}(t)\vvec{v_0}=(v,\partial_tv)(t)$, where $v(t)$ is the solution to 
\begin{equation}
 \label{LWQ}
(\partial_t^2+L_Q)v=0,\quad \vec{v}_{\restriction t=0}=\vvec{v_0},\quad  v_{\restriction{\partial\Omega}}=0.
 \end{equation}


\subsubsection{The symplectic structure and the invariant subspace decomposition}

Let us define $\omega(\cdot,\cdot)$
to be the symplectic inner product
$$\omega\bigl(\vec{f},\vec{g}\bigr)=\int_{\Omega} f_0g_1-
\int_{\Omega} f_1g_0,\quad \forall\;
\vec{f}=\begin{pmatrix} 
f_0  \\
f_1
\end{pmatrix},\;
\vec{g}=\begin{pmatrix} 
g_0 \\
g_1 
\end{pmatrix}.
$$
That is, $\omega(\cdot,\cdot)=\langle\;\cdot\;| \,\mathcal{J}\,\cdot\,\rangle$, where $\mathcal{J}=\begin{pmatrix} 
0 & 1 \\
-1 & 0
\end{pmatrix}$, 
$\langle\, \vec{f}\,|\,\vec{g}\,\rangle=\langle f_0,g_0\rangle+\langle f_1,g_1\rangle$
and 
$\langle\cdot,\cdot\rangle$ is the real inner product in $L^2(\Omega)$.  
\medskip

We denote a vector $\vec{u}$ in $\HHH$ by 
$\vec{u}=\begin{pmatrix} 
u_0  \\
u_1
\end{pmatrix}$.

We put $L_Q$   into the Hamiltonian form 
\begin{equation}
 \label{defLLL}
\mathcal{L}_Q=\begin{pmatrix} 
0 & 1 \\
-L_Q & 0
\end{pmatrix}=\mathcal{J}\begin{pmatrix} 
L_Q  & 0 \\
0& 1
\end{pmatrix}.
\end{equation}
This defines an unbounded operator $\LLL_Q$  on $\dot{H}^1_0(\Omega)\times L^2(\Omega)$, with domain $D(L_Q)\times H^1_0(\Omega)$. The spectral properties of $L_Q$ may be reformulated as follows:
\begin{claim}
	\label{claim:essspec-hmt-general}
	Let $L_Q$ and $\LLL_Q$ be as above.
	Then $\LLL_Q$ has exactly $2(k+1)$ eigenvalues 
	(which are all simple) $\EEE^+\cup \EEE^-$ with 
	$\mathcal{E}^{\pm}=\{\pm e_{0},\ldots,\pm e_j,\ldots\pm e_{k}\}$, and $-e_0^2<\ldots<-e_k^2$ are the $k+1$ negative eigenvalues of $L_Q$ (where we choose $e_j>0$). The spectrum of $\LLL_Q$ is given by $\mathcal{E}^-\cup \mathrm{i}\,\R \cup\mathcal{E}^+$. The corresponding eigenfunctions are given by $\mathcal{Y}^{\pm}_{j}=(2e_{j})^{-\frac{1}{2}}\bigl(
	1 ,	\pm e_{j} \bigr){Y}_{j}$. The essential spectrum of $\mathcal{L}_Q$ is  $\mathrm{i}\,\R$. 
	
\end{claim}
\begin{proof}
 Let $(v_0,v_1)\in H_0^{1}(\Omega)\times L^2(\Omega)$, $\lambda\in \C$. The equation
 $$ (\LLL_Q-\lambda)(u_0,u_1)=(v_0,v_1),\quad (u_0,u_1)\in D(\LLL_Q)$$
 is equivalent to 
 \begin{equation}
 \label{eigenfunct}
 u_1-\lambda u_0=v_0,\quad -L_Q u_0-\lambda u_1=v_1.  
 \end{equation} 
From this, one deduces easily that $\lambda$ is in the resolvent set of $\LLL_Q$ if and only if $-\lambda^2$ is in the resolvent set of $L_Q$, and that we have 
$$( \LLL_Q-\lambda)^{-1}=-(\lambda^2+L_Q)^{-1} (\LLL_Q+\lambda)=-(\LLL_Q+\lambda)(\lambda^2+L_Q)^{-1}. $$
Also, $\lambda$ is an eigenvalue of $\LLL_Q$ if and only if $-\lambda^2$ is an eigenvalue of $L_Q$. The claim then follows easily from Theorem \ref{T:essspec}. Note that the formula $\mathcal{Y}^{\pm}_{j}=(2e_{j})^{-\frac{1}{2}}\bigl(
	1 ,	\pm e_{j} \bigr){Y}_{j}$ follows directly from \eqref{eigenfunct}.
 \end{proof}

By direct computation, one easily verifies that  $\{\mathcal{Y}^\pm_{j}\}_{ j\in\{0,\ldots k\}}$ is a standard symplectic  basis satisfying 
$$
\omega\bigl(\mathcal{Y}^+_{j},\mathcal{Y}^-_{\ell}\bigr)=\begin{cases}
-1,\;&j=\ell\\
0,\;& j\neq \ell
\end{cases},
\;\quad\omega(\mathcal{Y}^+_{j},\mathcal{Y}^+_{\ell})=\omega(\mathcal{Y}^-_{j},\mathcal{Y}^-_{\ell})=0,\;\forall\,j,\ell.
$$
Thus, 
\begin{align}
\notag
\bigoplus_{j=0}^k{\rm span}\Bigl\{\mathcal{Y}^+_{j}, \mathcal{Y}^-_{j}\Bigr\}&={\rm span}
	\Bigl\{\mathcal{Y}^+_{0},\ldots, \mathcal{Y}^+_{k}\Bigr\}\oplus {\rm span}\Bigl\{ \mathcal{Y}^-_{0},\ldots, \mathcal{Y}^-_{k}\Bigr\}\\
	\label{P+-}&={\rm span} \Bigl\{(Y_{0},0),\ldots, (Y_{k},0),(0,Y_0),\ldots,(0,Y_k)\Bigr\}.
\end{align}
is a symplectic vector space of dimension $2k+2$.
The dual basis of  $\{\mathcal{Y}^\pm_{j}\}_{ j\in\{0,\ldots k\}}$   with respect to $\omega$ is $\{\mp \mathcal{Y}^\mp_{j}\}_{j\in \{0,\ldots,k\}}$.
\medskip

Let $\PPP^\pm_{j}$ be the 
Riesz projection from $\mathcal{H}$ onto the
subspace spanned by $\mathcal{Y}^\pm_{j}$ 
\begin{equation}
\label{eq:Riesz}
\PPP^\pm_{j}:\vec{h}\mapsto \alpha^\pm_{j}\mathcal{Y}^\pm_{j}
\end{equation}
with coefficients
\begin{equation}
\label{eq:al-def}
\alpha^\pm_{j}:=\mp \omega\bigl(\,  \vec{h},\, \mathcal{Y}^\mp_{j}\bigr).
\end{equation}
 Note that this is well-defined for all $\vec{h}\in \HHH$ by the smoothness and exponential decay of ${Y}_{j}$.\medskip

Define  $\PPP^\pm=\sum_{j=0}^{k}\PPP^\pm_{j}$ to be the orthogonal projections (in the sense of $\omega$) on the spaces $\mathrm{span}
\left\{\YYY_0^{\pm},\ldots,\YYY_k^{\pm}\right\}$. One can check that the projection on the orthogonal  of the sum of these spaces is
$\PPP_c={\rm Id}-\sum_{\pm}\PPP^{\pm}=(P_c,P_c)$, where $P_c$ is defined in Subsection \ref{sub:Strichartz}. Here ``orthogonal'' means again in the sense of the symplectic form $\omega$. However, in this case, by \eqref{P+-}, it is also the orthogonal projection in the sense of the $L^2$ inner product.
Note that $\PPP^\pm$ and $\PPP_c$ are continuous projections which commute with $\mathcal{L}_Q$, \emph{i.e.}, $\mathcal{L}_Q\PPP^{\pm}=\PPP^{\pm}\mathcal{L}_Q$ and $\mathcal{L}_Q\PPP_c=\PPP_c\mathcal{L}_Q$. Moreover, every function $ \vec{h}\in\mathcal{H}$ can be written as
\begin{equation}
\label{eq:h-decomp}
\vec{h}=\PPP^+\vec{h}+\PPP^-\vec{h}
+\PPP_c  \vec{h}. 
\end{equation}

For $ \ell=0,1$ we will denote by  $\Pi_\ell$ the projection to the first and second component of an element $\vec{h}=(h_0,h_1)$ of $\HHH$.\\


To study the dynamics in the positive direction of time,
we shall make use of the above functional space decomposition and define the \emph{center}, \emph{stable}, \emph{unstable} and \emph{center-stable} subspaces (respectively $\HHH^c$, $\HHH^s$, $\HHH^u$ and $\HHH^{cs}$) as follows 
%
\begin{gather}
\label{defHcus}
\HHH^c:= \PPP_c \HHH,\quad \HHH^u:=\PPP^+\HHH,\quad \HHH^s:=\PPP^-\HHH,\\
\label{defHcs}
\HHH^{cs}:=\HHH^c \oplus \HHH^s=({\rm Id}-\PPP^+)\mathcal{H}=(\PPP_c+\PPP^-)\HHH.
\end{gather} 
Based on the decomposition 
$$\vec{h}=\sum_{\pm}\sum_{0\le j\le k}\alpha_{j}^\pm \mathcal{Y}^\pm_{j}+\PPP_c\vec{h},$$ 
we may introduce the linearized energy norm $\lVert \cdot\rVert_{\HHH_Q}$
$$
\lVert \vec{h}\rVert_{\HHH_Q}^2=\sum_{\pm}\sum_{0\le j\le k}\bigl(\alpha_{j}^\pm\bigr)^2-\omega\bigl(\PPP_c\vec{h},\, \mathcal{L}_Q\,\PPP_c\vec{h}\bigr).
$$
Note that, if $\vec{h}=(h_0,h_1)$,
$$- \omega(\vec{h},\LLL_Q \vec{h})=\int h_1^2+\int |\nabla h_0|^2-(2m+1)\int Q^{2m}h_0^2.$$

By Lemma 3.1 in \cite{DuyckaertsYang23P}, thanks to the absence of zero energy states, one has, $\langle P_c h_0, L_Q P_ch_0 \rangle\gtrsim\lVert P_c h_0\rVert^2_{\dot{H}^1_0}$ and thus $\|\vec{h}\|_{\HHH_Q}\approx \|\vec{h}\|_{\HHH}$ for every $\vec{h}\in \HHH$.\medskip

If  $\vec{h}\in \HHH^{cs}$, then we have 
$\alpha_{j}^+\equiv 0,\forall\,j$ and 
$$\lVert \vec{h}\rVert_{\HHH_Q}^2=\sum_{0\le j\le k} \bigl(\alpha_{j}^-\bigr)^2-\omega\bigl(\PPP_c\vec{h},\, \mathcal{L}_Q \PPP_c\vec{h}\bigr).$$
Similarly, if $h\in \HHH^u$,
$$\lVert \vec{h}\rVert_{\HHH_Q}^2=\sum_{0\le j\le k} \bigl(\alpha_{j}^+\bigr)^2.$$

\subsubsection{The linearized wave equation projected  into various invariant subspaces}
Let us reformulate the linearized equation \eqref{eq:linearizedeq} into the $\HHH-$valued  form
\begin{equation}
\label{eq:linearized}
\partial_{t} \vec{h}=\mathcal{L}_Q  \vec{h}+ \mathcal{R}_Q(h), \quad \vec{h}(t_0)= \vvec{h_0},
\end{equation}
for $t_0\geqslant 0$ with
\begin{align*}
\vec{h}(t)=\begin{pmatrix} 
h\\
\partial_{t} h
\end{pmatrix},\quad
\vvec{h_0}=\begin{pmatrix} 
h_0\\
h_1
\end{pmatrix}\in\mathcal{H},\quad
\mathcal{R}_Q(h):=\begin{pmatrix} 
0\\
R_Q(h)
\end{pmatrix}.
\end{align*}

The nonlinear functional $\mathcal{R}_Q$ is a smooth map from $\dot{H}^1(\Omega)$ to $\mathcal{H} $ thanks to the radial Sobolev embedding and the fact that $R_Q(h)$ is a polynomial in $Q$, $h$.
\medskip

Using the decomposition \eqref{eq:h-decomp} and applying the Riesz projection operators to both sides of  \eqref{eq:linearized}, we obtain
\begin{claim}
	Let $h_c(t)=P_c h(t)$. Then, equation \eqref{eq:linearized} is equivalent to 
	\begin{equation}
	\label{eq:linearized-comp}
	\begin{cases}
	\frac{d}{dt}\alpha_{j}^\pm(t)=\pm e_{j}\alpha^\pm_{j}(t)\mp \omega\bigl( \mathcal{R}_Q(h),\,\mathcal{Y}^\mp_{j}\bigr),&\quad \alpha_{j}^\pm(t_0)=\mp\omega(\vvec{h_0},\mathcal{Y}_{j}^\mp),\quad 0\le j\le k,\\
	\partial_{t}^2 h_c+L_Q h_c=P_c R_Q(h),&\quad(h_c, \partial_{t} h_c)|_{t=t_0}=(P_c h_0, P_c h_1).
	\end{cases}
	\end{equation}
\end{claim}

%

%
%
%
Writing the equations of the Claim in the Duhamel form, we obtain the following equivalent system of  integral equations
	\begin{equation}
	\label{eq:key}
	\begin{cases}
	\alpha_{j}^\pm(t)=\,e^{\pm e_{j}(t-t_0)}\alpha^\pm_{j}(t_0)\mp\int_{t_0}^t e^{\pm e_{j}(t-s)}
	\omega\bigl( \mathcal{R}_Q(h(s)),\,\mathcal{Y}^\mp_{j}\bigr)ds,\\
	h_c(t)=S_Q(t-t_0)\PPP_c\vvec{h_0}+\int_{t_0}^t	\frac{\sin \left((t-s)\sqrt{L_Q}\right)} {\sqrt{L_Q}} P_c R_Q(h(s))\;ds,
	\end{cases}
	\end{equation}
where  $\frac{\sin \left((t-s)\sqrt{L_Q}\right)} {\sqrt{L_Q}} P_c$ makes sense since $L_Q$ is a positive self-adjoint operator on  $P_cL^2$.

\subsection{The Lyapunov-Perron method: role of the reduced integral equation}
In this subsection, by using the Lyapunov-Perron method, we reduce the proof of the existence of the local center-stable manifold to solving an integral equation, subject to a stability condition, which will be referred as  \emph{the reduced integral equation}.

\begin{lemma}
	\label{L:SC}
	Let $\vec{h}\in C([t_0,+\infty),\HHH)$ be a solution of \eqref{eq:linearized-comp} such that $\|\vec{h}(t)\|_{\HHH}$ is at most of polynomial growth  as $t\to\infty$. Then,
	\begin{equation}
	\label{eq:stable-cond}
	\omega(\vec{h}(t),\mathcal{Y}_j^-)=-\alpha_j^+(t)=-\int_{t}^{+\infty}e^{e_{j}(t-s)}
	\omega\bigl( \mathcal{R}_Q(h(s)),  \mathcal{Y}^-_{j}\bigr) ds,\quad t\ge t_0.
	\end{equation}
\end{lemma}
\begin{proof}

	By using \eqref{eq:key}, we have
	$$
	\alpha_{j}^+(t_0)=e^{-e_j (t-t_0)} \alpha_{j}^+(t)+\int_{t_0}^t e^{-e_j (s-t_0)}\omega \bigl( \mathcal{R}_Q(h(s)),\mathcal{Y}_j^-\bigr)ds
	$$
	where $\omega \bigl( \mathcal{R}_Q(h(s)),\mathcal{Y}_j^-\bigr)=-(2e_j)^{-\frac{1}{2}}\langle R_Q(h(s)),Y_j\rangle$. Since $\|\vec{h}(t)\|_{\HHH}$ is at most of polynomial growth, 
	$\langle R_Q(h(t)),Y_j\rangle$ is at most of polynomial growth as $t\to\infty$, and thus the second term on the right side converges as $t\to +\infty$ while the first term tends to zero.
	Thus, we have
	\begin{equation}
	\label{eq:alpha+}
	\alpha^+_{j}(t_0)= \int_{t_0}^{+\infty} e^{-e_j (s-t_0)}\omega \bigl( \mathcal{R}_Q(h(s)),\mathcal{Y}_j^-\bigr)ds,
	\end{equation}
	which yields \eqref{eq:stable-cond}.
    The proof is complete. 
\end{proof}

\begin{remark}
	We will see that the stability condition \eqref{eq:stable-cond} is sufficient for the boundedness of $\|\vec{h}\|_{\HHH}$. 
\end{remark}

To prove the existence of the local manifold, combining \eqref{eq:linearized-comp} with \eqref{eq:stable-cond}, we are reduced to solving the following integral equation system: for $j\in \{0,1,\ldots, k\}$
\begin{equation}
\label{eq:Reduced}
\begin{cases}
\alpha_j^+(t)=\int_{t}^{+\infty}
e^{e_{j}(t-s)}\omega
\bigl( \mathcal{R}_Q(h(s)),  \mathcal{Y}^-_{j}\bigr) ds\\
\alpha_j^-(t)=e^{- e_{j} (t-t_0)}\alpha_j^-(t_0)
+\int_{t_0}^{t}e^{- e_{j}(t-s)}\omega\bigl( \mathcal{R}_Q(h(s)), \mathcal{Y}^+_{j}\bigr) ds\\
h_c(t)=
S_Q(t-t_0)\PPP_c \vvec{h_0}+\int_{t_0}^{t}\frac{\sin \left((t-s)\sqrt{L_Q}\right)}{\sqrt{L_Q}}P_c R_Q(h(s))ds,
\end{cases}
\end{equation} 
This is referred as \emph{the reduced integral equation}.

\subsection{Convergence to the stationary solution in the dispersive norm}
Recall from \eqref{defS} the definition of $\Ssf(I)$.
The main result of this subsection is the following proposition:
\begin{proposition}
\label{P:CVdispersive_norm}
 There exists $\eps_0$ with the following property. Let $\eps\in (0,\eps_0]$, $(u_0,u_1)\in \HHH$ and $u$ be the corresponding solution of \eqref{eq:NLW}. Assume that there exists a solution $v_L$ of the linear wave equation \eqref{eq:LW} such that 
 \begin{gather}
  \label{P10}
  \left\| v_L\right\|_{\Ssf([0,\infty))}+\sup_{t\geq 0} \left( \|v_L(t)\|_{L^{\infty}}+\left\|\la x\ra^{-4}\partial_tv_L(t)\right\|_{L^2} \right)\leq \eps\\
  \label{P11}
  \sup_{t\geq 0}\left\| \vec{u}(t)-(Q,0)-\vec{v}_L(t)\right\|_{\HHH}\leq \eps.
 \end{gather}
 Let $h=u-Q$. Then 
 \begin{gather}
  \label{P12}
  h\in \Ssf([0,\infty)),\quad \|h\|_{\mathsf{S}([0,\infty))} \lesssim \eps,\\
  \label{P13}
  \sup_{T\geq 0}\left(
  \left\| S_Q(\cdot)\PPP_c\left( \vec{h}(T)\right)\right\|_{\Ssf([0,\infty))} +\sum_{j,\pm} \left| \omega\left( \YYY^{\pm}_j,\vec{h}(T)\right)\right|\right)\lesssim \eps.
 \end{gather}
\end{proposition}

\begin{corollary}
\label{CR:CVdispersive_norm}
Let $u$ be a solution of \eqref{eq:NLW} defined for $t\in [0,\infty)$ such that 
$$\vec{u}(t) \xrightharpoonup[t\to\infty]{}(Q,0)\text{ in } \HHH.$$
Let $h=u-Q$. Then $h\in \Ssf([0,\infty))$ and 
\begin{equation}
 \label{dis10}
 \lim_{T\to\infty}\left\| S_Q(\cdot)(\mathrm{Id} -\PPP^+)\vec{h}(T)\right\|_{\Ssf([0,\infty)}=0
\end{equation} 
\end{corollary}
Corollary \ref{CR:CVdispersive_norm} says that $u\in \MMM_k^+$ converges to $(Q,0)$ also in the sense of the space-time dispersive norms that appear in standard well-posedness and perturbation theories for equation \eqref{eq:NLW}. The same result was proved in \cite{JiLiScXu17} for the defocusing critical nonlinear wave equation with a potential.

Let $u$ be as in Corollary \ref{CR:CVdispersive_norm}.
By the main result of \cite{DuyckaertsYang21}, there exists a solution $v_L$ of the linear wave equation \eqref{eq:LW}, such that 
\begin{equation}
 \label{dis11}
 \lim_{t\to\infty}\left\|\vec{h}(t)-\vec{v}_L(t)\right\|_{\HHH}=0.
\end{equation}
By Remark \ref{R:mixed_free}, and standard estimate for the  linear wave equation, 
$$\|v_L\|_{\Ssf([0,\infty))}<\infty\text{ and } \lim_{t\to\infty} \left(\|v_L(t)\|_{L^{\infty}}+\left\|\langle x\rangle^{-4} v_L(t)\right\|_{L^2}\right)=0,$$
and the 
the conclusion \eqref{dis10} of Corollary \ref{CR:CVdispersive_norm} follows immediately from Proposition \ref{P:CVdispersive_norm} applied to $t\mapsto u(t+t_0)$, where $t_0$ is large. 
\begin{proof}[Proof of Proposition \ref{P:CVdispersive_norm}]

We let 
\begin{equation}
 \label{dis12}
 g(t)=h(t)-v_L(t)=u(t)-Q-v_L(t).
\end{equation} 
We see that $g$ satisfies 
\begin{equation}
 \label{dis13}
 \partial_t^2g+L_Qg=\left( Q+v_L+g \right)^{2m+1}-Q^{2m+1}-(2m+1)Q^{2m}g=: \widetilde{R}_Q(g,v_L).
\end{equation} 
We have $\widetilde{R}_Q(g,v_L)=(Q+v_L+g)^{2m+1}-(Q+g)^{2m+1}+R_Q(g)$, where $R_Q$ satisfies \eqref{pppp}. Thus
$$\left|\widetilde{R}_Q(g,v_L)\right|\lesssim (Q+g)^{2m}|v_L|+|v_L|^{2m+1}+Q^{2m-1}g^2+|g|^{2m+1},$$
and, since $g^{2m}|v_L|\lesssim |g|^{2m+1}+|v_L|^{2m+1}$, 
\begin{equation}
 \label{dis20}
 \left|\widetilde{R}_Q(g,v_L)\right|\lesssim Q^{2m}|v_L|+|v_L|^{2m+1}+Q^{2m-1}g^2+|g|^{2m+1}.
\end{equation} 
From the definition \eqref{def_Pc} of $P_c$, we obtain, for a constant $c>0$,
\begin{multline}
 \label{dis21}
 \left|P_c(\widetilde{R}_Q(g,v_L))\right|\lesssim Q^{2m}|v_L|+|v_L|^{2m+1}+Q^{2m-1}g^2+|g|^{2m+1}\\
 +
 \left( \left\| \langle x\rangle^{-3} v_L(t)\right\|_{L^2} +\left\| v_L(t)\right\|_{L^{2(2m+1)}}^{2m+1}\right)e^{-c|x|}\\+\left(\|\langle x\rangle^{-3}g(t)\|^2_{L^2}+\|g(t)\|^{2m+1}_{L^{2(2m+1)}}\right)e^{-c|x|},
\end{multline}

Let $\tau>0$. In view of \eqref{dis21} and the equation \eqref{dis13} satisfied by $g$, Proposition \ref{P:mixed} implies 
\begin{equation}
 \label{dis22}
 \left\|P_cg\right\|_{\Ssf([0,\tau])}\lesssim \|\vec{g}(0)\|_{\HHH}+ \|v_L\|_{\Ssf([0,\tau])} +\|v_{L}\|^{2m+1}_{\Ssf([0,\tau])}+\|g\|^2_{\Ssf([0,\tau])}+\|g\|^{2m+1}_{\Ssf([0,\tau])}.
\end{equation} 
Hence, going back to $h=v_L+g$.
\begin{equation}
 \label{dis30}
 \left\|P_ch\right\|_{\Ssf([0,\tau])}\lesssim \|\vec{g}(0)\|_{\HHH}+ \|v_L\|_{\Ssf([0,\tau])} +\|v_{L}\|^{2m+1}_{\Ssf([0,\tau])}+\|h\|^2_{\Ssf([0,\tau])}+\|h\|^{2m+1}_{\Ssf([0,\tau])}.
 \end{equation} 
 Next, we consider the coefficients $\alpha_j^{\pm}(t)=\mp\omega(\vec{h},\YYY_j^{\mp})$. Since $h$ is bounded, it satisfies the reduced equations \eqref{eq:Reduced}.  By the second equation in \eqref{eq:Reduced}, for all $t\in [0,\tau]$,
 \begin{equation*}
  \left|\alpha_j^{-}(t)\right|\lesssim e^{-e_j t}\left|\alpha_j^{-}(0)\right|+\int_{0}^t e^{-e_j(t-s)} \|h(s)\|^{2}_{L^{12}}ds
 \end{equation*}
(we have discarded the contributions of the higher powers of $h$, using \eqref{P10}, \eqref{P11} and radial Sobolev embedding). The integral in the right-hand side can be viewed as a convolution between $\|h\|^{2}_{L^{12}}\indic_{[0,\tau]}$ and $e^{-e_j \cdot}\indic_{[0,\infty)}$ (which is in $L^1\cap L^{\infty}$). Thus
\begin{equation}
\label{dis31}
\left\| \alpha_{j}^{-}\right\|_{L^2([0,\tau])}+\left\| \alpha_{j}^{-}\right\|_{L^{2m+1}([0,\tau])}\lesssim \|h\|^2_{L^4([0,\tau],L^{12})}+|\alpha_j^-(0)|. 
\end{equation} 
By the first equation in \eqref{eq:Reduced}, 
\begin{equation}
 \label{dis40}
 \alpha_j^+(t)=\underbrace{\int_{t}^{\tau} e^{e_j(t-s)}\omega\left(\mathcal{R}_Q(h(s)),\YYY_j^- \right)ds}_{A_1(t)}+\underbrace{\int_{\tau}^{\infty} e^{e_j(t-s)} \omega\left( \mathcal{R}_Q(h(s)),\YYY_j^- \right)ds}_{A_2(t)}.
\end{equation}
Viewing $A_1$ as a convolution between $e^{e_j\cdot}\indic_{(-\infty,0)}$ and $\omega\left( \mathcal{R}_Q(h),\YYY_j^- \right)\indic_{(0,\tau)}$, we obtain
\begin{equation}
 \label{dis41}
 \left\|A_1\right\|_{L^{2m+1}([0,\tau])\cap L^2([0,\tau])}\lesssim \|h\|^2_{\Ssf([0,\tau])}.
\end{equation} 
On the other hand
$$|A_2(t)|\lesssim e^{e_j(t-\tau)} \int_{\tau}^{\infty} e^{e_j(\tau-s)} \left( \|v_L(s)\|^2_{L^{12}}+\|g(s)\|_{L^{\infty}}^2 \right)ds.$$
Since 
\begin{gather*}
e^{e_j(t-\tau)} \int_{\tau}^{\infty} e^{e_j(\tau-s)}\|g(s)\|_{L^{\infty}}^2ds\lesssim e^{e_j(t-\tau)} \sup_{s \geq 0}\|g(s)\|^2_{L^{\infty}},\\
e^{e_j(t-\tau)} \int_{\tau}^{\infty} e^{e_j(\tau-s)}\|v_L(s)\|_{L^{12}}^2ds\lesssim e^{e_j(t-\tau)}\|v_L\|_{L^4((\tau,\infty),L^{12})}^2.
\end{gather*}

we obtain
\begin{equation}
 \label{dis42}
 \left\|A_2\right\|_{L^{2m+1}([0,\tau])\cap L^{2}([0,\tau])}\lesssim \|v_L\|^{2}_{\mathsf{S}([0,\infty))}+\sup_{s\geq 0}\|g(s)\|^2_{L^{\infty}}.
\end{equation} 
By the assumptions \eqref{P10}, \eqref{P11}, and the radial Sobolev inequality, we see that
\begin{equation}
\label{dis42'}
\sup_{s\geq 0}\|g(s)\|_{L^{\infty}}\lesssim \sup_{s\geq 0}\|\vec{g}(s)\|_{\HHH}\lesssim \eps,\quad |\alpha^-_j(0)|\lesssim \eps
\end{equation}
Combining \eqref{P10}, \eqref{P11}, \eqref{dis30}, \eqref{dis31}, \eqref{dis40}, \eqref{dis41}, \eqref{dis42} and \eqref{dis42'}, we obtain
\begin{equation*}
 \|h\|_{\Ssf([0,\tau])}\lesssim \eps+\|h\|^2_{\Ssf([0,\tau])}+\|h\|^{2m+1}_{\Ssf([0,\tau])}. 
\end{equation*} 
Taking $\eps$ small enough, we obtain, by a straightforward continuity argument,
$$ \forall \tau>0,\quad \|h\|_{\Ssf([0,\tau])}\lesssim \eps,$$
and thus $h\in \Ssf([0,\infty))$ as announced.

Next, we use the last equation in \eqref{eq:Reduced} and see that, for all $T>0$,  
$$\left\|P_c(h)-S_Q(\cdot-T)\PPP_c(\vec{h}(T))\right\|_{\Ssf([T,\infty))}\lesssim \left\|h\right\|^2_{\Ssf([T,\infty))}$$
and thus, 
$$\forall T\geq 0,\quad \left\| S_Q\PPP_c(\vec{h}(T) )\right\|_{\Ssf([0,\infty))}\lesssim \eps.$$
Since  by \eqref{P10}, \eqref{P11},
$$\sup_{t\geq 0} \left|\omega\left(\vec{h}(t),\YYY_j^{\pm}\right)\right|\lesssim \eps,$$
we deduce \eqref{P13}.
\end{proof}

\subsection{Existence of the reduction map }\label{sect:proof-contrac}
In this subsection, we show the existence of 
 the local center-stable manifold by using a fixed point argument. More precisely, we show that there exists
 an open neighborhood $\mathscr{B}_{\delta}$ of the origin in $\HHH^{cs}$, a locally Lipschitz map $\mathscr{G}^+$ defined on $\mathscr{B}_{\delta}$, taking values in $\HHH^u$, such that for every vector $\vvec{h_0}$ belonging to 
 the graph of $\mathscr{G}^+$, there exists a global solution $\vec{h}(t)$ to the reduced integral equation \eqref{eq:Reduced} with initial data equal to $\vvec{h_0}$.
 
 To state the main result of this subsection, let us first introduce some notations.
 For a given vector $\vvec{v_0}\in \HHH^{cs}$, we have
$$
\overrightarrow{S_Q}(t)(\vvec{v_0})=\sum_{j=0}^ke^{-e_{j}t}\PPP^-_{j}\bigl(\vvec{v_0}\bigr)+\overrightarrow{S_Q}(t)\left(\PPP_c\vvec{v_0}\right),
$$
where as before $\PPP_j^-:\vvec{v_0}\mapsto \omega(\vvec{v_0},\mathcal{Y}_j^+)\mathcal{Y}_j^-$ and the action of the linearized wave group on $\PPP_c\HHH$ can be written as
$$
\overrightarrow{S_Q}(t)\circ \PPP_c=
\begin{pmatrix} 
\cos t \sqrt{L_Q} & 
\frac{\sin t\sqrt{L_Q}}{\sqrt{L_Q}}  \\
- \sqrt{L_Q} \,\sin t \sqrt{L_Q}  & \cos t \sqrt{L_Q} 
\end{pmatrix}
\circ \PPP_c.
$$

Recall that by Proposition \ref{P:mixed}, if $\vvec{v_0}\in \PPP_c(\HHH)$, then $S_Q\vvec{v_0}\in \Ssf(\R)$, and 
\begin{equation*}
 \left\|S_Q\vvec{v_0}\right\|_{\Ssf(\R)}\lesssim \|\vvec{v_0}\|_{\HHH}.
\end{equation*} 
As a consequence, if $\vvec{v_0}\in \HHH^{cs}$, then $S_Q\vvec{v_0}\in \mathsf{S}\left((0,+\infty)\right)$ and 
\begin{equation}
\label{def_vertiii}
\vertiii{ \vvec{v_0}}:=\max_{0\le j\le k}|\omega(\vvec{v_0},\mathcal{Y}_j^+)|+\left\|S_Q(\PPP_c \vvec{v_0})\right\|_{\mathsf{S}([0,\infty))} \lesssim \|\overrightarrow{v_0}\|_{\HHH}.
\end{equation}
If $\delta>0$, we define the following open neighborhood of $0$ in $\HHH^{cs}$:
$$ \mathscr{B}_{\delta}:=\left\{\vvec{v_0}\in \HHH^{cs}\text{ s.t. } \vertiii{\vvec{v_0}}<\delta\right\}.$$

Now, we can state our main result.
\begin{proposition}
	\label{P:good}
	There exists a constant $M>0$ with the following property. 
	For $\delta>0$ sufficiently small, for all $\vvec{v_0}\in \mathscr{B}_{\delta}$, there exists a unique global solution $\vec{h}\in C^0\left( [0,+\infty),\HHH \right)$ of the reduced equations \eqref{eq:Reduced} such that 
\begin{equation}
 \label{initial_data}
 ({\rm Id}-\PPP^+)(\vec{h})(0)=\vvec{v_0}\,,
\end{equation}
and 
\begin{equation}
 \label{integrabilityh}
 h\in \Ssf([0,\infty)),\quad \|h\|_{\Ssf([0,\infty))}\leq M\delta.
\end{equation} 
Furthermore
\begin{equation}
 \label{scattering_to_Q}
 \sup_{t\geq 0} \left\| \vvec{S_Q}(t)\left(\vvec{v_0}\right)-\vec{h}(t)\right\|_{\HHH}\lesssim \vertiii{\vvec{v_0}}^2.
\end{equation}
The map $\Phi:\vvec{v_0}\to \vec{h}$ is $C^{\infty}$ from $\mathscr{B}_{\delta}$  to $C^0([0,\infty),\HHH)$. It is Lipschitz for the norm $\vertiii{\cdot}$, in the sense that:
$$ \sup_{t\geq 0}\left\|\Phi(\vvec{v_0})(t)-\Phi(\vvec{w_0})(t)\right\|_{\HHH}\lesssim \vertiii{\vvec{v_0}-\vvec{w_0}}.$$
\end{proposition}
 In the statement of Proposition \ref{P:good} as in all the article, we use Fr\'echet derivatives to define the notion of $C^{\infty}$ functions between Banach spaces.
\begin{proof}
For $\vvec{v_0}\in \mathscr{B}_{\delta}$ and $\vec{h}\in L^\infty\left((0,\infty),\HHH\right)$ such that the following integrals make sense, we define
\begin{align*}
\Psi^c(h,\vvec{v_0})(t)&=S_Q(t) \PPP_c\left(\vvec{v_0}\right)+\int_{0}^{t}\frac{\sin (t-s)\sqrt{L_Q}}{\sqrt{L_Q}}P_c(R_Q(h(s)))ds,\\
\intertext{ and, for all $j=0,1,\ldots,k$,}
\Psi^{+,j}\big(h,\vvec{v_0}\big)(t)&=\int_{t}^{+\infty}
e^{e_{j}(t-s)}
\omega\bigl(\RRR_Q(h(s)),  \mathcal{Y}_j^-\bigr) ds\\
\Psi^{-,j}\big(h,\vvec{v_0}\big)(t)&=\omega\left(\vvec{v_0},\YYY_j^+\right)e^{-e_jt}+\int_{0}^{t}
e^{-e_{j}(t-s)}
\omega\bigl(\RRR_Q(h(s)),  \mathcal{Y}^+_{j}\bigr) ds.
\end{align*}
We also define
$$
\Psi(h,\vvec{v_0})(t)=\sum_{j=0}^k\Bigl(\Psi^{+,j}\big(h,\vvec{v_0}\big)(t)\;\mathcal{Y}^+_{j}+\Psi^{-,j}\big(h,\vvec{v_0}\big)(t)\;\mathcal{Y}^-_{j}\Bigr)+
\Psi^c(h,\vvec{v_0})(t).
$$
Let
$$\mathsf{D}=\mathsf{D}_{M\delta}=\left\{ h\in \mathsf{S}([0,\infty));\; \|h\|_{\mathsf{S}([0,\infty))}\leq M\delta\right\}.$$
We will prove that if $\delta$ is small enough, $\Psi$ is a contraction on $\mathsf{D}$. This will follow from the estimate \eqref{pppp} on $R_Q$ and the mixed Strichartz/weighted energy estimates of Proposition \ref{P:mixed}.

\medskip

\noindent\emph{Stability of $\mathsf{D}$.}
Let $h\in \mathsf{D}$.
Using that by \eqref{pppp}, $R_Q(h)\lesssim Q^{2m-1}h^2+|h|^{2m+1}$, we obtain, 
by Proposition \ref{P:mixed}, 
\begin{multline*}
\left\| \Psi^c(h,\vvec{v_0})\right\|_{\Ssf([0,\infty))} \lesssim \left\|S_Q(\cdot)\PPP_c\vvec{v_0}\right\|_{\Ssf([0,\infty))}
\\
+\left\| h^{2m+1}\right\|_{L^1([0,\infty),L^2(\Omega))}+\left\|\langle x\rangle^2Q^{2m-1}h^2\right\|_{L^2([0,\infty)\times \Omega)}.
\end{multline*}
Since $\langle x\rangle^2 Q^{2m-1} \lesssim \langle x\rangle^{-3}\in L^3(\Omega)$, we have
$$
\left\|\langle x\rangle^2Q^{2m-1}h^2\right\|_{L^2([0,\infty)\times \Omega)}\lesssim \left\| h^2\right\|_{L^2([0,\infty),L^6)}=\left\| h\right\|_{L^4([0,\infty),L^{12})}^2$$
and thus
\begin{equation}
\label{FP20}
\left\| \Psi^c(h,\vvec{v_0})\right\|_{\Ssf([0,\infty))} 
\leq \left\|S_Q\PPP_c\vvec{v_0}\right\|_{\Ssf([0,\infty))}+C((M\delta)^2+(M\delta)^{2m+1}).
\end{equation}
On the other hand, using the exponential decay of $Y_j$, we obtain
$$ \left|\Psi^{+,j}(h,\vvec{v_0})\right|\lesssim \int_{t}^{\infty}e^{e_j(t-s)} \left(\|h(s)\|^{2}_{L^{12}}+\|h(s)\|^{2m+1}_{L^{2(2m+1)}}\right)ds.$$
By Young's inequality for convolution in the time variable, we obtain
\begin{equation}
 \label{FP30} 
 \left|\Psi^{+,j}(h,\vvec{v_0})\right|_{(L^2\cap L^{\infty})} \lesssim (M\delta)^2+(M\delta)^{2m+1}.
\end{equation} 
By a similar argument,
\begin{multline}
 \label{FP31} 
 \left|\Psi^{-,j}(h,\vvec{v_0})\right|_{(L^2\cap L^{\infty})} \lesssim \left|(\omega(\vvec{v_0},\YYY_j^+)\right|+(M\delta)^2+(M\delta)^{2m+1}\\ \lesssim \delta +(M\delta)^2+(M\delta)^{2m+1}.
\end{multline} 
Summing up, we obtain that there exist constants $C_0$, $C_1$ and $C_2$ such that
\begin{equation}
\label{Stability}
\left\|\Psi\left(h,\vvec{v_0}\right)\right\|_{\Ssf([0,\infty)}\leq C_0 \delta +C_1(M\delta)^2+C_2(M\delta)^{2m+1}. 
\end{equation} 
Taking $M=2C_0$, we obtain that $\Psi$ maps $\mathsf{D}$ into itself for small $\delta$. 

\medskip

\noindent\emph{Contraction property.}
The proof of the fact that $\Psi$ is a contraction on $\mathsf{D}$ is similar.  For example, by \eqref{pppp}, if $h,\ell \in \mathsf{D}$,
\begin{multline*}
\left\|\Psi^{c}\left(h,\vvec{v_0}\right)-\Psi^c\left(\ell,\vvec{v_0}\right)\right\|_{\Ssf([0,\infty)}\\
\lesssim \left\| |h-\ell| \left(h^{2m}+\ell^{2m}\right)\right\|_{L^1([0,\infty),L^2)}+\left\| \langle x\rangle^2 Q^{2m-1}(|h|+|\ell|) |h-\ell|\right\|_{L^2([0,\infty)\times \Omega)}
\\
\lesssim_{M}\delta^{2m} \|h-\ell\|_{L^{2m+1}([0,\infty),L^{2(2m+1)})}+\delta\left\|h-\ell\right\|_{L^{4}([0,\infty), L^{12})}\lesssim_M \delta \|h-\ell\|_{\Ssf([0,\infty)}.
\end{multline*}
The bounds of the other terms go along the same lines. This proves that $\Psi$ is a contraction, and thus that for any $\vvec{v_0}\in \mathscr{B}_{\delta}$, the reduced equations \eqref{eq:Reduced} have a unique solution $h$ that satisfies \eqref{initial_data} and \eqref{integrabilityh}. Also, by the above computations, we obtain that the fixed point $h=\Psi(h,\vvec{v_0})$ satisfies
$$ \left\|h-S_Q(\cdot)(\vvec{v_0})\right\|_{\Ssf([0,\infty)}+  \sup_{t\geq 0} \left\|\vec{h}(t)-\vvec{S_Q}(t)(\vvec{v_0})\right\|_{\HHH} \lesssim \vertiii{v_0}^2, $$
which implies \eqref{scattering_to_Q}.

\medskip
 
\noindent\emph{Smoothness}. We are left with proving that the map 
$\vvec{v_0}\mapsto \vec{h}$ defined above is a $C^{\infty}$ map from $\mathscr{B}_{\delta}$ to $C^{0}([0,\infty),\HHH)$.

We first prove that $\vvec{v_0}\mapsto h$ is $C^{\infty}$ from $\mathscr{B}_{\delta}$ to $\Ssf([0,\infty))$. 
This follows from Lemma \ref{L:smooth} in the appendix, once we have checked the following two properties:
\begin{itemize}
 \item $\left\| \Psi(h,\vvec{v_0})-\Psi(h,\vvec{w_0})\right\|_{\Ssf([0,\infty)}\lesssim \left\|\vvec{v_0}-\vvec{w_0}\right\|_{\HHH}$
 \item $\Psi$ is a $C^{\infty}$ map from $\mathsf{D}_o\times\mathscr{B}_{\delta}$ to $\mathsf{D}$, 
\end{itemize}
where $\mathsf{D}_o$ is an open neighborhood of $\mathsf{D}$, say $\mathsf{D}_{o}=\left\{ h\in \mathsf{S}([0,\infty));\; \|h\|_{\mathsf{S}([0,\infty))}< \frac{3}{2}M\delta\right\}$. 
Since 
$\Psi(h,\vvec{v_0})-\Psi(h,\vvec{w_0})=\sum_{j=0}^k \omega(\vvec{v_0}-\vvec{w_0},\mathcal{Y}_j^+)e^{-e_j t}\mathcal{Y}_j^-+S_Q(t)(\PPP_c(\vvec{v_0}-\vvec{w_0})),$
the first point is a direct consequence of Proposition \ref{P:mixed}. 

To get the second one, we observe that, by the above argument,  $\Psi$ maps $\mathsf{D}_o\times \mathscr{B}_{\delta}$ to $\mathsf{D}$. We write 
\begin{equation}
 \label{exp_Psi}
\Psi(h,\vvec{v_0})=S_Q(t)\PPP_c(\vvec{v_0})+ \sum_{j=0}^k \omega\left(\vvec{v_0},\YYY_j^+\right)e^{-e_jt}\mathcal{Y}^-_j+\Psi_L(R_Q(h)) ,
\end{equation} 
where $\Psi_L$ is the following linear operator: 
\begin{multline*}
\Psi_L(R)=\sum_{j=0}^k\Bigl(
\mathcal{Y}^+_{j}\int_{t}^{+\infty}
e^{e_{j}(t-s)}
\omega\bigl(\RRR(s),  \mathcal{Y}_j^-\bigr) ds\;\\
+\mathcal{Y}^-_{j}\int_{0}^{t}
e^{-e_{j}(t-s)}
\omega\bigl(\RRR(s),  \mathcal{Y}^+_{j}\bigr) ds \Bigr)+    \int_{0}^{t}\frac{\sin (t-s)\sqrt{L_Q}}{\sqrt{L_Q}}P_c(R(s))ds
\end{multline*}
(with the notation $\RRR(s)=
\left(\begin{smallmatrix} 0\\ R(s) \end{smallmatrix}\right)$.
By Proposition \ref{P:mixed}, $\vvec{v_0}\mapsto S_Q(t)\PPP_c(\vvec{v_0})$ is a bounded linear operator from $\HHH$ to $\mathsf{S}([0,\infty))$. This is thus a $C^{\infty}$ map between these two spaces.

One can also check, using Proposition \ref{P:mixed} again, that $\Psi_L$ is a bounded linear operator (thus a $C^{\infty}$ map) from 
$$E_2:=L^1\left([0,\infty),L^2(\Omega)\right)+L^2\left([0,\infty),L^2(\Omega,\langle x\rangle^{4}dx)\right).$$
to $\mathsf{S}([0,\infty))$.

It thus remains to check that 
$h\mapsto R_Q(h)$
is a $C^{\infty}$ map from $$E_1:=L^4\left([0,\infty),L^{12}(\Omega)\right)\cap L^{2m+1}\left([0,\infty),L^{2(2m+1)}(\Omega)\right)$$ to 
$E_2$. Since $R_Q(h)$ is a polynomial in $h$ and maps continuously $E_1$ to $E_2$, the result follows easily. 

To conclude the proof that 
$\vvec{v_0}\mapsto (h,\partial_th)$ is a
$C^{\infty}$ map from $\mathscr{B}_{\delta}$ to $C^0([0,\infty),\HHH)$, we observe that 
$h=\Psi(h,\vvec{v_0})$. We can easily check (with similar arguments as above) that $(\Psi,\partial_t\circ \Psi)$ is $C^{\infty}$ from $\Ssf([0,\infty))\times\mathscr{B}_{\delta}$ to $C^0([0,\infty),\HHH)$, which, combined with the fact that $\vvec{v_0}\mapsto h$ is $C^{\infty}$ from $\mathscr{B}_{\delta}$ to $\Ssf([0,\infty))$ yields the desired conclusion.
\end{proof}

\begin{corollary}
 \label{CR:local-manifold}
 Let 
 \begin{multline*}
  \mathcal{N}_{\delta}=  \mathcal{N}_{\delta,k}=
  \\
  \left\{(u_0,u_1)\in \MMM_k^+,\; \vertiii{({\rm Id}-\PPP^+)\Big[(u_0,u_1)-(Q,0)\Big]}<\delta,\; \left\|u-Q\right\|_{\Ssf([0,\infty))}<M\delta\right\},
 \end{multline*}
 where as before $Q=Q_k$, $u$ is the solution of \eqref{eq:NLW} with initial data $(u_0,u_1)$, and $M$ is as in Proposition \ref{P:good}.
Then for $\delta$ small, $\mathcal{N}_{\delta}$ is a $C^{\infty}$ submanifold of $\HHH$ of codimension $k+1$.
\end{corollary}
\begin{proof}
We will write $\mathcal{N}_{\delta}$  as the graph of an appropriate function from $\BBB_{\delta}$ to $\HHH^u$, and use Lemma \ref{L:graph} to conclude.
 Let $\widetilde{\NNN}_{\delta}=\left\{ (Q,0)+\Phi(\vvec{v_0})(0),\; \vvec{v_0}\in \mathscr{B}_{\delta}\right\}$, where $\Phi$ is the map defined in Proposition \ref{P:good}.
 We first prove that $\mathcal{N}_{\delta}=\widetilde{\NNN}_{\delta}$. 
 
 First, let $(u_0,u_1)\in \NNN_{\delta}$ and $h=u-Q$. Since $(u_0,u_1)\in \MMM_k^+$, $h$ is solution of the reduced integral equations \eqref{eq:Reduced}. Furthermore, by the definition of $\mathscr{B}_{\delta}$, $\|h\|_{\mathsf{S}((0,\infty))}<M\delta$. By the uniqueness part of Proposition \ref{P:good}, $h=\Phi(\vvec{v_0})$, where $\vvec{v_0}=(Id-\PPP^+)(\vvec{h}(0))=(Id-\PPP^+)(u_0-Q,u_1)\in \mathscr{B}_{\delta}$ by the definition of $\NNN_{\delta}$. Thus $(u_0,u_1)=(Q,0)+\vvec{h}(0)=(Q,0)+\Phi(\vvec{v_0})(0)$, which proves $(u_0,u_1)\in \widetilde{\NNN}_{\delta}$.

 On the other hand, if $(u_0,u_1)\in \widetilde{\mathcal{\NNN}}_{\delta}$, then by Proposition \ref{P:good}, $\vvec{u}(t)=(Q,0)+\vvec{h}(t)$, with $\|h\|_{\mathsf{S}((0,\infty))}<M\delta$ and $\vertiii{(Id-\PPP^+)\left( \vec{h}(0) \right)}<\delta$, which shows that $(u_0,u_1)\in \NNN_{\delta}$, concluding the proof that 
 $$\NNN_{\delta}=\left\{ (Q,0)+\Phi(\vvec{v_0})(0),\; \vvec{v_0}\in \mathscr{B}_{\delta}\right\}.$$
 Since
 $$ (Id-\PPP^+)\Phi(\vvec{v_0})(0)=\vvec{v_0},$$
 we can rewrite this as 
 $$\mathcal{N}_{\delta}=\left\{ (Q,0)+\vvec{v_0}+\vartheta(\vvec{v_0}),\; \vvec{v_0}\in \mathscr{B}_{\delta}\right\},$$
 where
 $$\vartheta:\vvec{v_0}\mapsto \PPP^+\left(\Phi(\vvec{v_0})(0)\right).$$
 Since $\mathscr{B}_{\delta}$ is an open subset of $\HHH^{cs}$, and $\vartheta$
 is a $C^{\infty}$ map from $\HHH^{cs}$ to $\HHH^u$, by Lemma \ref{L:graph} the result follows.  
\end{proof}
\begin{remark}
\label{R:tangent_space}
 By \eqref{scattering_to_Q}, $\vartheta'(0)=0$, which proves that the tangent space of $\NNN_{\delta}$ at $(Q,0)$ is exactly $\HHH^{cs}$. 
 In particular, $\HHH^u$ is transversal to $\NNN_{\delta}$ at $u$. This shows that if $(u_0,u_1)\in \MMM_k^+$ is close to $(Q,0)$, and $\vvec{h_0}\in \HHH^u\setminus \{(0,0)\}$ is small, then $(u_0,u_1)+\vvec{h_0}\notin \MMM_k^+$.
\end{remark}
\begin{remark}
 The manifold $\NNN_{\delta}$ is locally invariant by the flow of \eqref{eq:NLW}: if $(u_0,u_1)\in \NNN_{\delta}$, then $\vec{u}(t)\in \NNN_{\delta}$ for small $t\in \R$. Indeed (fixing a small $t$), the fact that $\vec{u}(t)$ is in $\MMM_k^+$ follows immediately from the definition of $\MMM_k^+$, and the bound
 $\left\|u-Q\right\|_{\Ssf([t,\infty))}<M\delta$ is a direct consequence of the same bound for $t=0$ and the fact that
 $u-Q\in \Ssf(t,+\infty)$ for small negative $t$. Finally, the bound
 $$\vertiii{({\rm Id}-\PPP^+)\Big[\vec{u}(t)-(Q,0)\Big]}<\delta$$
  follows from the corresponding bound at $t=0$ and the continuity of the flow of \eqref{eq:NLW}. 
\end{remark}

\subsection{The global manifold}
\label{sub:global}
We will again drop the index $k$ for the sake of brevity, writing $\mathcal{M}=\mathcal{M}_j^+$. Note that by Theorem \ref{T:DY19}, we have:
$$\mathcal{M}=\left\{(u_0,u_1),\; \vec{u}(t)\xrightharpoonup[t\to\infty]{} (Q,0)\right\},$$
where $Q=Q_k$ (recall that the integer $k\geq 0$ is fixed in all this section), and the weak convergence is in $\HHH$. 

In this subsection, we conclude the proof of Theorem \ref{T:cs}, proving that $\mathcal{M}$ is a submanifold of $\HHH$ of codimension $k+1$. We fix $(U_0,U_1)\in \mathcal{M}$ and let $U$ be the corresponding solution of \eqref{eq:NLW}. We let $V_L$ be the solution of the linear wave equation \eqref{eq:LW} such that 
\begin{equation}
 \label{P30}
 \lim_{t\to\infty} \left\|\vvec{U}(t)-\vvec{V_L}(t)-(Q,0)\right\|_{\HHH}=0. 
\end{equation} 
We fix a small $\delta>0$, so that the conclusion of Proposition \ref{P:good} holds with this $\delta$, and let $t_0\geq 0$ such that
\begin{gather}
 \label{P31} \forall t\geq t_0,\quad \left\|\vvec{U}(t)-\vvec{V_L}(t)-(Q,0)\right\|_{\HHH}<\delta^6\\
 \label{P32} \left\|V_L\right\|_{\Ssf([t_0,\infty))} +\sup_{t\geq t_0}\left(\left\|V_L(t)\right\|_{L^{\infty}}+\left\| \langle x\rangle^{-4}\partial_tV_L(t)\right\|_{L^2}\right)<\delta^6\\
 \label{P33} \forall t\geq t_0,\quad \int_{\Omega}\indic_{\{|x|<\delta^{-3}+t-t_0\}} \left|\nabla_{t,x}V_L(t,x)\right|^2dx<\delta^6.
\end{gather}
To see that, at fixed $\delta$, there exists $t_0$ such that \eqref{P33} holds, recall that
\begin{equation}
\label{asymptotic_VL}
\lim_{A\to\infty} \lim_{t\to+\infty} \int_{\Omega} \indic_{\{|x|<t-A\}} |\nabla_{t,x}V_L(t,x)|^2\,dx=0. 
\end{equation} 
(this follows directly from the asymptotics of $V_L(t)$ as $t\to\infty$, see \cite[Lemma 2.3]{DuyckaertsYang21}).
Note that $t_0$ depends on $\delta$. However (as will be seen in the arguments below), the smallness condition on $\delta$ is independent on $t_0$.

By \eqref{P31}, \eqref{P32} and Proposition \ref{P:CVdispersive_norm}, denoting $H=U-Q$, we have 
\begin{gather}
 \label{P34}
 \forall T\geq t_0,\quad \left\|S_Q(\cdot)\PPP_c\vvec{H}(T)\right\|_{\mathsf{S}([0,\infty)} +\sum_{j,\pm} \left|\omega\left( \YYY^{\pm}_j,H(T) \right)\right|\lesssim \delta^{6}\\
\label{P35}
 \left\|H\right\|_{\mathsf{S}([t_0,\infty)}\lesssim \delta^{6}.
 \end{gather}
 The following no-return result plays a crucial role in the construction of the global center-stable manifold:
\begin{proposition}
 \label{P:P40}
 Let $U$, $\delta$ and $t_0$ be as above. Let $u$ be a solution of \eqref{eq:NLW} such that $t_0<T_+(u)$ and
 \begin{equation}
  \label{P40} 
  \left\|\vvec{u}(t_0)-\vvec{U}(t_0)\right\|_{\HHH}<\delta^{6}.
 \end{equation} 
 If 
 \begin{equation}
  \label{P41}
  \sup_{t\in [t_0,T_+(u))} \left\|\vvec{u}(t)-\vvec{U}(t)\right\|_{\HHH}\leq \delta^2,
 \end{equation} 
 then $\vec{u}(t_0)\in \mathcal{N}_{\delta}$ (and in particular $(u_0,u_1)\in \mathcal{M}$). If on the contrary
 \begin{equation}
  \label{P42}
  \sup_{t\in [t_0,T_+(u))} \left\|\vvec{u}(t)-\vvec{U}(t)\right\|_{\HHH}>\delta^2,
 \end{equation} 
 then $(u_0,u_1)\notin \mathcal{M}$. More precisely, $T_+(u)<\infty$, or $T_+(u)=+\infty$ and $u$ scatters to one of the stationary solutions $\pm Q_j$, $j\in \{0,\ldots, k-1\}$ or to $0$. 
\end{proposition}
\begin{proof}
 We first assume that $u$ satisfies \eqref{P41}. This implies that $u$ is bounded in $\HHH$ on $[t_0,T_+(u))$, and thus that $T_+(u)=\infty$. By \eqref{P41} and \eqref{P31},
 \begin{equation}
  \label{P43}
  \sup_{t\geq t_0} \left\|\vvec{u}(t)-(Q,0)-\vvec{V_L}(t)\right\|_{\HHH}\lesssim \delta^2.
 \end{equation} 
 Combining with \eqref{P32}, we see that Proposition \ref{P:CVdispersive_norm} implies, denoting $h=u-Q$,
 \begin{gather*}
  h\in \Ssf([t_0,\infty)),\quad \left\| h\right\|_{\Ssf([t_0,\infty))} \lesssim \delta^2\\
  \left\|S_Q(\cdot)\PPP_c\left(\vec{h}(t_0)\right)\right\|_{\Ssf([0,\infty))} +\sum_{j,\pm}\left|\omega\left( \YYY^j_{\pm},\vvec{h}(t_0) \right)\right|\lesssim \delta^2.
 \end{gather*} 
 Assuming that $\delta$ is small enough, we deduce that $u(t_0)\in \mathcal{N}_{\delta}$.
 
 Next, we assume \eqref{P42} and prove that $(u_0,u_1)\notin \mathcal{M}$. The proof is divided into $3$ steps. 
 
 \smallskip
 
 \noindent\emph{Step 1.} We let $t_1> t_0$ such that
 \begin{equation}
  \label{P50}
  \left\|\vvec{u}(t_1)-\vvec{U}(t_1)\right\|_{\HHH}=\delta^2,\quad \forall t\in [t_0,t_1),\; \left\|\vvec{u}(t)-\vvec{U}(t)\right\|_{\HHH}<\delta^2.
 \end{equation} 
 We let $g(t)=u(t)-Q-V_L(t)$. Note that by \eqref{P31}, \eqref{P40} and \eqref{P50},
 \begin{equation}
  \label{P52}
  \|\vvec{g}(t_0)\|_{\HHH}\lesssim \delta^6,\; \|\vvec{g}(t_1)\|_{\HHH}=\delta^2+\mathcal{O}(\delta^6),\quad \forall t\in [t_0,t_1],\; \|\vvec{g}(t)\|_{\HHH}\lesssim \delta^2.
 \end{equation} 
 We claim
 \begin{equation}
  \label{P51}
  \left\|\PPP^+\vec{g}(t_1)\right\|_{\HHH}=\delta^2+\OOO(\delta^4),\quad
 \left\|({\rm Id}-\PPP^+)\vec{g}(t_1)\right\|_{\HHH}\lesssim \delta^4.
 \end{equation} 
Indeed, $g$ satisfies the equation 
\begin{equation}
 \label{P53}
 \partial_t^2g+L_Qg=\widetilde{R}_Q(g,V_L),
\end{equation} 
(see \eqref{dis13}), where $\widetilde{R}_Q$ satisfies \eqref{dis20}, which implies, 
 for $t_0<t<t_1$,
\begin{multline}
 \label{P53a}
 \left\|\widetilde{R}_Q(g,V_L)\right\|_{\Nsf([t_0,t])} \lesssim \|V_L\|_{\Ssf([t_0,t])}+\|V_L\|_{\Ssf([t_0,t])}^{2m+1}+\|g\|_{\Ssf([t_0,t])}^2+\|g\|_{\Ssf([t_0,t])}^{2m+1}\\
 \lesssim \delta^6+\|g\|^2_{\Ssf([t_0,t])}+\|g\|^{2m+1}_{\Ssf([t_0,t])},
\end{multline}
where we have used \eqref{P32} to obtain the last inequality.
Combining with Proposition \ref{P:mixed}, \eqref{P52}, and the equation \eqref{P53}, we obtain
\begin{equation}
 \label{P55} 
 \forall t\in [t_0,t_1],\quad \left\|\PPP_c\vec{g}(t)\right\|_{\HHH} +\left\|P_cg\right\|_{\Ssf([t_0,t])}\lesssim \delta^6+\left\|g\right\|^2_{\Ssf([t_0,t])}+\|g\|^{2m+1}_{\Ssf([t_0,t])}.
\end{equation} 
Moreover, projecting the equation \eqref{P53}, we obtain, for $j\in \{0,\ldots,k\}$, $t\in (t_0,t_1]$,
\begin{multline}
 \label{P56}
 \left| \omega\left(\vec{g}(t),\YYY^+_j\right)-e^{(t_0-t)e_j}\omega\left(\vec{g}(t_0),\YYY^+_j\right)\right|\\
 \lesssim \int_{t_0}^{t} e^{(s-t)e_j} \left|\omega\left(\YYY^+_j,\widetilde{\RRR}_Q(g,V_L)(s)\right)\right|ds=:A_3(t).
\end{multline}
By \eqref{dis20}, \eqref{P32}, \eqref{P52} and the radial Sobolev inequality, recalling that $\widetilde{\mathcal{R}}_Q(g,V_L)=\begin{pmatrix}
	0\\
	\widetilde{R}_Q(g,V_L)
	\end{pmatrix}$, we have
\begin{equation}
 \label{P53b}
 \forall t\in [t_0,t_1], \; \forall r>1,\quad \quad
\left|\widetilde{R}_Q(g,V_L)(t)\right|\lesssim \delta^4.
\end{equation}
By \eqref{P52}, \eqref{P56} and \eqref{P53b},
\begin{equation}
 \label{P56'}
 \left| \omega\left(\vec{g}(t),\YYY^+_j\right)\right|\lesssim \delta^4.
\end{equation}
As a consequence of \eqref{P56} we also have that for all $\tau\in [t_0,t_1]$,
 \begin{multline}
 \label{P57}
 \left\| \omega\left(\vec{g},\YYY^+_j\right)\right\|_{L^2\cap L^{2m+1}([t_0,\tau])} \lesssim
 \left\| e^{(t_0-\cdot)e_j}\omega\left(\vec{g}(t_0),\YYY^+_j\right)\right\|_{L^2\cap L^{2m+1}([t_0,\tau])}
 \\+\|A_3\|_{L^2\cap L^{2m+1}([t_0,\tau])}.
 \end{multline}
We have
$$ \left\| e^{(t_0-\cdot)e_j}\omega\left(\vec{g}(t_0),\YYY^+_j\right)\right\|_{L^2\cap L^{2m+1}([t_0,\tau])}\lesssim |\omega(\vec{g}(t_0))|\lesssim \delta^6.$$
Furthermore, interpreting $A_3(t)$ as a convolution, as in the proof of \eqref{dis41}, and using \eqref{dis20} and $\|V_L\|_{\Ssf([t_0,t_1])}\lesssim \delta^6$ by \eqref{P32}, we obtain
$$ \|A_3\|_{L^2\cap L^{2m+1}([t_0,\tau])}\lesssim \delta^6+\|g\|^{2}_{\Ssf([t_0,\tau])}+\|g\|^{2m+1}_{\Ssf([t_0,\tau])}.$$
Thus by \eqref{P57}
\begin{equation}
\label{P58}
\forall \tau\in [t_0,t_1],\quad
 \left\| \omega\left(\vec{g},\YYY^+_j\right)\right\|_{L^2\cap L^{2m+1}([t_0,\tau])} \lesssim \delta^6+\|g\|^{2}_{\Ssf([t_0,\tau])}+\|g\|^{2m+1}_{\Ssf([t_0,\tau])}.
\end{equation}
To estimate $\omega\left(\vec{g},\YYY^-_j\right)$, we fix $\tau\in (t_0,t_1]$. Projecting the equation \eqref{P53}, we obtain, for $j\in \{0,\ldots,k\}$, $t\in (t_0,\tau]$,
\begin{multline}
 \label{P59}
 \left| \omega\left(\vec{g}(t),\YYY^-_j\right)-e^{(t-\tau)e_j}\omega\left(\vec{g}(\tau),\YYY^-_j\right)\right|\\
 \lesssim \int_{t}^{\tau} e^{(t-s)e_j} \left|\omega\left(\YYY^-_j,\widetilde{\RRR}_Q(g,V_L)(s)\right)\right|ds=:A_4(t).
\end{multline}
With the same argument as above, using also the bound of $g(t)$ in \eqref{P52}, we see that
\begin{equation}
\label{P58'}
\forall \tau\in [t_0,t_1],\quad
 \left\| \omega\left(\vec{g},\YYY^-_j\right)\right\|_{L^2\cap L^{2m+1}([t_0,\tau])} \lesssim \delta^2+\|g\|^{2}_{\Ssf([t_0,\tau])}+\|g\|^{2m+1}_{\Ssf([t_0,\tau])}.
\end{equation}
Combining \eqref{P55}, \eqref{P58} and \eqref{P58'}, we deduce
\begin{equation*}
 \forall \tau\in [t_0,t_1],\quad \left\|g\right\|_{\Ssf([t_0,\tau])}\lesssim \delta^2+\left\|g\right\|^2_{\Ssf([t_0,\tau])}+\|g\|^{2m+1}_{\Ssf([t_0,\tau])},
\end{equation*}
and thus by a straightforward continuity argument $\left\|g\right\|_{\Ssf([t_0,t_1])}\lesssim \delta^2$. In view of \eqref{P55}, we see that
$$\forall t\in[t_0,t_1],\quad \left\|\mathcal{P}_c\vec{g}(t)\right\|_{\HHH}+\left\|\mathcal{P}_cg\right\|_{\Ssf([t_0,t])}\lesssim \delta^4.$$
Combining with \eqref{P56'}, we obtain
\begin{equation*}
 \left\|({\rm Id}-\PPP^+)\vvec{g}(t_1)\right\|_{\HHH}\lesssim \delta^4.
\end{equation*} 
Using the equality $\|\vvec{g}(t_1)\|_{\HHH}=\delta^2+\OOO(\delta^6)$ of \eqref{P52} we obtain \eqref{P51}.

Next, we recall that it follows from the main result of \cite{DuyckaertsYang21}, recalled in Theorem \ref{T:DY19} that if $(v_0,v_1)\in \pm \MMM_j^{+}$ and $v$ is the corresponding solution of \eqref{eq:NLW}, then
\begin{equation}
 \label{P70}
 \lim_{A\to\infty} \lim_{t\to\infty} \frac 12\int_{|x|\geq t-A} |\nabla_{t,x}v(t,x)|^2dx=E(v_0,v_1)-E(Q_j,0),
\end{equation} 
where by definition $|\nabla_{t,x}v|^2=(\partial_tv)^2+|\nabla_x v|^2$. Indeed, \eqref{P70} is a direct consequence of \eqref{asymptotic_u}, \eqref{eq:energy}, and the property \eqref{asymptotic_VL} satisfied by any solution of the linear wave equation on $\Omega$ with Dirichlet boundary conditions.

Recall that the negative eigenvalues of $L_Q$ are denoted by $-e_0^2<-e_1^2<\ldots <-e_k^2$ and that by radiality, each of these eigenvalues has multiplicity 
$1$. Assuming that $T_+(u)=\infty$ (if not, we are done), we will prove in the two next steps that there exists a constant $C>0$ such that
\begin{equation}
 \label{P70'} 
 \lim_{t\to\infty} \frac{1}{2}\int_{|x|\geq t-t_1+\frac{1}{2e_0}|\log \delta|} |\nabla_{t,x}u(t,x)|^2\geq E(u_0,u_1)-E(Q,0)+\frac{1}{C}\delta^5,
\end{equation} 
(where as before $Q=Q_k$)
which shows, since $E(Q_k,0)<E(Q_j,0)$ for $j>k$, that $(u_0,u_1) \in \pm\MMM_j^{\pm}$ for some $j<k$ or that $(u_0,u_1)$ scatters to $0$. 

\smallskip

\noindent\emph{Step 2.} In this step we prove that there exists $t_2\geq t_1$ such that 
\begin{equation}
 \label{P71}
 \forall t\geq t_2,\quad \frac{1}{2}\int_{|x|\geq t-t_0+\delta^{-3}} |\nabla_{t,x}u(t,x)|^2dx\geq E(u_0,u_1)-E(Q,0)-C\delta^6.
\end{equation} 
By the assumption \eqref{P40} on $u$, we have $\left\|\vvec{u}(t_0)-\vvec{U}(t_0)\right\|_{\HHH}<\delta^6$, so that
\begin{equation}
 \label{P72}
 \left|E(u_0,u_1)-E(U_0,U_1)\right|\lesssim \delta^6,
\end{equation} 
(where the implicit constant depends only on $U$). 
Furthermore, for $t\geq t_0$, $x\in \Omega$,
$$|\partial_t^2(u-U)-\Delta(u-U)|\lesssim |u-U|^{2m+1}+|u-U|U^{2m}.$$
Let $\Gamma_{\delta}=\left\{(t,x),\; t\geq t_0,\; |x|\geq t-t_0+\delta^{-3}\right\}$. We have
\begin{multline*}
 \left\|\indic_{\Gamma_{\delta}}U\right\|_{L^{2m+1}(\R,L^{2(2m+1)}(\Omega))} \\
 \lesssim \left\|\indic_{\Gamma_{\delta}}Q\right\|_{L^{2m+1}(\R,L^{2(2m+1)}(\Omega))}+\left\|\indic_{\Gamma_{\delta}}V_L\right\|_{L^{2m+1}(\R,L^{2(2m+1)}(\Omega))}+\delta^6\lesssim \delta^{5/2},
\end{multline*}
where we have used \eqref{P31}, \eqref{P32}, and the fact that $Q$ is of order $1/r$ at infinity. Combining with \eqref{P40}, finite speed of propagation, the dispersive estimates of Proposition \ref{P:mixed} (see Remark \ref{R:mixed_free}), and a standard bootstrap argument using the  Gronwall type inequality of \cite[Lemma 8.1]{FaXiCa11}, we obtain
\begin{equation}
\label{P80}
 \forall t\geq t_0,\quad \sqrt{\int_{|x|\geq t-t_0+\delta^{-3}} |\nabla_{t,x}(u(t,x)-U(t,x))|^2}\lesssim \delta^6.
\end{equation} 

By \eqref{P31}, for $t\geq t_0$,
$$ \frac{1}{2}\int_{|x|\geq t-t_0+\delta^{-3}} |\nabla_{t,x}U(t,x)|^2dx\geq \frac{1}{2}\int_{|x|\geq t-t_0+\delta^{-3}} |\nabla_{t,x}(Q+V_L)(t,x)|^2dx-C\delta^6.$$
By dominated convergence, then \eqref{P33}, we deduce\begin{multline*}
\liminf_{t\to\infty} \frac{1}{2}\int_{|x|\geq t-t_0+\delta^{-3}} |\nabla_{t,x}U(t,x)|^2dx
\\
\geq \liminf_{t\to\infty} \frac{1}{2}\int_{|x|\geq t-t_0+\delta^{-3}} |\nabla_{t,x}V_L(t,x)|^2dx-C\delta^6\\
\geq \liminf_{t\to\infty} \frac{1}{2}\int_{\Omega} |\nabla_{t,x}V_L(t,x)|^2dx-C\delta^6
=E(U_0,U_1)-E(Q,0)-C\delta^{6}.
\end{multline*}
Combining with \eqref{P80}, we obtain \eqref{P71}. 

\medskip

\noindent\emph{Step 3.} We prove that there exists a constant $C>0$ (that might depend on $U$ but is independent of $u$) such that for large $t$,
\begin{equation}
 \label{P89}
 \int_{t-t_1+|\log\delta|/2e_0\leq |x|\leq t-t_1+10|\log\delta|/e_k}|\partial_tu(t)|^2+\left|\nabla(u(t)-Q)\right|^2dx\geq \frac{1}{C}\delta^5.
\end{equation} 
Note that since $10|\log \delta|/e_j\leq \delta^{-3}$,  for small $\delta>0$, and $t_1>t_0$, the inequalities \eqref{P71} and \eqref{P89} imply \eqref{P70'} for large $t$, since $\lim_{t\to\infty} \int_{|x|>t-t_1+|\log \delta|/2e_0} |\nabla Q|^2dx=0$, concluding the proof. 
As in Step 1, we denote by $g(t)=u(t)-Q-V_L(t)$. Then $g$ satisfies
\begin{equation}
 \label{P88}
 \partial_t^2g-\Delta g=\left( g+Q+V_L \right)^{2m+1}-Q^{2m+1}:=\overline{R}_Q(V_L,g).
\end{equation} 
Using finite speed of propagation, we consider this equation on the set 
$$ \widetilde{\Gamma}_{\delta}=\left\{(t,r),\; t\geq t_1,\; r\geq \frac{|\log \delta|}{2e_0}+t-t_1\right\}.$$
Since 
$$ \left|\overline{R}_Q(V_L,g)\right|\lesssim |g|^{2m+1}+|V_L|^{2m+1}+Q^{2m}(|g|+|V_L|),$$
we have 
\begin{multline*}
\forall t\geq t_1,\quad \left\|\indic_{\widetilde{\Gamma}_{\delta}} \overline{R}_Q(V_L,g)\right\|_{\Nsf\left( [t_1,t] \right)} \lesssim\left\|\indic_{\widetilde{\Gamma}_{\delta}}g\right\|^{2m+1}_{\Ssf([t_1,t])}+\left\|\indic_{\widetilde{\Gamma}_{\delta}}V_L\right\|^{2m+1}_{\Ssf([t_1,t])}\\
+
\left\|Q\indic_{\widetilde{\Gamma}_{\delta}}\right\|^{2m}_{\Ssf([t_1,t])} \left( \left\|\indic_{\widetilde{\Gamma}_{\delta}}g\right\|_{\Ssf([t_1,t])} + \left\|\indic_{\widetilde{\Gamma}_{\delta}}V_L\right\|_{\Ssf([t_1,t])} \right)
.
\end{multline*} 
Using the mixed Strichartz/weighted energy estimates of Proposition \ref{P:mixed} (see Remark \ref{R:mixed_free}), we deduce, since $\|g(t_1)\|_{\HHH}\approx \delta^2$ by \eqref{P52}, and $\|V_L\|_{\Ssf([t_1,\infty)}\lesssim \delta^6$ by \eqref{P32},
\begin{equation}
 \label{P91}
 \forall t\geq t_1,\quad \left\| \indic_{\widetilde{\Gamma}_{\delta}} g\right\|_{\Ssf([t_1,t])} \lesssim \delta^2+o_{\delta}(1)\left\| \indic_{\widetilde{\Gamma}_{\delta}} g\right\|_{\Ssf([t_1,t])}+\left\| \indic_{\widetilde{\Gamma}_{\delta}} g\right\|^{2m+1}_{\Ssf([t_1,t])}.
\end{equation} 
By a straightforward continuity argument,
\begin{equation}
 \label{P100}
 \left\| \indic_{\widetilde{\Gamma}_{\delta}} g\right\|_{\Ssf([t_1,\infty))}\lesssim \delta^2.
\end{equation} 
Let $g_L$ be the solution of 
\begin{equation}
 \label{P101}
 \partial_t^2g_L+L_Qg_L=0,\quad \vec{g}_L(t_1)=\vec{g}(t_1).
\end{equation} 
Then
$$ (\partial_t^2-\Delta)(g-g_L)=(2m+1)Q^{2m}(g-g_L)+\widetilde{R}_Q(g,V_L),$$
where $\widetilde{R}_Q(g,V_L)$ is defined as in \eqref{dis13}. By \eqref{dis20} we have, for $t\geq t_1$,
\begin{multline*}
\left\|\widetilde{R}_Q(g,V_L)\indic_{\widetilde{\Gamma}_{\delta}}\right\|_{\mathsf{N}(t_1,t)}\lesssim \left\|V_L\right\|_{\mathsf{S}([t_1,t])}\\
+\left\|V_L\right\|_{\mathsf{S}([t_1,t])}^{2m+1} +\left\|g\indic_{\widetilde{\Gamma}_{\delta}}\right\|^2_{\mathsf{S}([t_1,t])}+\left\|g\indic_{\widetilde{\Gamma}_{\delta}}\right\|^{2m+1}_{\mathsf{S}([t_1,t])}
\lesssim \delta^4, 
\end{multline*}
where we have used \eqref{P32} and \eqref{P100}.
Also
$$ \left\|Q^{2m}(g-g_L)\indic_{\widetilde{\Gamma}_{\delta}}\right\|_{\mathsf{N}([t_1,t])} \lesssim o_{\delta}(1)\left\|g-g_L\right\|_{\mathsf{S}([t_1,t])}.$$
Hence
\begin{equation}
 \label{P102}
 \left\|\indic_{\widetilde{\Gamma}_{\delta}} (g-g_L)\right\|_{\Ssf([t_1,\infty))} +\left(\sup_{t\geq t_1} \int_{|x|\geq \frac{|\log \delta|}{2e_0}+t-t_1}|\nabla_{t,x}(g-g_L)(t,x)|^2dx\right)^{\frac 12}\lesssim \delta^4.
\end{equation} 
Furthermore, since by \eqref{P51}
$$\left\|(\PPP^-+\PPP_c)\vec{g}_L(t_1)\right\|_{\HHH} \lesssim \delta^4,$$
we have
\begin{equation}
 \label{P103} \sup_{t\geq t_1} \left\|(\PPP^-+\PPP_c)\vec{g}_L(t)\right\|_{\HHH}\lesssim \delta^4.
\end{equation} 
Next, we use the following claim, whose proof is postponed to the end of this subsection.
\begin{claim}
 \label{claim:growth_unstable}
 There exists $C>0$, $R_0>0$ (depending only on $k$) such that for all $\vvec{h_0}\in \PPP^+\HHH=\HHH^u$, if $h_L(t)=S_Q(t)\vvec{h_0}$, 
 \begin{multline*}
R\geq R_0\Longrightarrow \forall t\geq 0,\quad \frac{1}{C} e^{-2e_0R}\left\|\vvec{h_0}\right\|^2_{\HHH}\leq \int_{|x|>R+|t|}|\nabla_{t,x}h_L(t,x)|^2dx\\
 \leq Ce^{-2e_k R}\left\|\vec{h_0}\right\|^2_{\HHH}.
 \end{multline*}
\end{claim}
By the claim and \eqref{P51}, denoting $g_L^+(t)=S_Q(t-t_1)\PPP^+\vvec{g}(t_1)$, so that $\PPP^+\vvec{g_L}=(g_L^+,\partial_tg_L^+)$,
\begin{equation}
 \label{P104}
 \int_{|x|\geq \frac{|\log \delta|}{2e_0}+(t-t_1)} \left|\nabla_{t,x}g_L^+(t)\right|^2dx\gtrsim e^{-\frac{|\log \delta| e_0}{e_0}}\left\|\PPP^+\vvec{g}(t_1)\right\|^2_{\HHH}\gtrsim \delta^5,
\end{equation}
and
\begin{equation}
 \label{P105}
 \int_{|x|\geq \frac{10|\log \delta|}{e_k}+(t-t_1)} \left|\nabla_{t,x}g_L^+(t)\right|^2dx\lesssim e^{-10\frac{|\log \delta| e_k|}{e_k}}\left\|\PPP^+\vvec{g}(t_1)\right\|^2_{\HHH}\lesssim \delta^{14},
\end{equation}
Combining \eqref{P33}, \eqref{P102}, \eqref{P103}, \eqref{P104} and \eqref{P105}, we obtain the desired conclusion \eqref{P89}. We are left with proving the claim.
\end{proof}
\begin{proof}[Proof of Claim \ref{claim:growth_unstable}]
 We write
 $$ \vvec{h_0}=\sum_{j=0}^k \alpha_j\YYY_j^+,\quad \alpha_j=\omega\left(\YYY^-_j,\vvec{h_0}  \right).$$
 so that 
 $$\vvec{h}(t)=\sum_{j=0}^k \alpha_je^{e_jt}\YYY_j^+,\quad \sum_{j=0}^k \left( \alpha_j \right)^2\approx \left\|\vvec{h_0}\right\|^2_{\HHH}.$$
 Recall that $\YYY_j^+=(Y_j,e_jY_j)$, and that, by Lemma \ref{kdlre}, there exists $c_j \in \R\setminus \{0\}$ such that 
 $$Y_j(r)=\left(\frac{c_j}{r}+\OOO(r^{-3/2})\right)e^{-e_jr},\quad \partial_r Y_j(r)=\left(\frac{-c_je_j}{r}+\OOO(r^{-3/2})\right)e^{-e_jr},\quad r\to\infty.$$
 Thus, assuming that $R$ is large enough,
 \begin{multline*}
\int_{|x|>R+t}\left|\nabla_{t,x}h_L(t,x)\right|^2dx
 \approx \int_{R+t}^{\infty} \bigg|\sum_{j=0}^k c_je^{-e_jr+e_jt}\alpha_j\bigg|^2dr\\
=\int_R^{\infty} \bigg|\sum_{j=0}^k \alpha_jc_j e^{-e_jr}\bigg|^2dr=e^{-2e_0R} \int_{0}^{\infty}\bigg|\sum_{j=0}^k\alpha_j c_j e^{(e_0-e_j)R}e^{-e_jr}\bigg|^2dr.
 \end{multline*}
 Using that all norms of a finite-dimensional vector space are equivalent, and the fact that the functions $e^{-e_0r},\ldots,e^{-e_jr}$ are linearly independent, we obtain that there exists a constant  $C>0$ (depending only on $k$) such that 
 $$ \forall \left(\beta_j\right)_{j}\in \R^{k+1},\quad \int_0^{\infty} \left|\sum_{j=0}^k \beta_je^{-e_jr}\right|^2dr\geq \frac{1}{C}\sum_{j=0}^k\beta_j^2.$$
 Hence
 $$\int_{R+t}^{\infty}\left|\nabla_{t,x}h_L(t,x)\right|^2dx\geq \frac{1}{C}e^{-2e_0 R}\sum_{j=0}^{k} \alpha_j^2e^{2(e_0-e_j)R}\geq \frac{1}{C}e^{-2e_0R} \sum_{j=0}^k \alpha_j^2,$$
 which yields the claimed lower bound. The proof of the upper bound is similar and we omit it.
\end{proof}
We are now ready to conclude the proof of Theorem \ref{T:cs}. Recall that we have taken $(U_0,U_1)\in \MMM_k^+$, that we have denoted by $U$ the corresponding solution of \eqref{eq:NLW}, and that $t_0$ has been chosen so that \eqref{P31}, \eqref{P32} and \eqref{P33} hold. 

By Proposition \ref{P:CVdispersive_norm}, $\vertiii{(Id-\PPP^+)\left(\vvec{U}(t_0)-(Q,0)  \right)}\lesssim \delta^6$ and $\|U-Q\|_{\mathsf{S}([t_0,\infty))}\lesssim\delta^6$. Thus $\vvec{U}(t_0)\in \NNN_{\delta}$, which is a smooth submanifold of $\HHH$ of codimension $k+1$. Consider the map $\Phi: (u_0,u_1)\mapsto \vec{u}(t_0)$, where $u$ is the solution of \eqref{eq:NLW} with initial data $(u_0,u_1)$. By Proposition \ref{P:smooth_flow}, $\Phi$ is a smooth, Lipschitz map from a neighborhood $\omega$ if $(U_0,U_1)$ to $\Phi(\omega)$, and $\Phi'$ is bounded on $\omega$. Since \eqref{eq:NLW} is time reversible, $\Phi$ is a smooth diffeomorphism from $\omega$ to $\Phi(\omega)$. Taking a smaller neighborhood $\omega$ of $(U_0,U_1)$ if necessary, we see that $\Phi^{-1}\left( \NNN_{\delta}\cap \Phi(\omega) \right)$ is a $(k+1)$-codimensional submanifold of $\HHH$, which contains $(U_0,U_1)$ and is included in $\MMM_k^+\cap \omega$. 

It remains to prove that there exists a neighborhood $\tilde{\omega}\subset \omega$ of $(U_0,U_1)$ in $\HHH$ such that $\tilde{\omega}\cap \Phi^{-1}\left( \NNN_{\delta}\cap \Phi(\omega) \right)=\Phi^{-1}\left( \NNN_{\delta}\cap \Phi(\tilde{\omega}) \right)=\MMM_k^+\cap \tilde{\omega}$. We argue by contradiction, assuming that there exists a sequence $\Big((U_0^n,U_1^n)\Big)_n$ of elements of $\HHH$ that converges to $(U_0,U_1)$ as $n$ goes to infinity, and such that $(U_0^n,U_1^n)\in \MMM_k^+$ and $\Phi(U_0^n,U_1^n)\notin \NNN_{\delta} \cap \Phi(\omega)$. Since $\vvec{U}(t_0)=\Phi(U_0,U_1)$ satisfies \eqref{P31}, \eqref{P32} and \eqref{P33} the no-return property of Proposition \ref{P:P40} implies that for large $n$, $\Phi(U_0^n,U_1^n)\notin \MMM_k^+$, and thus that $(U_0^n,U_1^n)\notin \MMM_k^+$, a contradiction. The proof is complete. \qed

We conclude this section by the proof of Proposition \ref{P:intersection}. We first note that for small $\delta$,
\begin{equation}
\label{equal_manifold}
\MMM^+_k\cap B(Q_k,\delta^6)=\NNN_{\delta}\cap B(Q_k,\delta^6),
\end{equation}
where $B((Q_k,0),\delta^6)$ is the open ball of $\HHH$ centered at $(Q,0)$
 with radius $\delta^6$. Indeed, by definition of $\NNN_{\delta}$, $\NNN_{\delta}\subset \MMM^+$. On the other hand, if $(u_0,u_1)\in \MMM^+\cap B(Q_k,\delta^6)$, then by the no-return result (Proposition \ref{P:P40}), with $U=Q$ and $t_0=0$ (and thus $V_L=0$), we have $(u_0,u_1)\in \NNN_{\delta}$, and thus \eqref{equal_manifold} holds.

By Remark \ref{R:tangent_space}, this implies that the tangent space of $\MMM^+_k$ is $\HHH^{cs}$. Also, considering a solution $u$ of \eqref{eq:NLW} with initial data $(u_0,u_1)$, and denoting by $\sigma(u_0,u_1)=(u_0,-u_1)$, we see that the solution of \eqref{eq:NLW} with initial data $\sigma(u_0,u_1)$ is $u^{\sigma}(t)=u(-t)$. This proves that $\MMM^-_k=\sigma(\MMM_k^+)$, and that the tangent space to $\MMM^-_k$ at $(Q_k,0)$ is exactly $\sigma(\HHH^{cs})=\HHH^{cu}$.

\section{Instability}
\label{sec:instability}

\subsection{Positivity of the linear and nonlinear flows}
In this subsection we prove positivity properties for the radial linear and nonlinear wave equations outside the ball. These properties are essentially the same than the one for the linear wave equation on $\R^3$. In all the section, $\Omega=\R^3\setminus B$, and we consider only radial solutions.
\begin{proposition}
 \label{P:positive}
 Let $f\in L^1\left((0,T),L^2(\Omega)\right)$, $(u_0,u_1)\in \HHH$ and $\vec{u}\in C^0([0,T],\HHH)$ the solution of 
 \begin{equation}
 \left\{
  \begin{aligned}
   (\partial_t^2-\Delta_D)u&=f,\quad (t,x)\in [0,T]\times \Omega\\
   \vec{u}_{\restriction t=0}&=(u_0,u_1),\quad y_{\restriction \partial \Omega}=0.
  \end{aligned}\right.
 \end{equation}
 Assume
 \begin{gather}
 \label{P_hyp1}
  f(t,r)\geq 0 \text{ for a.a. }t\in (0,T),\; r>1\\
  \label{P_hyp2}
  u_0(r)\geq 0\text{ for all }r>1\\
  \label{P_hyp3}
  ru_1+\partial_r(ru_0)\geq 0\text{ for a.a. } r>1.
 \end{gather}
Then
\begin{gather}
 \label{P_conclu1}
 u(t,r)\geq 0, \quad \forall t\in [0,T],\; \forall r>1\\
 \label{P_conclu2}
 (\partial_t+\partial_r)(ru)\geq 0,\;\forall t\in [0,T],\;\text{ for a.a. }r>1, 
\end{gather} 
\end{proposition}
\begin{proof}
%
Since 
$$(\partial_t-\partial_r)(\partial_t+\partial_r)(ru)=rf\geq 0\text{ a.e. },$$
(here and in the sequel a.e. means for almost all $t\in (0,T)$, $r>1$) we obtain, using \eqref{P_hyp3}
$$(\partial_t+\partial_r)(ru)\geq 0\text{ a.e.}$$
Using that $u_0(r)\geq 0$ for all $r>1$ and that $u(t,1)=0$ for all $t\in [0,T]$, we deduce $u(t,r)\geq 0$ for all $t\in [0,T]$ and all $r>1$. This concludes the proof.
\end{proof}
\begin{corollary}
\label{CR:positive}
 Let $u$, $v$ be two solutions of \eqref{eq:NLW} with initial data $(u_0,u_1)$, $(v_0,v_1)$ respectively. Let $T>0$ such that $T<T_+(u)$ and $T<T_+(v)$. Assume
 \begin{gather}
 \label{C_positive1}
  u_0(r)-v_0(r)\geq 0\text{ for all }r>1\\
 \label{C_positive2}
 r(u_1-v_1)+\partial_r(ru_0-rv_0)\geq 0\text{ for a.a. }r>1.
 \end{gather}
 Then for all $t\in [0,T]$, 
\begin{gather*}
  u(t,r)-v(t,r)\geq 0\text{ for all }r>1\\
  (\partial_t+\partial_r)(r u-rv)\geq 0\text{ for a.a. }r>1.
 \end{gather*}
\end{corollary}
\begin{proof}
It is sufficient to prove the result for small $T>0$. An iteration procedure yields the conclusion for general $T$.

If $T$ is small enough, the local well-posedness theory proves that $u$ is the limit, in $L^{2m+1}([0,T],L^{2(2m+1)})$ of the sequence $U_n$ defined by $U_0(t,r)=0$ and
\begin{equation}
 \left\{ 
 \begin{aligned}
 \partial_t^2U_{n+1}-\Delta_D U_{n+1}&=U_n^{2m+1}\\
 \vec{U}_{n+1}&=(u_0,u_1).
\end{aligned}
 \right.
 \end{equation}
Using Strichartz estimates, one sees that $\vvec{u}$ is the limit of $\vvec{U_n}$ in $C^0([0,T],\HHH)$. 
Defining similarly the sequence $V_n$ for the solution $v$, we deduce by induction, from Proposition \ref{P:positive} that 
$$ U_n(t,r)\geq V_n(t,r)\text{ and } (\partial_t+\partial_r)(rU_n-rV_n)\geq 0$$
for almost all $(t,r)\in [0,T]\times (1,\infty)$. Passing to the limit we obtain the desired inequalities.
\end{proof}
\begin{corollary}
\label{CR:Y0}
 Let $u$ be a solution of \eqref{eq:NLW} with initial data $(u_0,u_1)$. Let $k\geq 0$, $-e_0^2<\ldots <-e_{k}^2$, ($e_j>0$) be the $k+1$ eigenvalues of the linearized operator $L_k=-\Delta  -(2m+1)Q_k^{2m}$. Let $Y_0,\ldots Y_k$ be the corresponding eigenvectors, which we choose so that $\|Y_j\|_{L^2}=1$ for all $j$, and there exists $R>1$ (depending only on $k$) such that 
 $$\forall j\in \{0,\ldots, k\},\quad \forall r\geq R,\quad Y_j(r)>0.$$
 Denote by $\mathcal{Y}_j^+=(Y_j,e_jY_j)$. Let $(\omega_0,\ldots,\omega_k)\in \R^{k+1}$ such that
 \begin{equation}
 \label{assump_positivity}
 \forall j\in \{1,\ldots,k\}, \quad \omega_j\geq 0,\quad \text{and}\quad  C_k\sum_{j=1}^k \omega_j< \omega_0,
 \end{equation} 
 for some large constant $C_k$.
 Let
 $(\phi_0,\phi_1)=\sum_{j=0}^k \omega_j \YYY_j^+$, and 
 $v$ be the solution of \eqref{eq:NLW} with initial data $(u_0,u_1)+(\phi_0,\phi_1)$. Then
 $$0\leq t<\min\left(T_{\max}(u),T_{\max}(v)\right)\Longrightarrow v(t,r)\geq u(t,r).$$
\end{corollary}
\begin{remark}
 \label{R:Y0}
 Letting $\omega_1=\ldots=\omega_k=0$ in the corollary, we see that if $u^{\alpha}$ is the solution of \eqref{eq:NLW} with initial data $(u_0,u_1)+\alpha \YYY_0^+$, then 
 $$\forall t\geq 0,\quad u^{\alpha}(t)\begin{cases} 
 \geq u(t) &\text{ if }\alpha>0\\
 \leq u(t) &\text{ if }\alpha<0,
                \end{cases}
$$ 
where we have applied the corollary to the solution $-(u_0,u_1)-\alpha\YYY_0^+$ to obtain the second line.
\end{remark}

\begin{proof}
 Since $-e_0^2$ is the first eigenvalue of $L_k$, we have $Y_0(r)>0$ for $r>1$. We claim 
\begin{equation}
 \label{Sufficient_to_check}
\forall r\geq 1,\quad 
 (\partial_r+e_0)(rY_0)> 0.
\end{equation} 
 Indeed, we have 
 $$(\partial_r-e_0)(\partial_r+e_0)(rY_0)=-r(2m+1)Q^{2m}Y_0<0,\quad {\rm a.e.}$$
which we can rewrite
$$ \frac{d}{dr}\Big(e^{-e_0r}(\partial_r+e_0) (rY_0)\Big)<0, \quad {\rm a.e}.$$
Since 
$$ \lim_{r\to\infty}e^{-e_0r}(\partial_r+e_0)(rY_0)=0,$$
we deduce \eqref{Sufficient_to_check}. 
By the same proof, using that $Y_j(r)>0$ for $r>R$, we obtain
\begin{equation}
 \label{positivity_RYj}
 \forall r> R,\quad (\partial_r+e_j)(rY_j(r))>0.
\end{equation}
To use Corollary \ref{CR:positive}, we need to show that 
\begin{gather}
\label{cond.positivity1}
\forall r\geq 1,\quad \phi_0=\sum_{j=0}^k \omega_jY_j\geq 0.
\\
\label{cond.positivity2}
\forall r\geq 1,\quad \partial_r (r\phi_0)+r\phi_1=\sum_{j=0}^k \omega_j(\partial_r+e_j)(rY_j)\geq 0.
\end{gather}
By the choice of $R$, \eqref{Sufficient_to_check} and \eqref{positivity_RYj}, we see that \eqref{cond.positivity1} and \eqref{cond.positivity2} hold for $r>R$.

Using that $Y_0(r)>0$, for $r>1$, and that $Y_0'(1)>0$, we obtain that there exists $C>0$ such that for all $r\in [1,R]$, $CY_0(r)\geq \sum_{j=1}^k \left|Y_j(r)\right|$. Combining with \eqref{assump_positivity}, we obtain that \eqref{cond.positivity1} holds for $1\leq r\leq R$. 

Similarly, using \eqref{Sufficient_to_check}, we see (taking a larger $C>0$ if necessary), that
$$ C(\partial_r+e_0)(rY_0)\geq \sum_{j=1}^k |(\partial_r +e_j)(rY_j)|.$$
for $1\leq r\leq R$. Using assumption \eqref{assump_positivity} again, we obtain that \eqref{cond.positivity2} holds also for $1\leq r\leq R$, which concludes the proof.
 \end{proof}
\subsection{Inequalities between stationary solutions}
\begin{lemma}
\label{L:InegStat}
 Let $k,j$ be integers with $0\leq j\leq k$. Then
 \begin{align}
  \label{InegStat2}
  \exists a_1>1, \quad Q_k(a_1)&>-Q_j(a_1).\\
\intertext{If furthermore $(j,k)\neq (0,0)$,}
 \label{InegStat4}
 \exists a_2>1, \quad -Q_k(a_2)&>Q_j(a_2).\\
  \intertext{Finally, if $0\leq j<k$,}
 \label{InegStat1}
  \exists a_3>1, \quad Q_k(a_3)&>Q_j(a_3)\\
  \label{InegStat3}
  \exists a_4>1, \quad -Q_k(a_4)&>-Q_j(a_4).
 \end{align}
\end{lemma}
\begin{proof}
Recall that for $\ell>0$, $Q_{\ell}(r)=r_{\ell}^{\frac{1}{m}}Z(r_{\ell} r)$, where $Z$ is a radial solution of $-\Delta Z = Z^{2m+1}$ on $\R^3\setminus \{0\}$, and $(r_j)_{j\geq 0}$ is the decreasing sequence of zeros of $Z$, that satisfies $\displaystyle\lim_{j\to+\infty}r_j=0$ (see Proposition \ref{P:DKM-stationary}).

As a consequence
$$\lim_{r\to\infty} rQ_{\ell}(r)= r_{\ell}^{\frac{1}{m}-1}.$$
The properties \eqref{InegStat2} (by sign consideration) and \eqref{InegStat1} (since $r_k<r_{j})$ follow immediately, taking $a_1$ and $a_3$ large enough.

Next, we observe that $-Q_k(r_0/r_k)=-r_{k}^{\frac{1}{m}}Z(r_0)=0$, and that 
$$-Q_j(r_0/r_k)=-r_j^{\frac{1}{m}}Z(r_0 r_j/r_k)<0,$$
where we have used that $r_0 r_j/r_k>r_0$, and that $Z$ is positive on $(r_0,\infty)$. Thus \eqref{InegStat3} follows with $a_4=r_0/r_k$.

If $j=k>0$, then \eqref{InegStat4} follows since $Q_j=Q_k$ changes sign on $(1,\infty)$. 

If $0\leq j<k$, since $k\geq 1$, there exists $a,b$ with $1\leq a<b$, $Q_k(a)=Q_k(b)=0$ and for all $r$ in $(a,b)$, $Q_k(r)<0$. We prove by  Liouville theory that there exists $a_2\in (a,b)$ such that \eqref{InegStat4} holds. Indeed, if not,
\begin{equation}
\label{sign_absurd}
\forall r\in [a,b], \quad 0\leq -Q_k(r)\leq Q_j(r),
\end{equation} 
and the inequalities in \eqref{sign_absurd} are strict except for a finite set of $r$ by standard ordinary differential equations theory. Consider the Wronskian
$$ w=\frac{d}{dr} \left( rQ_k\right)\,rQ_j-\frac{d}{dr}\left( rQ_j \right)\,rQ_k.$$
Then $w(a)=a^2\frac{d}{dr}\left( Q_k \right)(a) Q_j(a)\leq 0$, $w(b)=b^2\frac{d}{dr}\left( Q_k \right)(b) Q_j(b)\geq 0$ and
$$w'(r)=-r^2 Q_kQ_j\left( Q_k^{2m}-Q_j^{2m} \right)\leq 0,$$
with a strict inequality except in a finite set of points. 
This is a contradiction.
\end{proof}

\subsection{Instability of the ground state}
In this subsection, we prove Theorem \ref{T:M0} and Proposition \ref{P:positive_blowup}. We start by two preliminary lemmas. As above, we denote by $Y_0$ a positive eigenfunction associated to the (unique) eigenvalue $-e_0^2$ of the linearized operator $L_0=-\Delta-(2m+1)Q_0^{2m}$, and $\YYY_0^+=(Y_0,e_0Y_0)$ (to lighten notation, we will forget about the normalization constants in the definitions of the eigenmodes $\YYY_j^{\pm}$, since they are irrelevant in this section). 
\begin{lemma}
 \label{L:Y0}
 Let $\alpha\in \R$ and $u^{\alpha}$ be the solution of \eqref{eq:NLW} with initial data 
 $$(Q_0,0)+\alpha \YYY_0^+.$$
 Assume $|\alpha|$ is small. Then if $\alpha>0$, $u^{\alpha}$ blows up in finite positive time (with a positive blow-up). If $\alpha<0$ and $|\alpha|$ is small, $u^{\alpha}$ scatters to $0$ in the future. 
\end{lemma}
\begin{proof}
For small $\alpha$, we have
$$E(\vec{u}^{\alpha}(0))=E(Q_0,0)+\OOO(\alpha^3),$$
where $\vvec{u_\alpha}=(u^\alpha,\partial_{t} u^\alpha)$.
By Theorem \ref{T:DY19}, using that $E(Q_0,0)<E(Q_k,0)$ for all $k\geq 1$, we obtain that for small $\alpha$
$$ \vec{u}^{\alpha}(0)\in \MMM_0^+\cup (-\MMM_0^+) \cup \SSS^+\cup \FFF^+.$$
 By Proposition \ref{P:intersection} the tangent space of $\MMM_0^+$ at $(Q_0,0)$ is the center stable space $\HHH^{cs}=\HHH^c\oplus \HHH^s$. Since $\YYY_0^+$ belongs to the unstable space $\HHH^u$, we obtain that for small $\alpha$, $\vec{u}^{\alpha}(0)\notin \MMM_0^+$. Thus
 \begin{equation}
  \label{alternative}
  \vec{u}^{\alpha}(0)\in (-\MMM_0^+) \cup \SSS^+\cup \FFF^+. 
 \end{equation} 
 
 \medskip
 
 \noindent\emph{Case 1: $\alpha>0$.}
 By Remark \ref{R:Y0}, for small positive $\alpha$, $u^{\alpha}(t)\geq Q_0$ for all $t\geq 0$ in the domain of definition of $u^{\alpha}$. By \eqref{alternative}, we deduce that $u^{\alpha}$ blows up in finite positive times, with a positive blow-up. 
 
 By Remark \ref{R:Y0},
 $$ \alpha>\beta\Longrightarrow \forall t\geq 0,\quad u^{\alpha}(t)\geq u^{\beta}(t).$$
 Thus for any $\alpha>0$ (without smallness assumption) $u^{\alpha}$ blows up in finite positive times. 
 
 \medskip
 \noindent\emph{Case 2: $\alpha<0$.}
 By Remark \ref{R:Y0}, for all $t\geq 0$ in the domain of definition of $u^{\alpha}$,
 \begin{equation}
 \label{bound_above}
  u^{\alpha}(t)\leq Q_0
 \end{equation} 
 Since $Q_0(r)>0$ for $r>1$, $Q_0'(1)>0$ and $Q_0(r)$ is of order $1/r$ as $r\to\infty$, we see that there exists a constant $\alpha_0>0$ such that $Q_0\geq \alpha_0 Y_0$. Also, since $\lim_{r\to\infty} \partial_r(rQ_0)=0$, and $\partial_r^2(rQ_0)=-rQ_0^{2m+1}\sim -\frac{c}{r^{2m}}$ for some constant $c>0$ as $r\to\infty$, we have
 $$\forall r\geq 1, \quad \partial_r(rQ_0)>0, \quad \partial_r(rQ_0)\sim \frac{c}{(2m-1)r^{2m-1}}, \; r\to\infty.$$
 As a consequence, using the exponential decay of $Y_0$, and taking $\alpha_0>0$ smaller if necessary, we obtain
 $$\forall r\geq 1, \quad \partial_r(rQ_0)\geq \alpha_0\left|\partial_r(rY_0)+e_0(rY_0)\right|.$$
 Thus
 $$ -\alpha_0\leq \alpha\leq 0\Longrightarrow u^{\alpha}(0,r)\geq 0,\quad \partial_r(ru^{\alpha}(0,r))+r\partial_t u^{\alpha}(0,r)\geq 0.$$
 By Corollary \ref{CR:positive}, if $-\alpha_0\leq \alpha \leq 0$, $u^{\alpha}(t)\geq 0$ for all $t\geq 0$ in the domain of definition of $u^{\alpha}$.  Combining with \eqref{bound_above}, we see that $u^{\alpha}$ cannot blow up in finite positive time or scatter to $-Q_0$ as $t\to\infty$. Thus by \eqref{alternative}, $u^{\alpha}$ scatters to $0$ in the future.
 \end{proof}
\begin{lemma}
\label{L:scatt}
 Fix a small $\alpha_0>0$. Let for $\delta>0$, 
 $$\widetilde{B}_{\delta}=\left\{(u_0,u_1)\in \HHH,\; \left\|S_L(t)\left((u_0,u_1)-(Q_0,0)+\alpha_0\YYY_0^+\right)\right\|_{\mathsf{S}(0,\infty)}<\delta\right\},$$
 where $S_L(t)(v_0,v_1)=\cos(t\sqrt{-\Delta_D})v_0+\frac{\sin t\sqrt{-\Delta_D}}{\sqrt{-\Delta_D}}v_1$. Then if $\delta$ is small enough, any solution of \eqref{eq:NLW} with initial data in $\widetilde{B}_{\delta}$ scatters to $0$.
\end{lemma}
\begin{proof}
 This follows immediately from Lemma \ref{L:Y0} and the standard perturbation theory for equation \eqref{eq:NLW}.
\end{proof}
\begin{proof}[Proof of Theorem \ref{T:M0}]
We fix a small $\delta>0$, and define $\mathcal{N}_{\delta,0}$ as in Corollary \ref{CR:local-manifold} (for $k=0$).
It follows from the proof of Theorem \ref{T:cs} that if $(u_0,u_1)\in \MMM_0^+$, then for large $T$, $\vec{u}(T)\in \NNN_{\delta,0}$ and that there is a neighborhood $\omega$ of $\vec{u}(T)$ such that $\MMM_0^+\cap \omega=\NNN_{\delta,0}\cap \omega$.

 We first construct the set $\UUU_S$ (in all the proof, we will omit the superscript $+$ to lighten notations) as follows. We let, fixing a small constant $\delta>0$
 \begin{equation*}
\widetilde{\UUU}_S=\bigg\{(u_0,u_1)+\alpha \YYY_0^+;\; (u_0,u_1)\in \NNN_{\delta,0},\; \alpha \in (-\delta,0)\bigg\}.
 \end{equation*}
 This is obviously an open subset of $\HHH$. Furthermore, taking $\delta$ smaller if necessary, we see by Remark \ref{R:tangent_space} that $\widetilde{\UUU}_S\cap \MMM_0^+=\emptyset$. We define
 $$\UUU_S=\left\{(u_0,u_1)\in \HHH\text{ s.t. } \exists t\in [0,T_+(u)),\; \vec{u}(t)\in \widetilde{\UUU}_S\right\},$$
 where $u$ is the solution of \eqref{eq:NLW} with initial data $(u_0,u_1)$. By the well-posedness theory for \eqref{eq:NLW} and since $\widetilde{\UUU}_S$ is open in $\HHH$, $\UUU_S$ is also open in $\HHH$. 
 
 We next prove that $\MMM_0^+$ is in the closure of $\UUU_S$ in $\HHH$. Let $(u_0,u_1)\in \MMM_0^+$, and $u$ be the corresponding solution. Then by Proposition \ref{P:CVdispersive_norm}, Corollary \ref{CR:CVdispersive_norm} and
 the definition of $\NNN_{\delta,0}$, we see that for large $T>0$, $\vec{u}(T)$ satisfies
 $$ \left\|S_Q(t)(\vec{u}(T)-(Q_0,0))\right\|_{\mathsf{S}(0,\infty)}<\delta,\quad \vec{u}(T)\in \NNN_{\delta,0}.$$
 For $n\geq 0$, denote by $(w_{0,n},w_{1,n})=\vec{u}(T)-\frac{1}{2^n}\YYY_0^+$. Clearly, $(w_{0,n},w_{1,n})\in \widetilde{\UUU}_S$ for large $n$. Also, $(w_{0,n},w_{1,n})$ converges to $\vec{u}(T)$ in $\HHH$ as $n\to\infty$. 
 Denote by $u_n$ the solution of \eqref{eq:NLW} such that $\vec{u}_n(T)=(w_{0,n},w_{1,n})$.  
 The well-posedness theory for \eqref{eq:NLW} implies that for large $n$, $t=0$ is in the domain of existence of $\vec{u}_n$, and 
 $$ \lim_{n\to\infty}\left\|\vec{u}_n(0)-(u_0,u_1)\right\|_{\HHH}=0.$$
 Since $\vec{u}_n(0)\in \UUU_S$, this proves that the closure of $\UUU_S$ contains $\MMM_0^+$.
 
 We next prove that any solution with initial data in $\UUU_S$ scatters to $0$ in the future. It is of course sufficient to prove this property for solutions with initial data in $\widetilde{\UUU}_S$. Let $(U_0,U_1)\in \widetilde{\UUU}_S$, and $U$ be the corresponding solution. 
 Then $(U_0,U_1)=(u_0,u_1)+\alpha\YYY_0^+$, where $(u_0,u_1)\in \NNN_{\delta,0}$, and $-\delta<\alpha<0$. Let $u$ be the solution of \eqref{eq:NLW} with initial data $(u_0,u_1)$ and $v$ be the solution of \eqref{eq:NLW} with initial data $(u_0,u_1)-\alpha_0\YYY_0^+$, where $\alpha_0$ is given by Lemma \ref{L:scatt}. By the definition of $\NNN_{\delta,0}$, we have
 $$\|u-Q_0\|_{\Ssf([0,\infty))}<M\delta.$$
 Using that
 $$|\partial_t^2(u-Q_0)-\Delta (u-Q_0)|\lesssim |u-Q_0|\big(|u-Q_0|^{2m}+Q_0^{2m}\big),$$
 we obtain, by
 $$\|S_L(t)((u_0,u_1)-(Q_0,0))\|_{\mathsf{S}([0,\infty))}\leq C\delta.$$

According to Lemma \ref{L:scatt} (taking a smaller $\delta$ if necessary), $v$ scatters to $0$. By Lemma \ref{CR:Y0}, taking $\delta>0$ small enough to insure $|\alpha|<\alpha_0$,
 $$\forall t\in [0,T_+(U)), \quad v(t)\leq U(t)\leq u(t).$$
 Since $u$ and $v$ are global, this implies that $U$ is global, and that for all $r>1$,
 $$0\leq \lim_{t\to\infty} U(t,r)\leq Q_0(r).$$
 Thus $U$ cannot scatter to $-Q_0$ or to an excited state $Q_k$ or $-Q_k$ ($k\geq 1$), since all these stationary solutions take negative values. We have already seen that $\widetilde{U}_S\cap \MMM_0^+=\emptyset$, hence $U$ does not scatter to $Q_0$. Thus the only possibility is that $U$ scatters to $0$.
 
The construction of $\UUU_F$ is similar. We define, for a small $\delta>0$:
 \begin{equation*}
\widetilde{\UUU}_F=\bigg\{(u_0,u_1)+\alpha \YYY_0^+,\; (u_0,u_1)\in \NNN_{\delta,0},\; \alpha \in (0,\delta)\bigg\},  
 \end{equation*}
and
 $$\UUU_F=\left\{(v_0,v_1)\in \HHH\text{ s.t. } \exists t\in [0,T_+(v)),\; \vec{v}(t)\in \widetilde{\UUU}_S\right\},$$
 where $v$ is the solution of \eqref{eq:NLW} with initial data $(v_0,v_1)$. By Remark \ref{R:tangent_space}, $\widetilde{\UUU}_F\cap \MMM_0^+=\emptyset$, and thus $\UUU_F\cap \MMM_0^+=\emptyset$. Using  Corollary \ref{CR:Y0}, we see that if $U$ is a solution with initial data $(u_0,u_1)+\alpha \YYY_0^+\in \widetilde{\UUU}_F$, then for all $t\geq 0$ in the domain of definition of $U$, $U(t)\geq u(t)$, where $u$ is the solution with initial data $(u_0,u_1)$, which scatters to $Q_0$. We prove that $U$ blows up in finite time by contradiction: if $T_+(U)=+\infty$,
 \begin{equation}
 \label{asymp_Q0}
 \liminf_{t\to\infty} U(t)\geq Q_0.  
 \end{equation} 
 By Theorem \ref{T:DY19}, there exists $S\in \mathcal{Q}$ such that $U$ scatters to $S$. By \eqref{asymp_Q0}, $S\geq Q_0$. By  Lemma \ref{L:InegStat}, we must have
 $S=Q_0$, which is excluded.
 Thus $U$ blows up in finite time (with a positive blow-up since $U\geq Q_0$.

 With a proof that is analog than the one above, we can prove that $\MMM_0^+$ is in the closure of $\UUU_F$.

 Finally, letting $\UUU=\UUU_S\cup \MMM_0^+\cup \UUU_F$, we see that it satisfies the conclusion of Theorem \ref{T:M0}. Indeed, to see that $\UUU$ is an open set, observe that
 $$\UUU=\left\{(v_0,v_1)\in \HHH\text{ s.t. } \exists t\in [0,T_+(v)),\; \vec{v}(t)\in \widetilde{\UUU}\right\},$$
 where
 \begin{equation*}
\widetilde{\UUU}=\bigg\{(u_0,u_1)+\alpha \YYY_0^+,\; (u_0,u_1)\in \NNN_{\delta,0},\; -\delta<\alpha<\delta\bigg\}.
 \end{equation*}

\end{proof}

\subsection{Instability of excited states by blow-up }
In this subsection we prove Theorems \ref{T:blow-up} and \ref{T:blow-up'}. We fix $k\geq 1$, and use the notations of Sections 3. In particular, we denote $\YYY_j^{\pm}$, $j=0,\ldots,k$ the eigenfunction of the linearized operator $\LLL_{Q_k}$, and $\{\pm e_0,\ldots,\pm e_k\}$ the corresponding eigenfunctions. As before, we omit to indicate the dependence in $k$ of $(\YYY_j^{\pm})_j$ and $(e_j)_j$ to lighten notations. 
\begin{proof}[Proof of Theorem \ref{T:blow-up}]
 We recall from Subsection \ref{sect:proof-contrac} the definition of $\mathscr{B}_{\delta}$ and from Corollary \ref{CR:local-manifold} the definition of the local center-stable manifold $\NNN_{\delta,k}$. 

 We fix $(U_0,U_1)\in \MMM_{k}^+$, $\delta>0$ small. We let $U$ be the corresponding solution, and $t_0$ such that $\vvec{U}(t_0)\in  \NNN_{\delta,k}$ (such a $t_0$ exists by Proposition \ref{P:CVdispersive_norm}.
  Using the smoothness of the flow (see Proposition \ref{P:smooth_flow} in the appendix), it is sufficient to construct a neighborhood $\UUU$ of $\vvec{U}(t_0)$ that satisfies the conclusion of Theorem \ref{T:blow-up}. 

 We define, for a small $\delta>0$, 
 $$\FFF_{\delta}=\left\{(Q_k,0)+\vvec{v_0} +\vartheta\left(\vvec{v_0}\right)+\eps \YYY_0^+;\; \vvec{v_0}\in \mathscr{B}_{\delta},\; \eps\in (-\delta,\delta)\right\},$$
 where $\vartheta$ is defined in the proof of Corollary \ref{CR:local-manifold}. Of course $\FFF_{\delta}$ and $\mathscr{B}_{\delta}$ depend on $k$. As in Section \ref{sec:center-manifold} we do not indicate this dependence to lighten notation. 
 
 Since $\YYY_0^+\in \HHH^{u}$, we have that for small $\delta$, $\FFF_{\delta}$ 
 is a submanifold of $\HHH$ of codimension $k$ that contains the local center stable manifold $\NNN_{\delta,k}$.  Its tangent space at $(Q_k,0)$ is $\HHH^{cs}\oplus \mathrm{span}\{\YYY_0^+\}$. We let $(u_0,u_1) \in  \FFF_{\delta}$, $u$ be the corresponding solution, and assume that $(u_0,u_1)\notin \MMM_k^+$.
 We will prove by contradiction that $u$ blows up in finite time. We assume that $T_+(u)=+\infty$. Then 
 $$(u_0,u_1)\in \bigcup_{\substack{j\geq 0}}\MMM_j^{+}\cup \left( -\MMM_j^+ \right) \cup \mathcal{S}^+.$$
 By Proposition \ref{P:P40}, 
 $$(u_0,u_1)\in \bigcup_{\substack{0\leq j\leq k-1}}\MMM_j^{+}\cup \left( -\MMM_j^+ \right)\cup \mathcal{S}^+.$$
 As a consequence,
 \begin{equation}
  \label{limit}
 \exists j\in \{0,\ldots,k-1\},\; \exists \iota\in \{\pm 1\}, \; \forall r>1,\quad \lim_{t\to\infty}u(t,r)=\iota Q_j(r)
 \end{equation} 
 or 
 \begin{equation} 
  \label{limit'}
\forall r>1,\quad \lim_{t\to\infty}u(t,r)=0.
 \end{equation} 
 We have $(u_0,u_1)=(w_0,w_1)+\eps \YYY_0^+$, where $(w_0,w_1)=(Q_k,0)+\vvec{v_0}+\vartheta(\vvec{v_0})$ is in $\MMM_k^+$, and $-\delta<\eps<\delta$. If $\eps>0$, we have by Corollary \ref{CR:Y0} that $u(t)\geq w(t)$ for all $t$ (where $w(t)$ is the solution of \eqref{eq:NLW} with initial data $(w_0,w_1)$). Since $(w_0,w_1)\in \MMM_k^+$, we deduce 
 $$\forall r>1,\quad \lim_{t\to\infty} u(t,r)\geq Q_k(r).$$
 This contradicts \eqref{limit'}.
 Moreover \eqref{limit} would imply $\iota Q_j\geq Q_k$, contradicting Lemma \ref{L:InegStat}. This leads to a condtradiction in the case $\eps>0$. The case $\eps<0$ can be treated the same way, using that in this case, $u(t)\leq w(t)$ for all $t$.
 \end{proof}
\begin{proof}[Sketch of proof of Theorem \ref{T:blow-up'}]
 The proof is almost the same to the proof of Theorem \ref{T:blow-up}, replacing $\FFF_{\delta}$ by one of the two connected open sets:
 $$\left\{(Q_k,0)+\vvec{v_0}+\vartheta(\vvec{v_0})+\sum_{j=0}^J \omega_j \YYY_j^+;\; \vvec{v_0}\in \mathscr{B}_{\delta},\; 
 \forall j\in \{1,\ldots,J\},\; 0<C\omega_j<\omega_0<\eps
 \right\}$$
 or
 $$\left\{(Q_k,0)+\vvec{v_0}+\vartheta(\vvec{v_0})+\sum_{j=0}^J \omega_j \YYY_j^+;\; \vvec{v_0}\in \mathscr{B}_{\delta},\; 
 \forall j\in \{1,\ldots,J\},\; -\eps<\omega_0<C\omega_j<0
 \right\}.$$
 The fact that these two sets are open are straightforward. Indeed, if $\vvec{u_0}\in \HHH$, close to $\NNN_{\delta,k}$, the expansion
 $$ \vvec{u_0}=(Q_k,0)+\vvec{v_0}+\vartheta(\vvec{v_0})+\sum_{j=0}^{J} \omega_j\YYY_j^+,$$
 with $\vvec{v_0}\in \HHH^{cs}$ is unique, and the maps $\vvec{u_0}\mapsto \omega_j$ are continuous for the topology of $\HHH$. This follows from the fact that we have necessarily $\vvec{v_0}=({\rm Id}-\PPP^+)\left(\vvec{u_0}-(Q_k,0)\right)$. The openness of the two sets above follow immediately.
\end{proof}

\subsection{Behaviour for negative times}
In this subsection, we give a sketch of proof of Theorem \ref{T:negative_time}. We fix $k\geq 0$ and use the same notations than in the preceding subsection.

\subsubsection{Solution in $\MMM_k^+$ blowing-up in the past}
Let $\delta> 0$ be a small parameter. We let 
$$\vvec{v_0}=\delta \YYY_0^-+\delta^{4/3}\YYY_k^-,$$
so that $\vvec{v_0}\in \HHH^s$ and $\vertiii{\vvec{v_0}}\approx \delta$ (where the norm $\vertiii{\cdot}$ is defined in \eqref{def_vertiii}). By Proposition \ref{P:good}, there exists a solution $u=Q_k+h$ of \eqref{eq:NLW}, with $\vec{u}(0)\in \MMM_k^+$, such that
$$ (\Id-\PPP^+)\left( \vec{h}(0) \right)=\vvec{v_0},\quad \left\|\PPP^+\vec{h}(0)\right\|_{\HHH}\lesssim \delta^2.$$
By Corollary \ref{CR:local-manifold} and its proof (applied to $t\mapsto u(-t)$), we see that $\vec{u}(0)\notin \MMM_k^-$.
We claim that if $\delta$ is small enough, denoting $\vec{h}(0)=(h_0,h_1)$, 
\begin{equation}
 \label{positivity_h}
 \forall r>1,\quad 
 h_0(r)\geq 0,\quad \partial_r(rh_0)-rh_1\geq 0.
 \end{equation} 
Indeed, we first note, denoting by $(v_0,v_1)=\vvec{v_0}=\delta \YYY_0^-+\delta^{4/3}\YYY_k^-$ that 
$$ v_0=\delta Y_0+\delta^{4/3}Y_k\geq \frac{\delta^{4/3}}{Cr} e^{-e_k r}.$$
(this is true for $r>R$, $R$ large since $rY_k(r)\geq  \frac{1}{C} e^{-e_k r}$, and for $1<r\leq R$ using that $Y_0$ is strictly positive on $(1,R)$, that $Y'_0(1)>0$, and that $\delta$ is small).

Since $\|v_0-h_0\|_{\dot{H}^1}\lesssim \delta^2$, and $v_0-h_0$ is a linear combination of the eigenfunctions $Y_j$, $j=0,\ldots,k$, we also obtain that
$$|(v_0-h_0)(r)|\lesssim \frac{1}{r}\delta^2 e^{-e_k r},\quad r>1,$$
which proves the first inequality in \eqref{positivity_h}. The proof of the second inequality is similar, using that 
$ v_1=-\delta e_0 Y_0-\delta^{4/3}e_kY_k$, together with \eqref{Sufficient_to_check} and \eqref{positivity_RYj}.

By \eqref{positivity_h} and Corollary \ref{CR:positive} applied to $t \mapsto u(-t)$, we obtain 
$$ \forall t\in (T_-(u),0],\quad u(t)\geq Q_k.$$
Since $\vec{u}(0)\notin \MMM_k^-$, we deduce that $T_-(u)$ is finite by the same argument as in the proof of Theorem \ref{T:blow-up}. Note the blow-up is positive.

If $k\geq 1$, the same argument with $\vec{v}_0=-\delta\YYY_0^--\delta^{4/3}\YYY_k^-$ yields a solution with negative blow-up, with initial data arbitrarily close to $Q_k$.
\subsubsection{Solutions blowing-up in both time directions}

We fix again $k\geq 0$ and consider the solutions $u^{\alpha}$, $v^{\alpha}$ with initial data
respectively 
$$(u_0^{\alpha},u_1^{\alpha})=(Q_k,0)+\alpha (Q_0,0)$$
and
$$(v_0^{\alpha},v_1^{\alpha})=(Q_k,0)+\alpha (0,Q_0).$$
We have $Q_0(r)>0$ for all $r>1$. Also $\partial_r^2(rQ_0)=-rQ_0^{2m+1}<0$, which proves, since $\lim_{r\to\infty}\partial_r(rQ_0(r))=0$, that 
$\partial_r(rQ_0(r))>0$ for all $r>1$.  By Corollary \ref{CR:positive}, we obtain that for $\alpha >0$
$$ \forall t\geq 0, \quad u^{\alpha}(t)\geq Q_k,\quad v^{\alpha}(t)\geq Q_k$$
and 
$$\forall t\leq  0, \quad u^{\alpha}(t)\geq Q_k,\quad v^{\alpha}(t)\leq Q_k.$$
Since $Q_0>0$, we see that $\int Q_0 Y_0\neq 0$. This implies $\{(Q_0,0),(0,Q_0)\}\cap \left(\HHH^{cs}\cup \HHH^{cu}\right)=\emptyset$ (where $\HHH^{cu}=\HHH^c\oplus \HHH^{u}$). Thus $(u_0^{\alpha},u_1^{\alpha})\notin \MMM_k^+\cup \MMM_k^-$ for small $\alpha\neq  0$. Using Proposition \ref{P:positive} and the same reasoning as in the proof of Theorem \ref{T:blow-up}, we obtain that $u^{\alpha}$ and, for $k\geq 1$, $v^{\alpha}$ blow up in both times directions, and that the blow-up is positive in both time directions for $u^{\alpha}$, and (if $k\geq 1$), positive in the future and negative in the past for $v^{\alpha}$. 

If $k\geq 1$ and $\alpha<0$, we obtain that $u^{\alpha}$ blows up with a negative blow-up in both time directions, and that $v^{\alpha}$ blows up with a negative blow-up in the future and a positive blow-up in the past.
\appendix


\section{Smoothness of the flow}
In this appendix, we prove
\begin{proposition}
\label{P:smooth_flow}
Let $(\tu_0,\tu_1)\in \HHH$, $T>0$ and assume that $T$ is in the maximal interval of existence of the solution $\tu$ of \eqref{eq:NLW} with initial data $(\tu_0,\tu_1)$. Then the map $(u_0,u_1)\to \vec{u}$ (where $u$ is the solution of \eqref{eq:NLW} with initial data $(u_0,u_1)$) is $C^{\infty}$ from a neighborhood $\mathcal{U}$ of $(\tu_0,\tu_1)$ in $\HHH$ to $C^{0}([0,T],\HHH)$. In particular the map $(u_0,u_1)\mapsto \vec{u}(T)$ is $C^{\infty}$ from $\mathcal{U}$ to $\HHH$. Moreover the differential of this map is uniformly bounded on $\UUU$.
\end{proposition}
We will use the following result concerning contraction mapping with parameters and which is also needed in the proof of the smoothness of the center-stable manifold. 
\begin{lemma}
\label{L:smooth}
 Let $X,Y$ be Banach spaces, $U$ an open subset of $X$, $V_o$ an open subset of $Y$ and $V$ a closed subset of $V_o$. Let $\Phi: U\times V_o\to V$ be a $C^{\infty}$ map, which is Lipschitz on $U\times V$ and an uniform contraction with respect to the second variable:
 \begin{multline}
 \label{Lipschitz}
\exists C>0,\; \exists \lambda\in (0,1),\quad \forall (x_1,x_2)\in U,\; \forall (y_1,y_2)\in V,\\ 
\left\| \Phi(x_1,y_1)-\Phi(x_2,y_2)\right\|
\leq C\|x_1-x_2\|+\lambda \|y_1-y_2\|.  
 \end{multline}
For $x\in U$, we define $\varphi(x)$ as the unique element of $V$ such that $\Phi(x,\varphi(x))=\varphi(x)$. Then 
$\varphi$ is $C^{\infty}$ on $U$. Furthermore, there exists a constant $C_1>0$ (depending only on the constants $C$ and $\lambda$ in \eqref{Lipschitz}), such that 
\begin{equation}
\label{bound_phi'}
\forall x\in U,\quad \|\varphi'(x) \|\leq C_1.
\end{equation} 
\end{lemma}
\begin{proof}[Sketch of proof of the Lemma]
 This is classical. Let $a\in U$. Using \eqref{Lipschitz}, one easily proves that $\varphi$ is Lipschitz. The expansion
 \begin{multline*}
 \varphi(x)-\varphi(a)=\Phi(x,\varphi(x))-\Phi(a,\varphi(a))\\
 =D_1\Phi(a,\varphi(a))(x-a)+D_2\Phi(a,\varphi(a))(\varphi(x)-\varphi(a))+\underbrace{\mathcal{O}\left(\|x-a\|^2+\|\varphi(x)-\varphi(a)\|^2\right)}_{\mathcal{O}\left(\|x-a\|^2\right)},
 \end{multline*}
 where $x\in U$, implies that $\varphi$ is differentiable at $a$, with derivative
 \begin{equation}
 \label{phi'a}
  \varphi'(a)=\left(\Id_Y-D_2\Phi(a,\varphi(a))\right)^{-1}\circ D_1\Phi(a,\varphi(a)).
 \end{equation} 
Note that this is well defined, since by the assumption \eqref{Lipschitz}, the operator norm of $D_2\Phi(a,\varphi(a))$ is at most $\lambda$.
Using \eqref{phi'a}, we see that $\varphi$ is indeed $C^1$. Using again \eqref{phi'a}, we obtain that $a\mapsto \varphi'(a)$ is $C^1$, i.e. that $\varphi$ is $C^2$. By an easy induction we deduce that $\varphi$ is $C^{\infty}$. The bound \eqref{bound_phi'} follows from \eqref{phi'a} and the fact that $\Phi$ is Lipschitz.
\end{proof}
\begin{proof}[Proof of Proposition \ref{P:smooth_flow}]
We can assume that $T$ is small. An iteration argument together with the well-posedness theory for \eqref{eq:NLW} will give the general case.
We can write Duhamel's formula as 
$$\vec{u}=\Phi\Big((u_0,u_1),\vec{u}\Big),$$
where
\begin{equation*}
\Phi\left((u_0,u_1),\vec{u}\right)(t)
 =\begin{pmatrix}
S_L(t)(u_0,u_1)+\int_0^t S_L(t-s)\left(0, u(s)^{2m+1}\right)\\
\partial_t S_L(t)(u_0,u_1)+\int_0^t S_L(t-s)\left(u(s)^{2m+1},0\right)
\end{pmatrix}
\end{equation*}
and $S_L$ is the flow of the linear wave equation outside an obstacle with Dirichlet boundary conditions. Using the standard properties of $S_L$ and the fact that $u\mapsto u^{2m+1}$ is a $C^{\infty}$ function on $\dot{H}^1_0$, one obtain that $\Phi$ satisfies the assumptions of Lemma \ref{L:smooth} with $X=\HHH$, $Y=C^0([0,T],\HHH)$ for small $T$. This concludes the proof.
\end{proof}
\section{Finite codimensional submanifold}
\label{A:submanifold}
In this appendix, we gather a few facts about submanifolds in infinite dimensions that are needed in the article. All these facts are very classical (see e.g. \cite{Lang95Bo}), however we have not found in the litterature the exact statements that are needed. In all this appendix, we fix a Banach space $X$ and an integer $p\geq 1$.
\begin{definition}
\label{D:submanifold}
 Let $M\subset X$. We say that $M$ is a \emph{smooth submanifold} of $X$ of codimension $p$ if for all $\mathbf{u}\in M$, there exist two Banach spaces $\mathcal{Y}$ and $\mathcal{Z}$, with $\dim \ZZZ=p$, an open neighborhood $U$ of $\mathbf{u}$ in $X$, a smooth ($C^{\infty}$) diffeomorphism $\varphi$ from $U$ to $\VVV\times \WWW$, where $\VVV$ and $\WWW\ni 0$ are open subsets of respectively $\YYY$ and $\ZZZ$ such that $\varphi(M\cap U)=\VVV\times \{0\}$. The \emph{tangent space} $T_{\mathbf{u}}M$ of $M$ at $\mathbf{u}$ is
 $\varphi'(\mathbf{u})^{-1}(\YYY\times \{0\})$.
\end{definition}
In the sequel, all submanifolds and diffeomorphisms will be implicitely smooth.

It follows immediately from the definition that if $M$ is a submanifold of codimension $p$ of $X$ and $V$ is an open subset of $X$, then $V\cap M$ (if not empty) is a submanifold of codimension $p$ of $X$. If furthermore $\psi$ is a diffeomorphism from $V$ to an open subset $\psi(V)$ of $X$, then $\psi(V\cap M)$ is a submanifold of $X$ of codimension $p$.

Obviously, $T_{\mathbf{u}}M$ is a subspace of codimension $p$ of $X$. It is independent of the choice of the diffeomorphism $\varphi$. Indeed, it is exactly the vector space of all $\{\psi'(0)\}$, where $\psi$ goes over all $C^{\infty}$ functions $\R\mapsto X$ such that $\psi(0)=\mathbf{u}$ and $\psi(\R)\subset M$.

One can define a submanifold as the graph of a function:
\begin{lemma}
 \label{L:graph}
 Let $M$ be a subset of $X$. Then $M$ is a submanifold of codimension $p$ of $X$ if and only if for all $\mathbf{u}\in M$,  there exists an open neighborhood $U$ of $\mathbf{u}$ in $X$, closed subspaces $Y$ and $Z$ of $X$ such that $X=Y \oplus Z$ with $\dim Z=p$, an open neighborhood $V$ of $0$ in $Y$ and a $C^{\infty}$ map $\Theta:V\to Z$ such that $\Theta(0)=0$, $\Theta'(0)=0$ and
\begin{equation}
\label{graph}
U\cap M=\{\mathbf{u}+\mathbf{h}+\Theta(\mathbf{h}),\; \mathbf{h}\in V\}.
\end{equation}
The tangent space of $M$ at $\mathbf{u}$ is $Y$. Furthermore, in the preceding statement, one can choose any subspace $Z$ of $X$ such that $X=Y\oplus Z$.
\end{lemma}
\begin{proof}
We will just prove that if for all $\ubf\in M$, there exists $U$, $Y$, $Z$, $V$ and $\Theta$ as in the statement of the Lemma, then $M$ is a submanifold of $X$ of codimension $p$. The proof of the converse statement, relying on the implicit mapping theorem, is left to the reader since it is not needed in the paper.

We fix $\ubf$, and construct a diffeomorphism $\varphi$ as in Definition \ref{D:submanifold}.

Denote by $\Pi_1\in \mathcal{L}(X,Y)$ the projection along $Z$ onto $Y$ and $\Pi_2=1-\Pi_1\in \mathcal{L}(X,Z)$ the projection along $Y$ onto $Z$. The map $\varphi$ defined on $\{\mathbf{u}-y-z,\; y\in V, \;z\in Z\}$ by
$$\varphi(\mathbf{v})=\Big(\Pi_1(\mathbf{v}-\mathbf{u}),\Pi_2(\mathbf{v}-\mathbf{u})-\Theta(\Pi_1(\mathbf{v}-\mathbf{u})\Big))$$
is a local diffeormorphism from a neighborhood $U'$ of $\mathbf{u}$ to an open neighborhood $\varphi(U')$ of $0$ in $Y\times Z$.  Of course (taking a smaller $U'$) it is always possible to take $\varphi(U')$ of the form $U_1\times U_2$, where $U_1$ and $U_2$ are open neighborhood of $0$ in $Y$ and $Z$ respectively. The fact that
$$ \varphi(U'\cap \MMM)=U_1\times \{0\}$$
follows from \eqref{graph}. The assertion on the tangent space follows from the formula $\varphi'(\mathbf{u})=(\Pi_1,\Pi_2)$.

\end{proof}
\begin{corollary}
 \label{Cor:graph}
 Assume $X=Y\oplus Z$, where $Y$, $Z$ are closed subspaces of $X$ with $\dim Z=p$. Let $V$ be an open neighborhood of $0$ in $Y$. Let $\Theta:V \to Z$ be a $C^{\infty}$ function such that $\Theta(0)=0$, $\Theta'(0)=0$. Let $\ubf\in X$. Let 
 $$M=\left\{\ubf+y+\Theta(y),\; y\in V\right\}.$$
 Then $M$ is a submanifold of $X$ of codimension $p$. 
\end{corollary}
\begin{proof}[Sketch of proof]
Let $\tubf\in M$. Then $\tubf=\ubf+\ty+\Theta(\ty)$ for some $\ty\in V$. 

Consider $\widetilde{Y}=\{h+\Theta'(\ty)h,\; h\in Y\}$. This is a subspace of $X$ of codimension $p$. Indeed, one has $X=\widetilde{Y}\oplus Z$. Also, the map $L: h\mapsto h+\Theta'(\ty)h$ is a $C^{\infty}$ isomorphism from $Y$ to $\tY$. One proves that
\begin{equation*}
M=\left\{\tubf+h+\Theta'(\ty)h+\rho(h),\; h\in \widetilde{V}\right\}
=\left\{\tubf+k+\rho( L^{-1}(k)),\; k\in L(\widetilde{V})\right\},
\end{equation*}
where $\widetilde{V}=\{y-\ty,\; y\in V\}$, $\rho(h)=\Theta(\tilde{y}+h)-\Theta(\tilde{y})-\Theta'(\ty)h$. 
This proves that $M$ satisfies the criterion of Lemma \ref{L:graph}, and thus that $M$ is a submanifold of codimension $p$ of $X$.
\end{proof}

%

 \begin{proposition}
 \label{P:transverse_manifold}
 Let $M_1$, $M_2$ be two submanifolds of $X$, of codimension respectively $p_1$ and $p_2$ in $X$. Let $\ubf \in M_1\cap M_2$ and $Y=T_{\ubf}M_1\cap T_{\ubf}M_2$. Assume that for $j\in \{1,2\}$, there exists a subspace $Z_j$ of $T_{\ubf}M_j$ such that $\dim Z_j=p_j$ and
$$X=Y\oplus Z_1\oplus Z_2, \quad T_{\ubf}M_1=Y\oplus Z_1, \quad T_{\ubf}M_2=Y \oplus Z_2.$$
Assume also that there exists a neighborhood $U$ of $0$ in $X$ such that, for $j\in \{1,2\}$, 
$$ U\cap M_j=\left\{\ubf+h+\Theta_j(h),\; h\in V_j\right\},$$
where $\Theta_j$ is a $C^{\infty}$ map from a neighborhood $V_j$ of $0$ in $T_{\ubf}M_j$ to $Z_{3-j}$ such that $\Theta_j(0)=0$, $\Theta_j'(0)=0$. Then there is a neighborhood $U'$ of $\ubf$ in $X$ such that $U'\cap M_1\cap M_2$ is a submanifold of $X$ of finite codimension $p_1+p_2$. The tangent space of this submanifold at $\ubf$ is $Y$.
\end{proposition}
\begin{remark}
 The intersection of $M_1$ and $M_2$ is said to be \emph{transverse} at $\ubf$ if and only if $T_{\ubf}M_1+T_{\ubf}M_2=X$. One can prove that if $M_1$ and $M_2$ are two finite codimensional submanifolds of $X$ with transverse intersection at $\ubf$, then the assumptions of Proposition \ref{P:transverse_manifold} are satisfied. In other words, Proposition \ref{P:transverse_manifold} says that the transverse intersection of two submanifolds of finite codimension $p_1$ and $p_2$ is locally a submanifold of codimension $p_1+p_2$.
 We omit the proof of this fact, which is not used in the article.
\end{remark}
\begin{proof}
Let $\vbf\in U$. We write $\ubf-\vbf=y+z_1+z_2$, where $y\in Y$, $z_1\in Z_1$, $z_2\in Z_2$. If we assume furthermore $\vbf\in M_1\cap M_2$, we have 
$$ y+z_1+z_2=h_1+\Theta_1(h_1)=h_2+\Theta_2(h_2),$$
for some $(h_1,h_2)\in V_1\times V_2$. Thus $z_2=\Theta_1(h_1)$, $z_1=\Theta_2(h_2)$, which yields
$$ z_2=\Theta_1(y+z_1)=\Theta_1(y+\Theta_2(y+z_2)).$$
If $Z$ is a Banach space and $r>0$, we denote by $B_Z(r)$ the ball $\{z\in Z,\;\|z\|<r\}$. For a small $\eps>0$, we consider the $C^{\infty}$ map $\Phi: B_Y(\eps)\times \overline{B_{Z_2}}(\eps)\to B_{Z_2}(2\eps)$ defined by $\Phi(y,z_2)=\Theta_1(y+\Theta_2(y+z_2))$. Using that for $j\in \{1,2\}$, we have $\Theta_j(0)=0$, $\Theta_j'(0)=0$, one can check that if $\eps>0$ is small enough, $\Phi$ satisfies the assumptions of Lemma \ref{L:smooth} (fixed point with parameter). As a consequence, there exists a map $\rho: B_Y(\eps)\to B_{Z_2}(2\eps)$ such that for all $y\in B_Y(\eps)$, $\rho(y)$ is the only $a\in \overline{B_{Z_2}(\eps)}$ such that $\Phi(y,a)=a$. Using the equality
\begin{equation}
 \label{defrho}
 \rho(y)=\Theta_1(y+\Theta_2(y+\rho(y)))
\end{equation} 
and the fact that $\Theta_j(0)=0$, $\Theta_j'(0)=0$, we see that $\rho(0)=0$, $\rho'(0)=0$.

Let $\tU$ be the open subset of $X$: $\tubf+B_{Y}(\eps)+B_{Z_1}(\eps)+B_{Z_2}(\eps)$. Going back to the computations above, we see that if $\vbf\in \tU\cap M_1\cap M_2$, then (with the same decomposition $\vbf=\ubf+y+z_1+z_2$ as above),
$$z_2=\rho(y),\quad z_1=\Theta_2(y+\rho(y)),$$
and thus
$$ \vbf=\ubf+y+\Theta(y),\; y\in B_Y(\eps),$$
where by definition
$$\Theta(y)=\rho(y)+\Theta_2(y+\rho(y)).$$
Note that $\Theta$ is a $C^{\infty}$ map from $B_{Y}(\eps)$ to $Z_1\oplus Z_2$ such that $\Theta(0)=0$, $\Theta'(0)=0$. We have proved
$$ \tU\cap M_1\cap M_2\subset \left\{\ubf+y+\Theta(y),\; y\in B_Y(\eps)\right\}.$$
To prove the opposite inclusion, we consider $y\in B_Y(\eps)$ and let $\vbf=\ubf+y+\Theta(y)$. By the definition of $\Theta$, we have
$$\vbf=\ubf+y+\rho(y)+\Theta_2(y+\rho(y))=\ubf+h_2+\Theta_2(h_2),$$
where $h_2=y+\rho(y)$ is an element of $V_2$ (if $\eps$ is chosen small enough). Thus $\vbf\in M_2$. Furthermore, using that $\rho(y)=\Phi(y,\rho(y))$, we obtain $\rho(y)=\Theta_1(y+\Theta_2(y+\rho(y)))$. Thus $\vbf -\ubf=h_1+\Theta_1(h_1)$, where $h_1=y+\Theta_2(y+\rho(y))$ is an element of $V_1$ (again, if $\eps$ is small enough). This shows that $\vbf \in M_1$. As a conclusion
$$ \tU\cap M_1\cap M_2= \left\{\ubf+y+\Theta(y),\; y\in B_Y(\eps)\right\}.$$
This shows by Corollary \ref{Cor:graph} that $M_1\cap M_2\cap \tU$ is a submanifold of $X$ of codimension $p_1+p_2$, with tangent space $T_{\ubf}M_1\cap T_{\ubf}M_2$ at $\ubf$.
\end{proof}

\end{document}